\documentclass[11pt]{amsart}
\newfont{\cyr}{wncyr10 scaled 1100}
\usepackage[top=2.8cm, bottom=2.8cm, left=2.4cm, right=2.4cm]{geometry}
\usepackage{amsthm,amssymb,amsmath,amsfonts,mathrsfs,amscd, graphics}
\usepackage{mathtools}
\usepackage[latin1]{inputenc}

\usepackage[all]{xy}
\usepackage{latexsym}
\usepackage{longtable}
\usepackage{color}
\usepackage{tikz}
\usepackage{tikz-cd}
\usepackage{lua-visual-debug}
\usepackage{dsfont}
\usetikzlibrary{matrix}
\usepackage{hyperref}
\usepackage{enumitem}
\usepackage{setspace}

\newcommand*\ZZ{|[draw,circle]| \Z_2}
\numberwithin{equation}{section}
\theoremstyle{plain}
\newtheorem{theorem}{Theorem}[section]
\newtheorem*{theorem*}{Theorem}
\newtheorem{corollary}[theorem]{Corollary}
\newtheorem{lemma}[theorem]{Lemma}
\newtheorem{proposition}[theorem]{Proposition}
\newtheorem{conjecture}[theorem]{Conjecture}
\numberwithin{equation}{section}
\newtheorem{thm}{Theorem} 
\newtheorem{ass}[thm]{Assumption}
\newtheorem{conj}[thm]{Conjecture}

\theoremstyle{definition}
\newtheorem{definition}[theorem]{Definition}

\newtheorem{assumption}[theorem]{Assumption}

\theoremstyle{remark}
\newtheorem{obswr}[theorem]{Observation}
\newtheorem{remarkwr}[theorem]{Remark}

\newtheorem{intro-definition}[theorem]{Definition}

\newenvironment{remark}{\begin{remarkwr}\begin{upshape}}{\end{upshape}\end{remarkwr}}
\newenvironment{myproof}[2] {\paragraph{\emph{Proof of {#1} {#2} }}}{\hfill$\square$}

\def\Gal{\mathrm{Gal}}
\def\GL{\mathrm{GL}}

\def\PGL{\mathrm{PGL}}

\def\det{\mathrm{det}}
\def\Lie{\mathrm{Lie}}

\def\loc{\mathrm{loc}}
\def\ord{\mathrm{ord}}
\def\Spec{\mathrm{Spec}}
\def\im{\mathrm{im}}
\def\coker{\mathrm{coker}}
\def\Ch{\mathrm{CH}}
\def\AJ{\mathrm{AJ}}

\def\sing{\mathrm{sing}}
\def\Frob{\mathrm{Frob}}
\def\new{\mathrm{new}}
\def\ur{\mathrm{ur}}
\def\ac{\mathrm{ac}}

\DeclareMathOperator{\End}{End}

\def\calH{\mathcal{H}}

\def\calM{\mathcal{M}}
\def\calO{\mathcal{O}}

\def\calX{\mathcal{X}}
\def\calY{\mathcal{Y}}
\def\calZ{\mathcal{Z}}

\def\CC{\mathbf{C}}

\def\FF{\mathbf{F}}

\def\PP{\mathbf{P}}
\def\QQ{\mathbf{Q}}

\def\TT{\mathbb{T}}

\def\ZZ{\mathbf{Z}}

\def\rmA{\mathrm{A}}
\def\rmB{\mathrm{B}}

\def\rmH{\mathrm{H}}

\def\rmE{\mathrm{E}}

\def\rmM{\mathrm{M}}
\def\rmN{\mathrm{N}}

\def\rmV{\mathrm{V}}

\def\rmS{\mathrm{S}}
\def\rmT{\mathrm{T}}

\def\frakm{\mathfrak{m}}

\def\triplef{\underline{\mathbf{f}}}
\def\Adel{\mathbf{A}}

\newcommand{\Iw}{\mathrm{Iw}}

\newcommand{\beqcd}[1]{\begin{equation*}\label{#1}\tag{#1}}
\newcommand{\eeqcd}{\end{equation*}}

\include{thebibliography}

\begin{document}

\title[Level raising and Diagonal cycles I]
{Arithmetic level raising on triple product of Shimura curves and Gross-Schoen Diagonal cycles I: Ramified case}

\author{Haining Wang }
\address{\parbox{\linewidth} {Haining Wang\\ Department of Mathematics,\\ McGill University,\\ 805 Sherbrooke St W,\\ Montreal, QC H3A 0B9, Canada.~ }}
\email{wanghaining1121@outlook.com}

\begin{abstract}
In this article we study the Gross-Schoen diagonal cycle on a triple product of Shimura curves at a place of bad reduction. We relate the image of the diagonal cycle under the Abel-Jacobi map to certain period integral that governs the central critical value of the Garrett-Rankin type triple product $L$-function via level raising congruences. As an application we prove certain rank $0$ case of the Bloch-Kato conjecture for the symmetric cube motive of a weight $2$ modular form. 
\end{abstract}
    
\subjclass[2000]{Primary 11G18, 11R34, 14G35}
\date{\today}

\maketitle
\tableofcontents

\section{Introduction}
Let $\triplef=(f_{1}, f_{2}, f_{3})$ be a triple of normalized newforms in  $S^{\new}_{2}(\Gamma_{0}(N))^{3}$ of level $N\geq 5$. We assume that $N=N^{+}N^{-}$ is a factorization of $N$ such that $(N^{+}, N^{-})=1$ and $N^{-}$ is square-free with odd number of prime factors. We fix a prime $l\geq 5$ which will be used as the residual characteristic of the coefficient rings throughout this article. For each $i=1, 2, 3$, let $\rmV_{i}=\rmV_{f_{i}, \lambda_{i}}$ be the $\lambda_{i}$-adic Galois representation attached $f_{i}$  given by Eichler-Shimura construction where $\lambda_{i}$ is a place over $l$ in the Hecke field $\QQ(f_{i})$ of $f_{i}$. Then we can attach a \emph{Garret-Rankin type triple product $L$-function} $$L(f_{1}\otimes f_{2}\otimes f_{3}, s)$$ to the triple product Galois representation $\rmV_{1}\otimes \rmV_{2}\otimes \rmV_{3}$. The parity of the order of vanishing of $L(f_{1}\otimes f_{2}\otimes f_{3}, s)$ at the central critical point $s=2$ is controlled by the 
the \emph{global root number} $\epsilon\in\{1, -1\}$ which can be factored into $\epsilon=\prod_{v\mid N\infty}\epsilon_{v}$ of local root numbers $\epsilon_{v}\in\{\pm1\}$. In our setting, $\epsilon_{\infty}=-1$ since the weights $(2, 2, 2)$ of $\triplef$ are \emph{balanced}. In this article, we will assume the following assumption.
\begin{equation*}\tag{$\epsilon=1$}
\text{\emph{The product of local root numbers $\prod_{v\mid N}\epsilon_{v}=-1$ and therefore $\epsilon=1$.}} 
\end{equation*}
One can also attach a more arithmetic object to the triple tensor product Galois representation $\rmV_{1}\otimes \rmV_{2}\otimes \rmV_{3}$ called the Bloch-Kato Selmer group. We remind its definition below.  Denote by $\mathrm{B_{cris}}$ the crystalline period ring with respect to $\QQ_{l}$. The \emph{triple product Bloch-Kato Selmer group} $$\rmH^{1}_{f}(\QQ, \rmV_{1}\otimes \rmV_{2}\otimes \rmV_{3}(-1))$$ is the subset of classes $s\in \rmH^{1}(\QQ, \rmV_{1}\otimes \rmV_{2}\otimes \rmV_{3}(-1))$ such that 
${\rm{loc}}_{l}(s)\in \rmH^{1}_{f}(\QQ_{l}, \rmV_{1}\otimes \rmV_{2}\otimes \rmV_{3}(-1))$
where $\mathrm{loc}_{l}$ is the localization map to the local Galois cohomology group at $l$ and 
$\rmH^{1}_{f}(\QQ_{l}, \rmV_{1}\otimes \rmV_{2}\otimes \rmV_{3}(-1))$
is given by the Bloch-Kato local condition
\begin{equation*}
\ker[ \rmH^{1}(\QQ_{l}, \rmV_{1}\otimes \rmV_{2}\otimes \rmV_{3}(-1))\rightarrow \rmH^{1}_{f}(\QQ_{l}, (\rmV_{1}\otimes \rmV_{2}\otimes \rmV_{3})\otimes\rmB_{\mathrm{cris}}(-1))].
\end{equation*}
The \emph{Bloch-Kato conjecture} predicts a relationship between the order of the vanishing of the triple product $L$-function $L(f_{1}\otimes f_{2}\otimes f_{3}, s)$ at $s=2$ and the rank of $\rmH^{1}_{f}(\QQ, \rmV_{1}\otimes \rmV_{2}\otimes \rmV_{3}(-1))$.  The present article will be concerned with the \emph{rank 0 case} of the Bloch-Kato conjecture for $\rmV_{1}\otimes \rmV_{2}\otimes \rmV_{3}$. More precisely, we are concerned with the following conjecture.
\begin{conj}
Suppose that the central critical value $L(f_{1}\otimes f_{2}\otimes f_{3}, 2)$ is non-zero, then we have
\begin{equation*}
\rmH^{1}_{f}(\QQ, \rmV_{1}\otimes \rmV_{2}\otimes \rmV_{3}(-1))=0. 
\end{equation*}
\end{conj}
One of the most successful approaches to prove such conjectures is via the \emph{Euler-Kolyvagin system} arguments. In practice, one need to impose various assumptions on the triple $\triplef=(f_{1}, f_{2}, f_{3})$ to apply the Euler system argument. We will formulate a more precise conjecture with a few additional assumptions in  the last section of this article, see Conjecture \ref{main-conj}. This article makes the first step towards the construction of such an Euler-Kolyvagin system for the triple product representation. More precisely, the main result of this article concerns an explicit reciprocity formula \'{a} la Bertonilli-Darmon for the representation $\rmV_{1}\otimes \rmV_{2}\otimes \rmV_{3}$. Such formula, loosely speaking, relates cycles classes on Shimura varieties to period integrals that govern the central critical values of automorphic $L$-functions via level raising congruences among automorphic forms. We refer the readers to the recent article \cite{LTXZZ} for the most general results along this line. In this article, we relate the \emph{Gross-Schoen diagonal cycle class} on the triple product of Shimura curves at a place of bad reduction to a period integral that represents the algebraic part of the triple product $L$-function. Note that our results are not subsumed by \cite{LTXZZ}. In the companion article \cite{Wang}, we will study the Gross-Schoen diagonal cycle class on the triple product of Shimura curves at a place of good reduction and so-called second reciprocity law is proved there. In the case when $\triplef=(f, f, f)$ for a single modular form $f$ of weight two, our constructions here and in \cite{Wang} produce the desired Euler-Kolyvagin system for the symmetric cube representation of $f$ and we can give some evidences to the Bloch-Kato conjecture for the symmetric cube representation of $f$ in the rank $0$ and rank $1$ case. 

\subsection{Main results}
In order to state our results precisely, we introduce more notations. For $i=1, 2, 3$, let $\QQ(f_{i})$ be the Hecke field of $f_{i}$. For simplicity, we assume that $E=\QQ(f_{1})=\QQ({f_{2}})=\QQ({f_{3}})$ in this introduction. Let $\lambda$ be a place of $E$ above $l$ and $E_{\lambda}$ be the completion of $E$ at $\lambda$.  Let $\calO=\calO_{E_{\lambda}}$ be the ring of integers of $E_{\lambda}$. We denote by $\varpi$ a uniformizer of $\calO$ and set $\calO_{n}=\calO/\varpi^{n}$ for any $n\geq 1$. Let $\phi_{i}: \TT\rightarrow \calO$ be the natural morphism form the $l$-adic Hecke algebra to $\calO$ corresponding to the Hecke eigensystem of $f_{i}$ and let $\phi_{i, n}: \TT\rightarrow \calO_{n}$ be the reduction of $\phi_{i}$ modulo $\varpi^{n}$. Here the $l$-adic Hecke algebra $\TT$ is the $l$-adic completion of the Hecke algebra that acts faithfully on the subspace of $S_{2}(\Gamma_{0}(N))$ that is new at primes dividing $N^{-}$. In particular, $\phi_{i}$ sends the Hecke operator $T_{p}$ to the $p$-th Fourier coefficient $a_{p}(f)$ of $f$ for $p\nmid N$. We denote by $I_{i, n}$ the kernel of $\phi_{i, n}$ and by $\frakm_{i}$ the maximal ideal in $\TT$ containing $I_{i, n}$. Let $\frakm_{\triplef}=(\frakm_{1}, \frakm_{2}, \frakm_{3})$. We will always assume that the maximal ideals $\frakm_{i}$ are \emph{residually irreducible} in the sense explained below. Let 
\begin{equation*}
\rho_{i}: G_{\QQ}\rightarrow \GL(\rmV_{i})
\end{equation*} 
be the Galois representation attached to $f_{i}$. Then we denote by $\bar{\rho}_{i}$ the residual Galois representation of $\rho_{i}$. We say $\frakm_{i}$ is residually irreducible if $\bar{\rho}_{i}$ is absolutely irreducible.  
We introduce the notion of \emph{n-admissible primes} for $\triplef=(f_{1}, f_{2}, f_{3})$ in Definition \ref{n-adm}. A prime $p$ is $n$-admissible for the triple $\triplef$ if 
\begin{enumerate}
\item $p\nmid Nl$;
\item $l\nmid p^{2}-1$;
\item $\varpi^{n}\mid p+1-\epsilon_{p, i}a_{p}(f_{i})$ with $i=1, 2, 3$ and $\epsilon_{i}\in \{\pm1\}$;
\item $\epsilon_{p, 1}\epsilon_{p, 2}\epsilon_{p, 3}=1$.
\end{enumerate}
This is an extension of the notion of $n$-admissible prime in \cite{BD-Main} to the triple product setting which can be loosely interpreted as those primes $p$ for which one can find a triple $\triplef^{[p]}=(f^{[p]}_{1}, f^{[p]}_{2}, f^{[p]}_{3})$ of weight $2$ newforms of level $pN$ that are congruent to $\triplef=(f_{1}, f_{2}, f_{3})$ modulo $\varpi^{n}$. We highlight the additional condition $(4)$. This condition is imposed to make a sign change for the triple product $L$-function attached to $\triplef^{[p]}=(f^{[p]}_{1}, f^{[p]}_{2}, f^{[p]}_{3})$ at $p$. Indeed the local sign at $p$ is given by $-\epsilon_{p,1}\epsilon_{p, 2}\epsilon_{p, 3}$ by \cite[1.3]{GK92} which is $-1$ by our assumption. This sign change makes it reasonable to consider a cycle class attached to the triple $\triplef^{[p]}$ in light of the Bloch-Kato conjecture of odd rank. For the triple $\triplef^{[p]}$ and $i=1, 2, 3$, we have morphisms $\phi^{[p]}_{i}: \TT^{[p]}\rightarrow \calO$ and $\phi^{[p]}_{i,n}: \TT^{[p]}\rightarrow \calO_{n}$ defined similarly as before with $\TT^{[p]}$ the $l$-adic Hecke algebra corresponding to the subspace of $S_{2}(\Gamma_{0}(pN))$ that is new at primes dividing $pN^{-}$ . Let $I^{[p]}_{i, n}$ be the kernel of  $\phi^{[p]}_{i,n}$ and let $\frakm^{[p]}_{i}$ be the maximal ideal containing $I^{[p]}_{i, n}$ in $\TT^{[p]}$. Let $\frakm^{[p]}_{\triplef}=(\frakm^{[p]}_{1}, \frakm^{[p]}_{2}, \frakm^{[p]}_{3})$. We fix a such $n$-admissible prime $p$ for $\triplef$. For the cycle class, it is natural to consider the diagonal cycle on the triple product of Shimura curves. The diagonal cycles on the triple product of curves are generally referred to as the \emph{Gross-Schoen diagonal cycles}. These cycles are introduced and studied in \cite{GS95}. Their connections to tripe product $L$-functions are given in \cite{GK92}, \cite{YZZ-dia}. To define the Shimura curves in our case, we need to introduce the following quaternion algebras. Let $B$ be the definite quaternion algebra over $\QQ$ of discriminant $N^{-}$ and $B^{\prime}$ be the indefinite quaternion algebra over $\QQ$ of discriminant $pN^{-}$. Then one can associate a Shimura set $X^{B}=X^{B}_{N^{+}, N^{-}}$ to $B$ and a Shimura curve $X=X^{B^{\prime}}_{N^{+}, pN^{-}}$ over $\QQ$ to $B^{\prime}$. We refer the reader to \S 2.1 for the constructions. In particular we have an integral model of $\mathfrak{X}$ of $X$ over $\ZZ_{(p)}$. Note that since $p$ is ramified in $B^{\prime}$, the completion of $\mathfrak{X}$ at its special fiber admits \emph{Cerednick-Drinfeld uniformization}. We consider the diagonal morphism 
\begin{equation*}
\theta: \mathfrak{X}\rightarrow \mathfrak{X}^{3}
\end{equation*} of $\mathfrak{X}$ into the triple fiber product $\mathfrak{X}^{3}$. We obtain thus a cycle class $\theta_{*}[\mathfrak{X}\otimes\QQ]\in \mathrm{CH}^{2}(\mathfrak{X}^{3}\otimes\QQ)$ in the Chow group of $\mathfrak{X}^{3}\otimes\QQ$. Since the triple $\frakm^{[p]}_{\triplef}$ is residually irreducible, the  K\"{u}nneth formula implies that
\begin{equation*}
\rmH^{3}(\mathfrak{X}^{3}\otimes{{\QQ}^{\ac}}, \calO(2))_{\frakm^{[p]}_{\triplef}}=(\otimes^{3}_{i=1}\rmH^{1}(X_{\QQ^{\ac}}, \calO(1))_{\frakm^{[p]}_{i}})(-1)
\end{equation*}
and $\rmH^{*}(\mathfrak{X}^{3}\otimes{{\QQ}^{\ac}}, \calO(2))_{\frakm_{\triplef}^{[p]}}$ vanishes identically for $*\neq 3$. Thus the cycle class map and the Hochschild-Serre spectral sequence induces the following {Abel-Jacobi map}
\begin{equation*}
\mathrm{AJ}_{\triplef}: \mathrm{CH}^{2}(\mathfrak{X}^{3}\otimes \QQ)\rightarrow \rmH^{1}(\QQ, \rmH^{3}(\mathfrak{X}^{3}\otimes{{\QQ}^{\ac}}, \calO(2))_{\frakm^{[p]}_{\triplef}}).
\end{equation*}
We denote by $\rmM^{[p]}_{n}(\triplef)$ the Galois module over $\calO_{n}$ given by 
\begin{equation*}
\otimes^{3}_{i=1}\rmH^{1}(X_{{\QQ}^{\ac}}, \calO(1)){/I^{[p]}_{i, n}}
\end{equation*}
The Abel-Jacobi map composed with the canonical map 
\begin{equation*}
\rmH^{3}(X^{3}\otimes{{\QQ}^{\ac}}, \calO(2))_{\frakm_{\triplef}^{[p]}}\rightarrow \rmM^{[p]}_{n}(\triplef)(-1)
\end{equation*}
gives rise to the following mod $\varpi^{n}$-version of the Abel-Jacobi map
\begin{equation*}
\mathrm{AJ_{\triplef, n}}: \Ch^{2}(\mathfrak{X}^{3}\otimes \QQ)\rightarrow \rmH^{1}(\QQ, \rmM^{[p]}_{n}(\triplef)(-1)).
\end{equation*}
We thus obtain a global cohomology class $\Theta^{[p]}_{n}\in \rmH^{1}(\QQ, \rmM^{[p]}_{n}(\triplef)(-1))$ by applying the map $\mathrm{AJ_{\triplef, n}}$ to the cycle $\theta_{*}[\mathfrak{X}\otimes \QQ]\in \mathrm{CH}^{2}(\mathfrak{X}^{3}\otimes \QQ)$.

On the other hand, let $\triplef^{B}=(f^{B}_{1}, f^{B}_{2}, f^{B}_{3})\in S^{B}_{2}(N^{+}, \calO)^{\oplus 3}$ be the \emph{Jacquet-Langlands transfer} of $\triplef=(f_{1}, f_{2}, f_{3})$ in the space of \emph{quaternionic modular forms}  $S^{B}_{2}(N^{+}, \calO)^{3}$ as in \cite[Definition 1.1]{BD-Main}. We consider the following period integral
\begin{equation*}
I(f^{B}_{1}, f^{B}_{2}, f^{B}_{3})=\sum_{z\in X^{B}} f^{B}_{1}(z)f^{B}_{2}(z)f^{B}_{3}(z).
\end{equation*} 
By the main result of  \cite{KH91} which resolves a conjecture of Jacquet, this  period integral is non-vanishing if $L(f_{1}\otimes f_{2}\otimes f_{3}, 2)$ is non-vanishing. Our first goal of this article is to provide an explicit relation between the cohomology class $\Theta^{[p]}_{n}$ and the period integral $I(f^{B}_{1}, f^{B}_{2}, f^{B}_{3})$. To compare these two objects, we pass the class $\Theta^{[p]}_{n}$ living in the global Galois cohomology group $\rmH^{1}(\QQ, \rmM^{[p]}_{n}(\triplef)(-1))$ to its \emph{singular quotient}  $\rmH^{1}_{\sing}(\QQ_{p}, \rmM^{[p]}_{n}(\triplef)(-1))$ at $p$
whose definition is recalled in \eqref{fin-sing}. The singular quotient $\rmH^{1}_{\sing}(\QQ_{p}, \rmM^{[p]}_{n}(\triplef)(-1))$ itself has a geometric description and this description is referred to as the \emph{ramified arithmetic level raising} in the recent literature \cite{LT}, \cite{LTXZZ}. Here the word ramified refers to the fact that the relevant Shimura variety has bad reduction at the prime $p$. Our first main result is such an arithmetic level raising theorem for the triple product of Shimura curves. 
\begin{thm}[Ramified arithemtic level raising ]\label{level-raise-intro}
Let  $p$ be an $n$-admissible prime for the triple $\triplef=(f_{1}, f_{2}, f_{3})$.  For each $i=1, 2, 3$, assume that 
\begin{enumerate}
\item the maximal ideal $\frakm_{i}$ is residually irreducible;
\item each $S^{B}_{2}(N^{+}, \calO)_{\frakm_{i}}$ is a free rank $1$ module over $\TT_{\frakm_{i}}$.
\end{enumerate}
Then the singular quotient $\rmH^{1}_{\sing}(\QQ_{p}, \rmM^{[p]}_{n}(\triplef)(-1))$ is free of rank $3$ over $\calO_{n}$ and we have an isomorphism 
\begin{equation}\label{intro-equ}
 \rmH^{1}_{\sing}(\QQ_{p}, \rmM^{[p]}_{n}(\triplef)(-1))\cong \otimes^{3}_{i=1}S^{B}_{2}(N^{+}, \calO)/I_{i,n}.
\end{equation}
\end{thm}
This theorem is proved in Theorem \ref{arithmetic-level-raising} and Corollary \ref{main-coro}. We remark that the hypothesis $(2)$ on the forms $f_{i}$ is called \emph{$l$-isolated} in \cite{BD-Main}.  One can achieve hypothesis $(2)$ by imposing conditions on the residual Galois representation $\bar{\rho}_{i}$ of $\rho_{i}$ and then applying the Taylor-Wiles argument refined by Diamond in \cite{Diamond-TW}. For example, one can impose the following conditions in \cite[Hypothesis $\mathrm{CR}^{+}$]{CH-1}.
\begin{ass}[$\mathrm{CR}^{+}$]  For $i=1, 2, 3$, the residual Galois representation $\bar{\rho}_{i}$ satisfies the following assumptions
\begin{enumerate}
\item $\bar{\rho}_{i}$ is absolutely irreducible when restricted to $G_{\QQ(\sqrt{p^{*}})}$ where $p^{*}=(-1)^{\frac{p-1}{2}}p$;
\item If $q\mid N^{-}$ and $q\equiv \pm1\mod l$, then $\bar{\rho}_{i}$ is ramified;
\item If $q\mid \mid N^{+}$ and $q\equiv 1\mod l$, then $\bar{\rho}_{i}$ is ramified;
\item The Artin conductor $N_{\bar{\rho}_{i}}$ is prime to $N/N_{\bar{\rho}_{i}}$.
\end{enumerate}
\end{ass}
We refer the reader to \cite[Proposition 6.8]{CH-1} for an exposition of the Taylor-Wiles argument mentioned above to achieve hypothesis $(2)$ in Theorem \ref{level-raise-intro}. The proof of Theorem \ref{level-raise-intro} is geometric in nature by analyzing the semistable model of  $X^{3}$ at an $n$-admissible prime $p$. To calculate the singular quotient, we use the machineries developed in \cite{Liu-cubic}, in particular the so-called \emph{potential map} which we give a down-to-earth introduction in the first part of the article. The right-hand side of the isomorphism \eqref{intro-equ} can be made even more geometrically in terms of the $\emph{component groups}$ of the Shimura curves at $p$. Let $\mathcal{J}$ be the N\'{e}ron model of the Jacobian $\mathrm{Jac}(X_{\QQ_{p^{2}}})$ of $X_{\QQ_{p^{2}}}$ and let $\Phi$ be the group of connected components of the special fiber of $\mathcal{J}$, then the isomorphism in \eqref{intro-equ} comes more canonically from the following isomorphism
\begin{equation}
 \rmH^{1}_{\sing}(\QQ_{p}, \rmM^{[p]}_{n}(\triplef)(-1))\cong \oplus^{3}_{j=1}(\otimes^{3}_{i=1}\Phi_{\calO} /I^{[p]}_{i,n}). 
\end{equation}
Comparing to the pioneering works on arithmetic level raising \cite{BD-Main}, \cite{Liu-cubic}, \cite{LTXZZ}, one of the main difference in our setting is that the  singular quotient is of rank $3$ in our settings as opposed to being of rank $1$ in all the previously mentioned works. This interesting phenomenon suggests some richer structures hidden in our setting which we do not completely understand yet. Nevertheless, the above theorem still provides the natural setting for relating the cohomology class  $\Theta^{[p]}_{n}$ to the period integral $I(f^{B}_{1}, f^{B}_{2}, f^{B}_{3})=\sum_{z\in X^{B}} f^{B}_{1}(z)f^{B}_{2}(z)f^{B}_{3}(z).$ In order to state this relation, we introduce a natural pairing 
\begin{equation*}
\begin{aligned}
(\hphantom{a},\hphantom{a}):\otimes^{3}_{i=1}S^{B}_{2}(N^{+}, \calO)[I_{i, n}]\times \otimes^{3}_{i=1}S^{B}_{2}(N^{+}, \calO)/(I_{i, n})\rightarrow \calO_{n}\\
\end{aligned}
\end{equation*}
in \eqref{pairing}. 
Let $\partial_{p}\Theta^{[p]}_{n}$ be the image of $\Theta^{[p]}_{n}$ in $\rmH^{1}_{\sing}(\QQ_{p}, \rmM^{[p]}_{n}(\triplef)(-1))$. For $j=1, 2, 3$, we denote by 
\begin{equation*}
\partial^{(j)}_{p}\Theta^{[p]}_{n}\in \otimes^{3}_{i=1}S^{B}_{2}(N^{+}, \calO){/I_{i,n}}
\end{equation*}
its projection to the $j$-th copy of 
\begin{equation*}
\oplus^{3}_{j=1}(\otimes^{3}_{i=1}S^{B}_{2}(N^{+}, \calO){/I_{i,n}}).
\end{equation*} 
Then our second main result is the following explicit reciprocity formula which we will call the \emph{first reciprocity law}.
\begin{thm}\label{recip-intro}
Let $p$ be an $n$-admissible prime for the triple $\triplef$. We assume the assumptions in Theorem \ref{level-raise-intro} are satisfied. Let $\phi_{1}\otimes \phi_{2}\otimes \phi_{3}\in \otimes^{3}_{i=1}S^{B}_{2}(N^{+}, \calO)[I_{i,n}]$. Then we have the following formula,
\begin{equation}\label{reci-formula}
(\partial^{(j)}_{p}\Theta^{[p]}_{n}, \phi_{1}\otimes \phi_{2}\otimes \phi_{3})=(p+1)^{3}\sum_{z\in X^{B}}\phi_{1}(z)\phi_{2}(z)\phi_{3}(z)
\end{equation} 
for $j=1, 2, 3$. 
\end{thm}
This formula is the analogue of the \emph{first reciprocity law} in \cite[Theorem 4.1]{BD-Main}, the \emph{congruence formulae} in \cite[Theorem 4.11]{Liu-HZ} and \cite[Theorem 4.5]{Liu-cubic}. The right-hand side of \eqref{reci-formula} obviously represents the period integral $I(f^{B}_{1}, f^{B}_{2}, f^{B}_{3})$ and this formula suggests that the class $\Theta^{[p]}_{n}$ should be useful to bound the Selmer group. However as we explained before, the singular quotient is of rank $3$ and the class $\Theta^{[p]}_{n}$ itself is not enough to bound the Selmer group. Combining the present work with the companion article \cite{Wang} suggests the following intuitive picture: the Selmer group of the level raised triple $\triplef^{[p]}$ skips the possibility of being rank $1$ and goes directly to rank $3$ from the rank $0$ Selmer group for the original triple $\triplef$. Therefore it should be reasonable to expect three global classes that fill up the singular quotient. We invite the reader to compare this with the paper of Darmon-Rotger \cite{DR-2}. They considered the so-called \emph{unbalanced case} where the weights of the triple of modular forms are $(2, 1, 1)$. The singular quotient at $p$ there is a direct sum of four one-dimensional spaces. They managed to find $4$ global classes starting from the primitive diagonal cycle class similar to our $\Theta^{[p]}_{n}$ in this article using techniques from $p$-adic deformations (Hida families) which fill up the singular quotient and such that each class only appears in exactly one of the four one-dimensional spaces. Using these classes, they can bound certain equivariant part of the Mordell-Weil group of an elliptic curve but not the Selmer group. This suggests that we should be able modify the class $\Theta^{[p]}_{n}$ and construct three classes which each sits in exactly one of the three direct summands of 
\begin{equation*}
\oplus^{3}_{j=1}(\otimes^{3}_{i=1}S^{B}_{2}(N^{+}, \calO){/I_{i, n}}).
\end{equation*}
Moreover these three classes should satisfy a similar reciprocity law as the one we obtained for $\Theta^{[p]}_{n}$ in Theorem \ref{recip-intro}. Using these conjectural classes, we are indeed able to bound the Selmer group for the triple product motive. Therefore we regard the present work as the first step towards the rank $0$ case of the Bloch-Kato conjecture for the triple product motive of modular forms. 

\subsection{The symmetric cube motive}
It is natural to consider the degenerate case when $f_{1}=f_{2}=f_{3}=f$ for a single modular form $f\in S^{\new}_{2}(\Gamma_{0}(N))$. We denote by $\rmV_{f}$ the Galois representation attached to $f$.  In this case the triple tensor product representation factors as 
\begin{equation*}
\rmV_{f}^{\otimes 3}(-1)= \mathrm{Sym}^{3} \rmV_{f}(-1)\oplus \rmV_{f}\oplus \rmV_{f}.
\end{equation*}
Correspondingly, we have a factorization of the $L$-function
\begin{equation*}
L(f\otimes f \otimes f, s)= L(\mathrm{Sym}^{3}f, s) L(f, s-1)^{2}.
\end{equation*}
We introduce in this case the notion of \emph{ $(n, 1)$-admissible primes} for $f$.  A prime $p$ is an $(n, 1)$-admissible prime for $f$ if 
\begin{enumerate}
\item $p\nmid Nl$;
\item $l\nmid p^{2}-1 $;
\item $\varpi^{n}\mid p+1-a_{p}(f)$.
\end{enumerate}
Comparing to the definition of an $n$-admissible prime for a triple $(f_{1}, f_{2}, f_{3})$,  we require here that $\epsilon_{p,i}=1$ for all $i=1, 2, 3$. We let $\triplef=(f, f, f)$. Let $\rmM^{\diamond}_{n}(\triplef)(-1)$ be the symmetric cube component of $\rmM^{[p]}_{n}(\triplef)(-1)$.  Then we find that the singular quotient $\rmH^{1}_{\sing}(\QQ_{p}, \rmM^{\diamond}_{n}(\triplef)(-1))$ at $p$ is in fact of rank $1$. Therefore the Euler system argument can be applied here. However, the factorization of the $L$-function and the Galois representation suggest that the results concerning the rank $0$ case of the Bloch-Kato conjecture here will be conditional on the non-vanishing of $L(f, 1)$. Therefore we have the following result.
\begin{thm}\label{main-symm}
Suppose that the modular form $f$ satisfies the following assumptions: 
\begin{enumerate}
\item the maximal ideals $\frakm_{f}$ are all residually irreducible;
\item the $\TT_{\frakm_{f}}$-module $S^{B}_{2}(N^{+}, \calO)_{\frakm_{f}}$ is free of rank $1$;
\item the residual Galois representation $\bar{\rho}_{f}$ is surjective;
\item  the value $L(f, 1)$ is non-vanishing.
\end{enumerate}
Assume that value $L(\mathrm{Sym}^{3}f, 2)$ is non-zero. Then the Bloch-Kato Selmer group for $\mathrm{Sym}^{3} \rmV_{f}(-1)$
\begin{equation*}
\rmH^{1}_{f}(\QQ, \mathrm{Sym}^{3} \rmV_{f}(-1))=0. 
\end{equation*}
\end{thm}
This theorem is proved in the last section of this article and is based on a familiar Euler system argument, see \cite{Gro-Koly}, \cite{BD-Main}, \cite{Liu-HZ} and \cite{Liu-cubic} for a few examples. A similar result in the rank $1$ case for the Bloch-Kato Selmer group of $\mathrm{Sym}^{3} \rmV_{f}(-1)$ will be proved in \cite{Wang}.

\subsection{Bipartite Euler system and other applications}
We close this introduction by discussing some applications and questions the present work points to. 
In the companion work \cite{Wang}, we prove the \emph{second reciprocity law} in our setting which should have applications to the rank $1$ case of the Bloch-Kato conjecture for the triple product motive of modular forms. It also implies that the diagonal cycle classes $\Theta^{[p]}_{n}$ and the period integrals $I(f^{B}_{1}, f^{B}_{2}, f^{B}_{3})$ form an analogue of the Bipartite Euler system in the sense of Howard \cite{How}. Note that the definition there requires the only Frobenius eigenvalues on the local Galois representation at an $n$-admissible prime $p$ are $p$ and $1$ and therefore the singular quotient is of  rank $1$. This is obviously different from our setting. Therefore what we have for the triple product motive is only an analogue of his  Bipartite Euler system. On the other hand, the classes $\Theta^{\diamond[p]}_{n}$ we produced for the symmetric cube motive of modular forms do form a Bipartite Euler system in his sense.  

Secondly, by replacing the $l$-adic \'{e}tale cohomology and $l$-adic Abel-Jacobi map in the present work with crystalline cohomology and the $p$-adic Abel-Jacobi map one should be able to prove certain type of $p$-adic Gross-Zagier formulas for certain triple product $p$-adic $L$-functions. See \cite{Hsieh} for the relevant construction of $p$-adic $L$-functions and see \cite{BD-unif} for the type of $p$-adic Gross-Zagier formula for Heegner points on Shimura curves. 

\subsection{Notations and conventions} We will use common notations and conventions in algebraic number theory and algebraic geometry. The cohomologies  of schemes appear in this article will be understood as computed over the \'{e}tale sites. 

For a field $K$, we denote by $K^{\ac}$ a separable closure of $K$ and put $G_{K}:=\Gal(K^{ac}/K)$ the Galois group of $K$. We let $\Adel$ be the ring of ad\`{e}les over $\QQ$ and $\Adel^{\infty}$ be the subring of finite ad\`{e}les.  For a prime $p$, $\Adel^{\infty, (p)}$ is the prime-to-$p$ part of  $\Adel^{\infty}$. 

Let $K$ be a local field with ring of integers $\calO_{K}$ and residue field $k$. We let $I_{K}$ be the inertia subgroup of $G_{K}$. Suppose $\rmM$ is a $G_{K}$-module. Then the finite part $\rmH^{1}_{\mathrm{fin}}(K, \rmM)$ of $\rmH^{1}(K, \rmM)$ is defined to be $H^{1}(k, \rmM^{I_{K}})$ and the singular part $\rmH^{1}_{\mathrm{sing}}(K, \rmM)$ of $\rmH^{1}(K, \rmM)$ is defined to be the quotient of $\rmH^{1}(K, \rmM)$ by the image of $\rmH^{1}_{\mathrm{fin}}(K, \rmM)$. 

Two quaternion algebras will be used in this article: the definite quaternion algebra $B$ with discriminant $N^{-}$ and the indefinite quaternion algebra $B^{\prime}$ with discriminant $pN^{-}$.

\subsection*{Acknowledgements} We would like to thank Henri Darmon and Liang Xiao for valuable discussions about this work. We would like to thank Minglun Hsieh for introducing the work of Bertonili-Darmon to the author in graduate school and for his interest to the present work. We also acknowledge the deep debt this article owes to the pioneering works of Bertonili-Darmon and Yifeng Liu. We are truly grateful to Pengfei Guan for his generous support in the unsettling period of time during which this article was written.

\section{Review of weight spectral sequence}
\subsection{Nearby cycles on semi-stable schemes} Let $K$ be a henselian discrete valuation field with residue field $k$ of characteristic $p$ and valuation ring $\calO_{K}$. We fix a uniformizer $\pi$ of $\calO_{K}$. We set $S=\Spec(\calO_{K})$, $s=\Spec(k)$ and $\eta=\Spec(K)$. Let $K^{\ac}$ be a separable closure of $K$ and $K_{\ur}$ the maximal unramified extension of $K$ in $K^{\ac}$. We denote by $k^{\ac}$ the residue field of $K_{\ur}$. 
Let $I_{K}=\Gal(K^{\ac}/K_{\ur})\subset G_{K}=\Gal(K^{\ac}/K)$ be the inertia group. Let $l$ be a prime different from $p$. We set $t_{l}: I_{K}\rightarrow \Lambda(1)$ to be the canonical surjection given by 
\begin{equation*}
\sigma \mapsto (\sigma(\pi^{1/l^{m}})/\pi^{1/l^{m}})_{m}
\end{equation*} 
for every $\sigma\in I_{K}$. 

Let $\mathfrak{X}$ be a \emph{strict semi-stable scheme} over $S$ purely of relative dimension $n$ which we also assume to be proper. This means that $\mathfrak{X}$ is locally of finite presentation and Zariski locally \'{e}tale over $$\Spec(\calO_{K}[X_{1}, \cdots, X_{n}]/(X_{1}\cdots X_{r}-\pi))$$ for some integer $1\leq r\leq n$. We let $X_{k}$ be the special fiber of $\mathfrak{X}$ and $X_{k^{\ac}}$ be its base-change to $k^{\ac}$. Let $X=\mathfrak{X}_{\eta}$ be the generic fiber of $\mathfrak{X}$ and $X_{K_{\ur}}$ be its base-change to $K_{\ur}$. We have the following natural maps
\begin{equation*}
\begin{aligned}
&i:X_{k}\rightarrow \mathfrak{X},\\
&j: X\rightarrow \mathfrak{X}, \\
&\bar{i}: X_{k^{\ac}}\rightarrow \mathfrak{X}_{\calO_{K_{\ur}}},\\ 
& \bar{j}: X_{K_{\ur}}\rightarrow \mathfrak{X}_{\calO_{K_{\ur}}}. \\
\end{aligned}
\end{equation*}
Throughout this paper, we fix a prime $l\geq 5$. Let $\Lambda$ be $\ZZ/l^{v}, \ZZ_{l}$ or a finite extension of $\ZZ_{l}$. We define 
the  \emph{Nearby cycle sheaf} by
\begin{equation*}
R^{q}\Psi(\Lambda)= \bar{i}^{*}R^{q}\bar{j}_{*}\Lambda
\end{equation*}
and the \emph{Nearby cycle complex} by 
\begin{equation*}
R\Psi(\Lambda)= \bar{i}^{*}R\bar{j}_{*}\Lambda.
\end{equation*}
We regard the latter as an object in the derived category $D^{+}(X_{k^{\ac}}, \Lambda[I_{K}])$ of sheaves of $\Lambda$-modules with continuous $I_{K}$-actions. By proper base change, we always have 
\begin{equation*}
\rmH^{*}(X_{k^{\ac}}, R\Psi(\Lambda))=\rmH^{*}(X_{K^{\ac}}, \Lambda).
\end{equation*} 

Let $D_{1},\cdots, D_{m} $ be the set of irreducible components of $X_{k}$. For each index set $I\subset \{1, \cdots, m\}$ of cardinality $p$, we set $X_{I, k}=\cap_{i\in I} D_{i}$. This is a smooth scheme of dimension $n-p$. For $1\leq p \leq m-1$, let
\begin{equation}
X^{(p)}_{k}=\bigsqcup_{I\subset \{1, \cdots, m\}, \mathrm{Card}(I)=p+1} X_{I, k}
\end{equation}
 and 
 \begin{equation}
 a_{p}: X^{(p)}_{k}\rightarrow X_{k}
 \end{equation}
 be the projection, we have $a_{p *}\Lambda=\wedge^{p+1}a_{0 *}\Lambda$. Consider the Kummer exact sequence in the case $\Lambda=\ZZ/l^{v}$
 \begin{equation}
 0\rightarrow \Lambda(1)\rightarrow \calO^{*}_{X} \rightarrow \calO^{*}_{X}\rightarrow 0.
 \end{equation}
 Let $\partial(\pi)\in i^{*}R^{1}j_{*}\Lambda(1) $ be the image of $\pi$ under the coboundary map by applying $i^{*}Rj_{*}$ to the above exact sequence. We let $\theta: \Lambda_{X_{k}}\rightarrow i^{*}R^{1}j_{*}\Lambda(1)$ be the map sending $1$ to $\partial(\pi)$ and $\delta:\Lambda_{X_{k}}\rightarrow a_{0*}\Lambda$ be the canonical map. Then we have the following results regarding to the resolution of the Nearby cycle sheaf.
\begin{proposition}\label{nearby-cycle-q}
We have the following.
\begin{enumerate}
\item There is an isomorphism of exact sequences 
\begin{equation*}
\begin{tikzcd}
\Lambda_{X_{k}}\arrow{r}{\delta}\arrow{d}&a_{0*}\Lambda\arrow{r}{\delta\wedge}\arrow{d}&\cdots\arrow{r}{\delta\wedge}\arrow{d}&a_{n*}\Lambda\arrow{r}\arrow{d}&0\\
\Lambda_{X_{k}}\arrow{r}{\theta}&i^{*}R^{1}j_{*}\Lambda(1)\arrow{r}{\theta\cup}&\cdots\arrow{r}{\theta\cup}&i^{*}R^{n+1}j_{*}\Lambda(n+1)\arrow{r}&0\\
\end{tikzcd}.
\end{equation*}
\item For $p\geq 0$, we have an exact sequence
\begin{equation*}
\begin{tikzcd}
R^{p}\Psi(\Lambda)\arrow{r}{\theta\cup}&i^{*}R^{p+1}j_{*}\Lambda(1)\arrow{r}{\theta\cup}&\cdots\arrow{r}{\theta\cup}&i^{*}R^{n+1}j_{*}\Lambda(n+1-p)\arrow{r}&0\\
\end{tikzcd}.
\end{equation*}
\item For $p\geq 0$, we have a quasi-isomorphism of complexes
\begin{equation*}
R^{p}\Psi(\Lambda)[-p]\xrightarrow{\sim} [a_{p*}\Lambda(-p)\xrightarrow{\delta\wedge}\cdots\xrightarrow{\delta\wedge} a_{n*}\Lambda(-p)\rightarrow0]
\end{equation*}
\end{enumerate}
\end{proposition}
\begin{proof}
The first two statements are taken from  \cite[Corollary 1.3]{Saito} and the last one is an immediate consequence of the first two. 
\end{proof}
\subsection{Monodromy filtration and spectral sequence} Suppose $A$ is an object in an abelian category and $N$ is a nilpotent endomorphism on $A$. We define two filtrations on $A$.
\begin{itemize}
\item We define the kernel filtration $F_{\bullet}$ by putting $F_{p}A=\ker(N^{p+1}: A\rightarrow A)$ for $p\geq 0$.
\item We define the image filtration $G^{\bullet}$ by putting $G^{q}A=\im(N^{q}: A\rightarrow A)$ for $q>0$.
\end{itemize}
Using these two filtrations, 
\begin{itemize}
\item we define the convolution filtration $M_{\bullet}A$ by $M_{r}A=\oplus_{p-q=r}F_{p}A\cap G^{q}A$. 
\end{itemize}
To calculate the graded piece of the convolution filtration, we define the induced $G^{\bullet}$-filtration on the graded piece of $F_{\bullet}$ by $G^{q}Gr^{F}_{p}A=\im(G^{q}A\cap F_{p}A\rightarrow Gr^{F}_{p}A)$.  Then we see that the $r$-th graded piece of the convolution filtration $M_{\bullet}A$ is given by
\begin{equation*}
Gr^{M}_{r}A\cong\bigoplus_{p-q=r}Gr^{q}_{G}Gr^{F}_{p}A.
\end{equation*}
The convolution filtration in this case is known as the \emph{Monodromy filtration} and it is characterized by
\begin{enumerate}
\item $M_{n}A=0$ and $M_{-n-1}A=0$.
\item $N: A\rightarrow A$ sends $M_{r}A$ into $M_{r-2}A$ for $r\in\ZZ$.
\item $N^{r}: Gr^{M}_{r}A\rightarrow Gr^{M}_{-r}A$ is an isomoprhism. 
\end{enumerate}

Let $A=R\Psi(\Lambda)$. Let $T$ be an element in $I_{K}$ such that $t_{l}(T)$ is a generator of $\Lambda(1)$ then $T$ induces a nilpotent operator $T-1$ on $R\Psi(\Lambda)$. Let $N=(T-1)\otimes\breve{T}$ where $\breve{T}\in \Lambda(-1)$ be the dual of $t_{l}(T)$. Then with respect to $N$, we have the following characterization of the Monodromy filtration on $R\Psi(\Lambda)$
\begin{enumerate}
\item $M_{n}R\Psi(\Lambda)=0$ and $M_{-n-1}R\Psi(\Lambda)=0$.
\item $N: R\Psi(\Lambda)(1)\rightarrow R\Psi(\Lambda)$ sends $M_{r}R\Psi(\Lambda)(1)$ into $M_{r-2}R\Psi(\Lambda)$ for $r\in\ZZ$.
\item $N^{r}: Gr^{M}_{r}R\Psi(\Lambda)(r)\rightarrow Gr^{M}_{-r}R\Psi(\Lambda)$ is an isomoprhism. 
\end{enumerate}
We can use Proposition \ref{nearby-cycle-q} to calculate the Monodromy filtration on $R\Psi(\Lambda)$. In addition to Proposition \ref{nearby-cycle-q}, we need the following results in \cite[Lemma 2.5, Corollary 2.6]{Saito}. 
\begin{enumerate}
\item The kernel filtration $F_{p}R\Psi(\Lambda)$ is given by the canonical truncated filtration $\tau_{\leq p}R\Psi(\Lambda)$ and therefore 
\begin{equation}\label{grad-ker}
Gr^{F}_{p}R\Psi(\Lambda)\cong R^{p}\Psi(\Lambda)[-p]\cong [a_{p*}\Lambda(-p)\xrightarrow{\delta\wedge}\cdots\xrightarrow{\delta\wedge} a_{n*}\Lambda(-p)\rightarrow0].
\end{equation}
\item The image filtration $G^{q}Gr^{F}_{p}R\Psi(\Lambda)$ on $Gr^{F}_{p}R\Psi(\Lambda)$ is given by the truncation in \eqref{grad-ker} to $p+q$ position
\begin{equation*}
G^{q}Gr^{F}_{p}R\Psi(\Lambda)= [a_{p+q*}\Lambda(-p)\xrightarrow{\delta\wedge}\cdots\xrightarrow{\delta\wedge} a_{n*}\Lambda(-p)\rightarrow0].
\end{equation*}
\item Combining the above two results, we have 
\begin{equation*}
Gr^{q}_{G}Gr^{F}_{p}R\Psi(\Lambda)=a_{p+q*}\Lambda[-p-q](-p)
\end{equation*}
and therefore we arrive at the following equation 
\begin{equation*}
Gr^{M}_{r}R\Psi(\Lambda)=\bigoplus_{p-q=r}a_{p+q*}\Lambda[-p-q](-p).
\end{equation*}
\end{enumerate}
The Monodromy filtration induces the \emph{weight spectral sequence}
\begin{equation}\label{wt-seq}
\rmE^{p,q}_{1}= \rmH^{p+q}(X_{k^{\ac}}, Gr^{M}_{-p}R\Psi(\Lambda))\Rightarrow \rmH^{p+q}(X_{k^{\ac}}, R\Psi(\Lambda))=\rmH^{p+q}(X_{K^{\ac}}, \Lambda).
\end{equation}
The $\rmE_{1}$-term of this spectral sequence can be made explicit by
\begin{equation*}
\begin{aligned}
\rmH^{p+q}(X_{k^{\ac}}, Gr^{M}_{-p}R\Psi(\Lambda))&=\bigoplus_{i-j=-p, i\geq0, j\geq0}\rmH^{p+q-(i+j)}(X^{(i+j)}_{k^{\ac}}, \Lambda(-i))\\
&=\bigoplus_{i\geq\mathrm{max}(0, -p)}\rmH^{q-2i}(X^{(p+2i)}_{k^{\ac}}, \Lambda(-i)).\\
\end{aligned}
\end{equation*}
This spectral sequence is first found by Rapoport-Zink in \cite{RZ} and thus is also known as the Rapoport-Zink spectral sequence. 

\subsection{Examples in dimension $1$ and $3$} We will make the weight spectral sequence explicit in dimension $1$ and dimension $3$ which are the only cases that will be used in the computations later. The following convention will used throughout this article: we will write $\rmH^{*}(a_{p*}\Lambda)$ instead of $\rmH^{*}(X_{k^{\ac}}, a_{p*}\Lambda)=\rmH^{*}(X^{(p)}_{k^{\ac}}, \Lambda)$.
\subsubsection{One dimensional case}: Let $\mathfrak{X}$ be a relative curve over $\Spec(\calO_{K})$. Then we immediately calculate that
\begin{equation}
\begin{aligned}
&Gr^{M}_{-1}R\Psi(\Lambda)=a_{1*}\Lambda[-1], \\
&Gr^{M}_{0}R\Psi(\Lambda)=a_{0*}\Lambda,  \\
&Gr^{M}_{1}R\Psi(\Lambda)=a_{1*}\Lambda[-1](-1). \\
\end{aligned}
\end{equation}
The $\rmE_{1}$-page of the weight spectral sequence is given by
\begin{center}
\begin{tikzpicture}[thick,scale=0.8, every node/.style={scale=0.8}]
  \matrix (m) [matrix of math nodes,
    nodes in empty cells,nodes={minimum width=5ex,
    minimum height=5ex,outer sep=-5pt},
    column sep=1ex,row sep=1ex]{
                &      &     &     & \\
          2     &  \rmH^{0}(a_{1*}\Lambda)(-1) &  \rmH^{2}(a_{0*}\Lambda)  & & \\
          1     &       & \rmH^{1}(a_{0*}\Lambda) &    & \\
          0     &    & \rmH^{0}(a_{0*}\Lambda) &  \rmH^{0}(a_{1*}\Lambda) &\\
    \quad\strut &   -1  &  0  &  1  & \strut \\};
\draw[thick] (m-1-1.east) -- (m-5-1.east) ;
\draw[thick] (m-5-1.north) -- (m-5-5.north) ;
\end{tikzpicture}
\end{center}
and it clearly degenerates at the $\rmE_{2}$-page.  We therefore have the Monodromy filtration
\begin{equation*}
0\subset^{\rmE^{1,0}_{2}} M_{1}\rmH^{1}(X_{K^{\ac}}, \Lambda)\subset^{\rmE^{0,1}_{2}} M_{0}\rmH^{1}(X_{K^{\ac}}, \Lambda)\subset^{\rmE^{-1,2}_{2}} M_{-1}\rmH^{1}(X_{K^{\ac}},\Lambda)
\end{equation*}
with graded pieces given by
\begin{equation*}
\begin{aligned}
&Gr^{M}_{-1}\rmH^{1}(X_{K^{\ac}}, \Lambda)=\ker[\rmH^{0}(a_{1*}\Lambda(-1))\xrightarrow{\tau} \rmH^{2}(a_{0*}\Lambda)]\\
&Gr^{M}_{0}\rmH^{1}(X_{K^{\ac}}, \Lambda)= \rmH^{1}(a_{0*}\Lambda);\\
&Gr^{M}_{1}\rmH^{1}(X_{K^{\ac}}, \Lambda)= \coker[\rmH^{0}(a_{0*}\Lambda)\xrightarrow{\rho} \rmH^{0}(a_{1*}\Lambda)];\\
\end{aligned}
\end{equation*}
where $\tau$ is the \emph{Gysin morphism} and $\rho$ is the \emph{restriction morphism}. 
Note that we have the following commutative diagram
\begin{equation}\label{picard-lef}
\begin{tikzcd}
\rmH^{1}(X_{K^{\ac}}, \Lambda(1)) \arrow[r] \arrow[d, "N"] & \ker[\rmH^{0}(a_{1*}\Lambda)\xrightarrow{\tau} \rmH^{2}(a_{0*}\Lambda)(1)] \arrow[d, "N"] \\
\rmH^{1}(X_{K^{\ac}}, \Lambda)                  & \coker[\rmH^{0}(a_{0*}\Lambda)\xrightarrow{\rho} \rmH^{0}(a_{1*}\Lambda)] .\arrow[l]
\end{tikzcd}
\end{equation}
In this case, we recover the \emph{Picard-Lefschetz formula} if we identify $\rmH^{0}(a_{1*}\Lambda)$ with the \emph{vanishing cycles} $\oplus_{x}R\Phi(\Lambda)_{x}$ on $X_{k^{\ac}}$ where $x$ runs through the set of singular points $X^{(1)}_{k}$. We refer the readers to \cite[Example 2.4.6]{Illusie} for the details about this case.

\subsubsection{Three dimensional case} Let $\mathfrak{X}$ be a relative threefold over $\Spec(\calO_{K})$. We can also easily list the graded pieces of the Monodromy filtration on $R\Psi(\Lambda)$
\begin{equation*}
\begin{aligned}
&Gr^{M}_{-3}R\Psi(\Lambda)=a_{3*}\Lambda[-3], \\
&Gr^{M}_{-2}R\Psi(\Lambda)=a_{2*}\Lambda[-2], \\
&Gr^{M}_{-1}R\Psi(\Lambda)=a_{1*}\Lambda[-1]\oplus a_{3*}\Lambda[-3](-1), \\
&Gr^{M}_{0}R\Psi(\Lambda)=a_{0*}\Lambda\oplus a_{2*}\Lambda[-1](-1),  \\
&Gr^{M}_{1}R\Psi(\Lambda)=a_{1*}\Lambda[-1](-1)\oplus a_{3*}\Lambda[-3](-2),\\
&Gr^{M}_{2}R\Psi(\Lambda)=a_{2*}\Lambda[-2](-2), \\
&Gr^{M}_{3}R\Psi(\Lambda)=a_{3*}\Lambda[-3](-3).\\
\end{aligned}
\end{equation*}
The $\rmE_{1}$-page of the weight spectral sequence is given below

\begin{equation}\label{E1-primitive}
\begin{tikzpicture}[thick,scale=0.6, every node/.style={scale=0.6}]
\matrix (m) [matrix of math nodes,
    nodes in empty cells,nodes={minimum width=5ex,
    minimum height=5ex,outer sep=-5pt},
    column sep=1ex,row sep=1ex]{
                &      &     &     & \\
          6    &\rmH^{0}(a_{3*}\Lambda)(-3) &  \rmH^{2}(a_{2*}\Lambda)(-2)  &\rmH^{4}(a_{1*}\Lambda)(-1) & \rmH^{6}(a_{0*}\Lambda) \\
          5    &    &\rmH^{1}(a_{2*}\Lambda)(-2)  &\rmH^{3}(a_{1*}\Lambda)(-1) &\rmH^{5}(a_{0*}\Lambda)& \\
          4    &       &\rmH^{0}(a_{2*}\Lambda)(-2) &\rmH^{2}(a_{1*}\Lambda)(-1)\oplus\rmH^{0}(a_{3*}\Lambda)(-2) &\rmH^{4}(a_{0*}\Lambda)\oplus\rmH^{2}(a_{2*}\Lambda)(-1)&\rmH^{4}(a_{1*}\Lambda)\\
          3   &        &         &\rmH^{1}(a_{1*}\Lambda)(-1)  &\rmH^{3}(a_{0*}\Lambda)\oplus\rmH^{1}(a_{2*}\Lambda)(-1) &\rmH^{3}(a_{1*}\Lambda)\\
          2   &        &         &\rmH^{0}(a_{1*}\Lambda)(-1)  &\rmH^{2}(a_{0*}\Lambda)\oplus\rmH^{0}(a_{2*}\Lambda)(-1) &\rmH^{2}(a_{1*}\Lambda)\oplus \rmH^{0}(a_{3*}\Lambda)(-1) &\rmH^{2}(a_{2*}\Lambda)\\
           1   &        &         &                             &\rmH^{1}(a_{0*}\Lambda) &\rmH^{1}(a_{1*}\Lambda) &\rmH^{1}(a_{2*}\Lambda)\\
           0   &        &         &       &\rmH^{0}(a_{0*}\Lambda) &\rmH^{0}(a_{1*}\Lambda) &\rmH^{0}(a_{2*}\Lambda)& \rmH^{3}(a_{3*}\Lambda)\\
          \quad\strut & -3  &  -2  &  -1  & 0 &1 &2 &3 & \strut \\};

\draw[thick] (m-1-1.east) -- (m-9-1.east) ;
\draw[thick] (m-9-1.north) -- (m-9-9.north) ;
\end{tikzpicture}
\end{equation}
and it does not necessarily degenerates at $\rmE_{2}$. 

\subsection{Potential map and Galois cohomology} Let $\rmM$ be a $G_{K}$-module over $\Lambda$, then we have the following exact sequence of Galois cohomology groups
\begin{equation}\label{fin-sing}
0\rightarrow \rmH^{1}_{\mathrm{fin}}(K, \rmM)\rightarrow \rmH^{1}(K, \rmM)\xrightarrow{\partial_{p}} \rmH^{1}_{\sing}(K, \rmM)\rightarrow 0  
\end{equation}
where 
\begin{equation}
\rmH^{1}_{\mathrm{fin}}(K, \rmM)=\rmH^{1}(k, \rmM^{I_{K}})
\end{equation}
is called the \emph{unramified} or \emph{finite} part of the cohomology group $\rmH^{1}(K, \rmM)$ and $\rmH^{1}_{\mathrm{sing}}(K, \rmM)$ defined as the quotient of $\rmH^{1}(K, \rmM)$ by its finite part is called the \emph{singular quotient} of  $\rmH^{1}(K, \rmM)$. The natural quotient map $\rmH^{1}(K, \rmM)\xrightarrow{\partial_{p}} \rmH^{1}_{\sing}(K, \rmM)$ will be referred to as the \emph{singular quotient map}. Let $x\in\rmH^{1}(K, \rmM)$, we call the element $\partial_{p}(x)\in \rmH^{1}_{\sing}(K, \rmM)$ the \emph{singular residue} of $x$. Let $\rmM=\rmH^{n}(X_{K^{\ac}},\Lambda(r))$ be $r$-th twist of the middle degree cohomology of $X_{K^{\ac}}$. We review in this section how to calculate the singular quotient using the weight spectral sequence recalled above.  In \cite{Liu-cubic}, the author postulated certain situations where the singular quotient can be calculated by a formula similar to the Picard-Lefschetz formula. What will be presented below is the same as his construction of the \emph{potential map}. We will not recall his general machinery but only the case of a curve or a threefold. We need the following elementary lemma.
\begin{lemma}
Let $\rmM=\rmH^{n}(X_{K^{\ac}},\Lambda(r))$, then we have 
\begin{equation}
\rmH^{1}_{\sing}(K, \rmM)\cong (\frac{\rmM(-1)}{N\rmM})^{G_{k}}
\end{equation}
\end{lemma}
\begin{proof}
This is well-known. The details can be found in \cite[Lemma 2.6]{Liu-cubic} for example. 
\end{proof}
\subsubsection{One dimensional case} In this case, let $\rmM=\rmH^{1}(X_{K^{\ac}}, \Lambda(1))$, we can use the Picard-Lefschetz formula. Recall the diagram 
\begin{equation*}
\begin{tikzcd}
\rmH^{1}(X_{K^{\ac}}, \Lambda(1)) \arrow[r] \arrow[d, "N"] & \ker[\rmH^{0}(a_{1*}\Lambda)\xrightarrow{\tau} \rmH^{2}(a_{0*}\Lambda(1))] \arrow[d, "N"] \\
\rmH^{1}(X_{K^{\ac}}, \Lambda)                  & \coker[\rmH^{0}(a_{0*}\Lambda)\xrightarrow{\rho} \rmH^{0}(a_{1*}\Lambda)]  \arrow[l]
\end{tikzcd}
\end{equation*}
in \eqref{picard-lef}.
Then we have 
\begin{equation}\label{1-sing}
\begin{aligned}
\rmH^{1}_{\sing}(K, \rmM) &\cong (\frac{\rmM(-1)}{N\rmM})^{G_{k}}\cong (\frac{\coker[\rmH^{0}(a_{0*}\Lambda)\xrightarrow{\rho} \rmH^{0}(a_{1*}\Lambda)]}{N\ker[\rmH^{0}(a_{1*}\Lambda)\xrightarrow{\tau} \rmH^{2}(a_{0*}\Lambda(1))]})^{G_{k}}
\end{aligned}
\end{equation}
Composing $\tau$ and $\rho$, we have 
\begin{equation}\label{1-potential}
\begin{aligned}
 (\frac{\coker[\rmH^{0}(a_{0*}\Lambda)\xrightarrow{\rho} \rmH^{0}(a_{1*}\Lambda)]}{N\ker[\rmH^{0}(a_{1*}\Lambda)\xrightarrow{\tau} \rmH^{2}(a_{0*}\Lambda(1)))]})^{G_{k}} \cong\coker[\rmH^{0}(a_{0*}\Lambda)\xrightarrow{\rho} \rmH^{0}(a_{1*}\Lambda)\xrightarrow{\tau}  \rmH^{2}(a_{0*}\Lambda(1))]^{G_{k}}.
\end{aligned}
\end{equation}

\subsubsection{Three dimensional case} In this case, in light of later applications we will only consider $\rmM=\rmH^{3}(X_{K^{\ac}}, \Lambda(2))$. We put the following assumptions on $\rmM$.
\begin{assumption}\label{assump-E}
We assume the weight spectral sequence satisfy the following assumptions.
\begin{enumerate}
\item The weight spectral sequence which converges to $\rmM=\rmH^{3}(X_{K^{\ac}}, \Lambda(2))$ degenerates in its $\rmE_{2}$-page. Thus it induces the following filtration
\begin{equation*}
\begin{aligned}
&0\subset^{\rmE^{3,0}_{2}(2)}M_{3}\rmH^{3}(X_{K^{\ac}}, \Lambda(2))\subset^{\rmE^{2,1}_{2}(2)} M_{2}\rmH^{3}(X_{K^{\ac}}, \Lambda(2))\\&\subset^{\rmE^{1,2}_{2}(2)} M_{1}\rmH^{3}(X_{K^{\ac}}, \Lambda(2)) \subset^{\rmE^{0,3}_{2}(2)}  M_{0}\rmH^{3}(X_{K^{\ac}}, \Lambda(2))\\
&\subset^{\rmE^{-1,4}_{2}(2)}M_{-1}\rmH^{3}(X_{K^{\ac}}, \Lambda(2))\subset^{\rmE^{-2,5}_{2}(2)} M_{-2}\rmH^{1}(X_{K^{\ac}}, \Lambda(2))\\
&\subset^{\rmE^{-3,6}_{2}(2)} M_{-3}\rmH^{3}(X_{K^{\ac}},\Lambda(2))=\rmH^{3}(X_{K^{\ac}}, \Lambda(2)).\\
\end{aligned}
\end{equation*}

\item The term $\rmE^{i, 3-i}_{2}(2)$ in the above filtration which has a non-trivial subquotient that is $G_{k}$-invariant is only $\rmE^{-1, 4}_{2}(2)$. 

\item The term $\rmE^{i, 3-i}_{2}(2)(-1)=\rmE^{i, 3-i}_{2}(1)$ in the above filtration which has a non-trivial subquotient which is $G_{k}$-invariant is only $\rmE^{1, 2}_{2}(1)$. 

\end{enumerate}
\end{assumption}

Under the assumptions above, we have the following commutative diagram 
\begin{equation*}
\begin{tikzcd}
\rmH^{3}(X_{K^{\ac}}, \Lambda(2))^{G_{k}} \arrow[r] \arrow[d, "N"] & (\rmE^{-1,4}_{2}(2))^{G_{k}} \arrow[d, "N"] \\
\rmH^{3}(X_{K^{\ac}}, \Lambda(1))^{G_{k}}                  & (\rmE^{1,2}_{2}(1))^{G_{k}}  \arrow[l]
\end{tikzcd}
\end{equation*}
and from we obtain 
\begin{equation*}
\begin{aligned}
\rmH^{1}_{\sing}(K, \rmM) &\cong (\frac{\rmM(-1)}{N\rmM})^{G_{k}} \cong (\frac{\rmE^{1,2}_{2}(1)}{N \rmE^{-1,4}_{2}(2)})^{G_{k}}\\
\end{aligned}
\end{equation*}
We have
\begin{equation*}
\rmE^{-1,4}_{2}(2)=\frac{\ker[\rmH^{2}(a_{1*}\Lambda(1))\oplus \rmH^{0}(a_{3*}\Lambda)\xrightarrow{(\tau+\rho,\tau)}\rmH^{4}(a_{0*}\Lambda(2))\oplus \rmH^{2}(a_{2*}\Lambda(1))]}{\im[\rmH^{0}(a_{2*}\Lambda)\xrightarrow{(\tau, \rho)} \rmH^{2}(a_{1*}\Lambda(1))\oplus\rmH^{0}(a_{3*}\Lambda) ]}
\end{equation*}
and 
\begin{equation*}
\rmE^{1,2}_{2}(1)=\frac{\ker[\rmH^{2}(a_{1*}\Lambda(1))\oplus \rmH^{0}(a_{3*}\Lambda)\xrightarrow{(\rho,\tau)}\rmH^{2}(a_{2*}\Lambda(1))]}{\im[\rmH^{2}(a_{0*}\Lambda(1))\oplus \rmH^{0}(a_{2*}\Lambda)  \xrightarrow{(\rho,\tau+\rho)} \rmH^{2}(a_{1*}\Lambda(1))\oplus\rmH^{0}(a_{3*}\Lambda) ]}.
\end{equation*}
Then we find that
\begin{equation}\label{potential}
\begin{aligned}
& (\frac{\rmE^{1,2}_{2}(1)}{N \rmE^{-1,4}_{2}(2)})^{G_{k}}=(\frac{\im[\rmH^{2}(a_{1*}\Lambda(1))\xrightarrow{\tau}\rmH^{4}(a_{0*}\Lambda(2))]}{\tau\im[\rmH^{2}(a_{0*}\Lambda(1))\xrightarrow{\rho} \rmH^{2}(a_{1*}\Lambda(1))]})^{G_{k}}.\\
\end{aligned}
\end{equation}
 We will denote by 
 \begin{equation*}
 A^{2}(X_{k}, \Lambda)^{0}=\im[\rmH^{2}(a_{1*}\Lambda(1))\xrightarrow{\tau}\rmH^{4}(a_{0*}\Lambda(2))]^{G_{k}}
\end{equation*} 
which appear on the numerator of \eqref{potential} and by 
\begin{equation*}
A_{2}(X_{k}, \Lambda)^{0}=\im[\rmH^{2}(a_{0*}\Lambda(1))\xrightarrow{\rho}\rmH^{2}(a_{1*}\Lambda(1))]^{G_{k}}
\end{equation*}
which appear on the denominator of  \eqref{potential}. 
Let 
\begin{equation*}
A^{2}(X_{k}, \Lambda)^{0}\xrightarrow{\eta} \rmH^{1}_{\sing}(K, \rmH^{3}(X_{K^{\ac}}, \Lambda(2)))
\end{equation*}
be the natural map and we have the following exact sequence
\begin{equation}\label{sing-exact}
 A_{2}(X_{k}, \Lambda)^{0}\xrightarrow{\nabla}A^{2}(X_{k}, \Lambda)^{0}\xrightarrow{\eta} \rmH^{1}_{\sing}(K, \rmH^{3}(X_{K^{\ac}}, \Lambda(2)))\rightarrow 0.
\end{equation} 
the natural map $\nabla$ induced by $\tau$ is called the potential map, see \cite[Definition 2.5]{Liu-cubic}.  

We will now explain how this map is related to the image of the {Abel-Jacobi map}. Consider the cycle class map 
\begin{equation*}
\mathrm{cl}: \Ch^{2}(X_{K}, \Lambda)\rightarrow \rmH^{4}(X_{K}, \Lambda(2))
\end{equation*}
where $\Ch^{2}(X_{K}, \Lambda)$ is the Chow group of $X_{K}$ with coefficient in $\Lambda$. Let $\Ch^{2}(X_{K}, \Lambda)_{0}$ be the kernel of the composite map
\begin{equation*}
\Ch^{2}(X_{K}, \Lambda)\xrightarrow{\mathrm{cl}} \rmH^{4}(X_{K}, \Lambda(2))\rightarrow  \rmH^{4}(X_{K^{\ac}}, \Lambda(2)).
\end{equation*}
Then the cycle class map induces the following \emph{Abel-Jacobi map}
\begin{equation*}
\AJ_{\Lambda}: \Ch^{2}(X_{K}, \Lambda)_{0}\rightarrow \rmH^{1}(K, \rmH^{3}(X_{K^{\ac}}, \Lambda(2))).
\end{equation*}
Suppose that $z\in \Ch^{2}(X_{K}, \Lambda)_{0}$. Let $\calZ$ be the Zariski closure of $z$ in $\mathfrak{X}$ and $z^{*}$ be the cycle of codimension $2$ in $\mathfrak{X}$ supported on $\calZ$ and whose restriction to $X_{K}$ is $z$. Let $\tilde{z}\in \rmH^{4}(a_{0*}\Lambda(2))^{G_{k}} =\rmH^{4}(X^{(0)}_{k^{\ac}}, \Lambda(2))^{G_{k}}$ be the image of $z^{*}$ under the following composite maps
\begin{equation*}
\begin{aligned}
&\rmH^{4}_{\calZ}(\mathfrak{X}, \Lambda(2))\rightarrow \rmH^{4}(\mathfrak{X}, \Lambda(2))\rightarrow \rmH^{4}(X_{k}, \Lambda(2))
\rightarrow \rmH^{4}(X_{k^{\ac}}, \Lambda(2))^{G_{k}}\rightarrow \rmH^{4}(X^{(0)}_{k^{\ac}}, \Lambda(2))^{G_{k}}.\\
\end{aligned}
\end{equation*}
We have the following results of Liu \cite{Liu-cubic}.
\begin{proposition}\label{cal-aj}
The element $\tilde{z}$ belongs to $ A^{2}(X_{k}, \Lambda)^{0}$. Moreover  we have
\begin{equation*}
\partial_{p}\AJ_{\Lambda}(z)=\eta(\tilde{z})
\end{equation*} 
in $\rmH^{1}_{\sing}(K, \rmH^{3}(X_{k^{\ac}}, \Lambda(2)))$. 
\end{proposition}
\begin{proof}
The first statement follows easily from \cite[Lemma 2.17]{Liu-cubic} and the second is a special case of \cite[Theorem 2.18]{Liu-cubic}. 
\end{proof}
\section{Arithmetic level raising for Shimura curves} 
\subsection{Shimura curves and Shimura sets}
Let $N$ be a positive integer with a factorization $N=N^{+}N^{-}$ with $N^{+}$ and $N^{-}$ coprime to each other. We assume that $N^{-}$ is square-free and is a product of odd number of primes. Let $B$ be the definite quaternion algebra over $\QQ$ with discriminant $N^{-}$ and let $B^{\prime}$ be the indefinite quaternion algebra over $\QQ$ with discriminant $pN^{-}$. Let $\calO_{B^{\prime}}$ be a maximal order of $B^{\prime}$ and  let $\calO_{B^{\prime}, N^{+}}$ be the Eichler order of level $N^{+}$ in $\calO_{B^{\prime}}$. We let $G^{\prime}$ be the algebraic group over $\QQ$ given by $B^{\prime \times}$ and let $K^{\prime}_{N^{+}}$ be the open compact of $G^{\prime}(\mathbf{A}^{\infty})$ defined by $\hat{\calO}^{\times}_{B^{\prime}, N^{+}}$. We let $G$ be the algebraic group over $\QQ$ given by $B^{\times}$. Note that we have an isomorphism $G^{\prime}(\mathbf{A}^{\infty, (p)})\xrightarrow{\sim} G(\mathbf{A}^{\infty, (p)})$ and via this isomorphism we will view $K^{\prime p}$ as an open compact subgroup of $G(\mathbf{A}^{\infty, (p)})$ for any open compact subgroup $K^{\prime}$ of $G^{\prime}(\mathbf{A}^{\infty})$.

Let $X=X^{B^{\prime}}_{N^{+}, pN^{-}}$ be the Shimura curve over $\QQ$ with level $K^{\prime}=K^{\prime}_{N^{+}}$. The complex points of this curve is given by the following double coset
\begin{equation*}
X(\CC)=G^{\prime}(\QQ)\backslash \calH^{\pm} \times G^{\prime}(\mathbf{A}^{\infty})/K^{\prime}.
\end{equation*}
There is a natural model $\mathfrak{X}$ over $\ZZ_{(p)}$ of $X$. Recall it represents the following moduli problem. Let $S$ be a test scheme over $\ZZ_{(p)}$ and then $\mathfrak{X}(S)$ is the set of triples $(A, \iota, \bar{\eta})$ where
\begin{enumerate}
\item $A$ is an $S$-abelian scheme of relative dimension $2$;
\item $\iota: \calO_{B^{\prime}}\hookrightarrow \End_{S}(A)$ is an embedding which is special in the sense of \cite[131-132]{BC-unifor};
\item $\bar{\eta}$ is an equivalence class of isomorphisms 
\begin{equation*}
\eta: V^{p}(A)\xrightarrow{\sim} B^{\prime}(\mathbf{A}^{\infty,(p)})
\end{equation*}
up to multiplication by $K^{p}$, where $$V^{p}(A)=\prod_{q \neq p}T_{q}(A)\otimes \QQ$$ is the prime to $p$ rational Tate module of $A$. 
\end{enumerate}
It is well known this moduli problem is representable by a projective scheme over $\mathbb{Z}_{(p)}$ of dimension $1$. We will usually consider the base change of $\mathfrak{X}$ to $\ZZ_{p^{2}}$ and we will denote it by the same notation. Let $\FF_{q}$ be a finite extension of $\FF_{p}$, we will denote by $X_{\FF_{q}}$ the base change of the special fiber $X_{\FF_{p}}$ of $\mathfrak{X}$ to $\FF_{q}$.  

Let $K$ be the open compact subgroup of $G(\mathbf{A}^{\infty})$ given by the Eichler order $\calO_{B, N^{+}}$. We define the \emph{Shimura set} $X^{B}$ by the following double coset
\begin{equation}\label{Shi-set}
X^{B}= G(\QQ)\backslash G(\mathbf{A}^{\infty})/K.
\end{equation}
Let $K_{0}(p)$ be the open compact subgroup of $G(\mathbf{A}^{\infty})$ given by the Eichler order $\calO_{B, pN^{+}}$ of level $pN^{+}$ in $B$. We define $X^{B}_{0}(p)$ by the  double coset
\begin{equation}\label{Shi-set-p}
X^{B}_{0}(p)= G(\QQ)\backslash G(\mathbf{A}^{\infty})/K_{0}(p).
\end{equation}
Let $\Lambda$ be a ring. We will use the following notations throughout this article.
\begin{itemize}
\item We denote by $\rmH^{0}(X^{B}, \Lambda)$ the set of continuous function on $X^{B}$ valued in $\Lambda$. We will also write this space as $S^{B}_{2}(N^{+}, \Lambda)$ and refer to it as the space of \emph{quaternionic modular forms} of level $N^{+}$

\item In the same way, we define $\rmH^{0}(X^{B}_{0}(p), \Lambda)$.  We will also write this space as $S^{B}_{2}(pN^{+}, \Lambda)$ and refer to it as the space of \emph{quaternionic modular forms} of level $pN^{+}$

\item Let $W$ be a scheme we will let  $\PP^{d}(W)$ be the $\PP^{d}$-bundle over $W$. For example $\PP^{1}(X^{B})$ is the $\PP^{1}$-bundle over $X^{B}$.  
\end{itemize}

\subsection{The $p$-adic upper half plane} Let $k^{\ac}$ be an algebraic closure of $\FF_{p}$ and let $W_{0}=W(k^{\ac})$ be ring of Witt vectors of $k^{\ac}$ with fraction field $K_{0}$. Let $D$ be the quaternion algebra over $\QQ_{p}$ and let $\calO_{D}$ be the maximal order of $D$. We write 
\begin{equation*}
\calO_{D}=\ZZ_{p^{2}}[\Pi]/(\Pi^{2}-p)
\end{equation*}
with $\Pi a=\sigma(a)\Pi$ with $a\in \ZZ_{p^{2}}$. 
We first recall the notion of a \emph{special formal $\calO_{D}$-module}. Let $S$ be a $W_{0}$-scheme. A special formal $\calO_{D}$-module over $S$ is a formal $p$-divisible group $X$ of dimension $2$ and height $4$ with an $\calO_{D}$-action 
\begin{equation*}
\iota: \calO_{D}\rightarrow \End_{S}(X)
\end{equation*}
such that $\Lie(X)$ is a locally free $\ZZ_{p^{2}}\otimes \calO_{S}$-module of rank $1$. We fix a special $\calO_{D}$-module $\mathbb{X}$ over $k^{\ac}$ whose Dieudonn\'{e} module is denoted by $\mathbb{M}$. Consider the functor $\calM$ on the category Nilp of $W_{0}$-schemes $S$ such that $p$ is locally nilpotent in $\calO_{S}$ such that $\calM(S)$ classifies the isomorphism classes of pairs $(X, \rho_{X})$ where
\begin{enumerate}
\item $X$ is special formal $\calO_{D}$-module;
\item $\rho_{X}: X\times_{S}\bar{S}\rightarrow \mathbb{X}\times_{\FF} \bar{S}$ is a quasi-isogney. 
\end{enumerate}
The functor $\calM$ is represented by a formal scheme over $W_{0}$ which we also denote by $\calM$. The formal scheme $\calM$ breaks into a disjoint union
\begin{equation*}
\calM=\bigsqcup_{i\in\ZZ}\calM_{i}
\end{equation*}
according to the height $i$ of the quasi-isogeny $\rho_{X}$. Each formal scheme $\calM_{i}$ is isomorphic to the \emph{$p$-adic upper half plane} $\calH_{p}$ base-changed to $W_{0}$. The group $\GL_{2}(\QQ_{p})$ acts naturally on the formal scheme $\calM$ and each $\calM_{i}$ affords an action of the group $\GL^{0}_{2}(\QQ_{p}):=\{g\in\GL_{2}(\QQ_{p}): \ord_{p}(\det(g))=0\}$. We review  the descriptions of the special fiber of the formal scheme $\calM_{0}$.  Since this description is well-known, see  \cite{KR00} for example, we content on explaining this on the level of points. Let $(X, \rho)\in \calM(k^{\ac})$ and let $M$ be the covariant Dieudonn\'{e} of $X$ and the action of $\ZZ_{p^{2}}$ on $X$ induced a grading 
\begin{equation}
M=M_{0}\oplus M_{1}
\end{equation}
that satisifies
\begin{equation}
\begin{aligned}
& pM_{0}\subset^{1} VM_{1}\subset^{1} M_{0},\hphantom{a}pM_{1}\subset^{1} VM_{0}\subset^{1} M_{1}\\
& pM_{0}\subset^{1} \Pi M_{1}\subset^{1} M_{0},\hphantom{a}pM_{1}\subset^{1} \Pi M_{0}\subset^{1} M_{1}\\
\end{aligned}
\end{equation}
Since the action of $\Pi$ and $V$ commute, we have the induced maps 
\begin{equation}
\begin{aligned}
& \Pi: M_{0}/VM_{1}\rightarrow M_{1}/VM_{0},\\  
& \Pi: M_{1}/VM_{0}\rightarrow M_{0}/VM_{1}.\\
\end{aligned}
\end{equation}
Since both $M_{0}/VM_{1}$ and $M_{1}/VM_{0}$ are of dimension $1$ and the composite of the two maps is obviously zero, we can conclude that there is an $i\in \ZZ/2\ZZ$ such that $\Pi M_{i}\subset VM_{i}$. Since both $\Pi M_{i}$ and $VM_{i}$ are of colength $1$ in $M_{i+1}$, we conclude that they are in fact equal to each other. We say that $i$ is a \emph{critical index} if $VM_{i}=\Pi M_{i}$ and a critical index always exists for $M$. We let $\tau=\Pi^{-1}V$ and it acts as an automorphism on $M_{i}$ if $i$ is a critical index.

If $0$ is a critical index, then we set $L_{0}=M^{\tau=1}_{0}$ and we call it a \emph{vertex lattice of type $0$}. This is a $\ZZ_{p}$-lattice of rank $2$ and we associate to it the projective line $\PP(L_{0}/p)$. Then $VM_{1}/pM_{0}\subset^{1} M_{0}/pM_{0}=L_{0}/pL_{0}\otimes k^{\ac}$ gives a point on $\PP(L_{0}/pL_{0})(k^{\ac})$. If $1$ is a critical index, then we similarly put $L_{1}=\Pi M^{\tau=1}_{1}$ and we call it a \emph{vertetx lattice of type $1$}. We again associate to it the projective line $\PP(L_{1}/pL_{1})$. Similarly $VM_{0}/pM_{1}\subset^{1} M_{1}/pM_{1}=L_{1}/pL_{1}\otimes k^{\ac}$ gives a point on $\PP(L_{1}/pL_{1})(k^{\ac})$. This construction gives all the irreducible components of the special fiber of  $\calM_{0}$. If both $0$ and $1$ are critical indices, then we will identify the point on $\PP(L_{1}/pL_{1})$ given by $VM_{0}/pM_{1}\subset^{1} M_{1}/pM_{1}=L_{1}/pL_{1}\otimes k^{\ac}$ and the point on  $\PP(L_{0}/pL_{0})$ given by  $VM_{1}/pM_{0}\subset^{1} M_{0}/pM_{0}=L_{0}/pL_{0}\otimes k^{\ac}$. Notice in this case the vertex lattices $L_{0}, L_{1}$ satisfies the following inclusions $$pL_{0}\subset L_{1}\subset L_{0}.$$  We summarize the above discussion in the following proposition.

\begin{proposition}\label{Drinfeld} We have the following statements.
\begin{enumerate}
\item The irreducible components of the special fiber of $\calM_{0}$ can be partitioned into two types according to whether $0$ or $1$ is a critical index. Each irreducible component corresponds to a vertex lattice and is isomorphic to a projective line. 
\item Two irreducible components corresponding two vertex lattices of the same type will not intersect. Let $L_{0}$ be a vertex lattice of type $0$ and $L_{1}$ be a vertex lattice of type $1$, then the corresponding irreducible components intersect if an only if  $pL_{0}\subset L_{1}\subset L_{0}$. 

\item The irreducible components are parametrized by $\GL_{2}(\QQ_{p})/\GL_{2}(\ZZ_{p})\times \ZZ/2\ZZ$. The intersection points of the irreducible components are parametrized by $\GL_{2}(\QQ_{p})/\Iw_{p}$ where $\Iw_{p}$ is the  Iwahori subgroup of $\GL_{2}(\QQ_{p})$. 
\end{enumerate}
\end{proposition}
\begin{proof}
The first two points follow from the previous discussions. For the third point, note that the stabilizer of a vertex lattice in $\GL_{2}(\QQ_{p})$ is obviously $\GL_{2}(\ZZ_{p})$. The condition $pL_{0}\subset L_{1}\subset L_{0}$ defines a standard alcove in the Bruhat-Tits building of $\PGL_{2}(\QQ_{p})$ and therefore the stabilizer of the pair $(L_{0}, L_{1})$ is the  Iwahori subgroup $\Iw_{p}$. 
\end{proof}

\subsection{Cerednick-Drinfeld uniformization} From here on,  we will set $k=\FF_{p^{2}}$, $K=\QQ_{p^{2}}$ and $\calO_{K}=\ZZ_{p^{2}}$. Recall we have the integral model $\mathfrak{X}$ of the Shimura curve $X$ over $\calO_{K}$ and we denote by $\mathfrak{X}^{\wedge}$ its completion along the ideal defined by $p$. The Cerednick-Drinfeld uniformization theorem asserts that the ${\mathfrak{X}}^{\wedge}$ can be uniformized by the formal scheme $\calM$:
\begin{equation}\label{p-unifor}
{\mathfrak{X}}^{\wedge} \xrightarrow{\sim} G(\QQ)\backslash \calM \times G(\mathbf{A}^{\infty, (p)})/K^{p}. 
\end{equation}
It follows from the descriptions in Proposition \ref{Drinfeld} and the above uniformization theorem that the irreducible components of the special fiber $X_{k}$ are projective lines. It also follows that $\mathfrak{X}$ has strict semistable reduction. More precisely, we have the following proposition.
\begin{proposition}\label{curve-red}
We have the following descriptions of the scheme $X_{k}$.
\begin{enumerate}
\item The scheme $X_{k}$ is a union $\PP^{1}$-bundles over Shimura sets 
\begin{equation*}
X_{k}= \PP^{1}(X^{B}_{+})\cup \PP^{1}(X^{B}_{-}). 
\end{equation*}
where both $X^{B}_{+}$ and $X^{B}_{-}$ are isomorphic to the Shimura set $X^{B}$ as in \eqref{Shi-set}. 
\item The intersection points of the two $\PP^{1}$-bundles $\PP^{1}(X^{B}_{+})$ and $\PP^{1}(X^{B}_{-})$ are given by
\begin{equation*}
\PP^{1}(X^{B}_{+})\cap \PP^{1}(X^{B}_{-})=X^{B}_{0}(p).
\end{equation*}
This also can be identified with the set of singular points on $X_{k}$. 
\end{enumerate}
\end{proposition}
\begin{proof}
For $(1)$, recall that there are two types of irreducible components of the special fiber of $\calM_{0}$ corresponding to vertex lattices of type $0$ and vertex lattices of type $1$. By the Cerednick--Drinfeld uniformization \eqref{p-unifor}, we only need to notice that the irreducible components of corresponding to vertex lattices of type $0$ or vertex lattice of type $1$ are uniformized by
\begin{equation*}
G(\QQ)\backslash \GL_{2}(\QQ_{p})/\GL_{2}(\ZZ_{p})\times G(\mathbf{A}^{\infty, (p)})/K^{p}.
\end{equation*}
This is just another way of writing $X^{B}$. We will define $X^{B}_{+}$ to be the copy of $X^{B}$ corresponding to the vertex lattice of type $0$ and define $X^{B}_{-}$ to be the copy of $X^{B}$ corresponding to the vertex lattice of type $1$. For $(2)$, the Cerednick--Drinfeld uniformization \eqref{p-unifor} shows that the set of intersection points are given by
\begin{equation*}
G(\QQ)\backslash \GL_{2}(\QQ_{p})/\Iw_{p}\times G(\mathbf{A}^{\infty, (p)})/K^{p}.
\end{equation*}
Again this is another way of writing $X^{B}_{0}(p)$. The rest of the claims are clear.
\end{proof}

We denote by $\pi_{+}: X^{B}_{0}(p)\rightarrow X^{B}_{+}$ and $\pi_{-}: X^{B}_{0}(p)\rightarrow X^{B}_{-}$ the natural specialization maps. These two maps give the Hecke correspondence of $X^{B}$. This can be seen as follows: the map $\pi_{+}$ is induced by sending the pair of vertex lattices $(L_{1}\subset L_{0})$ to $L_{0}$ and the map $\pi_{-}$ is induced by sending $(L_{1}\subset L_{0})$ to $L_{1}$. The two lattices are related by the matrix 
$
\begin{pmatrix}
 0& p\\
 1 & 0\\
\end{pmatrix}
$
that is if one chooses a suitable basis $(e_{1}, e_{2})$ for $L_{0}$ then $(e_{2}, pe_{1})$ will be a basis for $L_{1}$. 

The nontrivial element in $\Gal(k/\FF_{p})$ acts on  $X_{k}$ and it permutes the two $\PP^{1}$-bundles $\PP^{1}(X^{B}_{+})$ and $\PP^{1}(X^{B}_{-})$. On the set $X^{B}_{0}(p)$ it acts by the classical \emph{Atkin-Lehner involution}. For this, see \cite[\S 1.7]{BD-Mumford} for example. 

\subsection{Ramified level raising on Shimura curves} Recall $N=N^{+}N^{-}$ with $N^{+}$ and $N^{-}$ defined previously. Let $f\in \rmS^{\new}_{2}(\Gamma_{0}(N))$ be a newform of weight $2$. Let $E=\QQ(f)$ be the Hecke field of $f$. Let $\lambda$ be a place of $E$ above $l$ and $E_{\lambda}$ be the completion of $E$ at $\lambda$. Let $\varpi$ be a uniformizer of $\calO=\calO_{E_{\lambda}}$ and we write $\calO_{n}=\calO/\varpi^{n}$ for $n\geq 1$. We let $\TT=\TT_{N^{+}, N^{-}}$, respectively  $\TT^{[p]}=\TT_{N^{+}, pN^{-}}$, be the $l$-adic Hecke algebra corresponding to the cusp forms of level $N=N^{+}N^{-}$, respectively of level $Np=N^{+}N^{-}p$, which is new at primes dividing $N^{-}$, respectively at primes dividing $pN^{-}$. Since $f$ is an eigenform, we have a morphism
$\phi_{f}: \TT\rightarrow \calO$
corresponding to the system of eigenvalues of $f$. More precisely, we have  $\phi_{f}(T_{v})=a_{v}(f)$ for $v\nmid N$  and  $\phi_{f}(U_{v})=a_{v}(f)$ for $v\mid N$.

\begin{definition}\label{n-adm}
Let $n\geq 1$ be an integer. We say that a prime $p$ is \emph{$n$-admissible} for $f$ if 
\begin{enumerate}
\item $p\nmid Nl$;
\item $l\nmid p^{2}-1 $;
\item $\varpi^{n}\mid p+1-\epsilon_{p}(f)a_{p}(f)$ for some $\epsilon_{p}(f)\in\{-1, 1\}$. 
\end{enumerate}
\end{definition}

Let $\phi_{f, n}: \TT\rightarrow \calO_{n}$ be reduction of the map $\phi_{f}: \TT\rightarrow \calO$ modulo $\varpi^{n}$. We denote by $I_{f, n}$  the kernel of this map and $\mathfrak{m}_{f}$ the unique maximal ideal of $\TT$ containing $I_{f, n}$. We will say $\mathfrak{m}_{f}$ is residually irreducible if the residue Galois representation $\bar{\rho}_{f}$ attached to $f$  is irreducible. The following result is known as the \emph{ramified arithmetic level raising for Shimura curves} and was first proved in \cite[Theorem 5.15, Corollary 5.18]{BD-Main}. The proof of the theorem below is inspired by two lectures given by Liang Xiao at the Morningside center \cite{Xiao} on the subject.

\begin{theorem}\label{level-raise-curve}
Let $p$ be an $n$-admissible prime for $f$. We assume that 
\begin{enumerate}[label=(\roman*)]
\item the maximal ideal $\mathfrak{m}_{f}$ is residually irreducible;
\item The module $S^{B}_{2}(N^{+}, \calO)_{\mathfrak{m}_{f}}$ is free of rank $1$ over $\TT_{\mathfrak{m}_{f}}$. 
\end{enumerate}
Then we have the following.
\begin{enumerate}
\item There exists a surjective homomorphism $\phi^{[p]}_{f, n}: \TT^{[p]}\rightarrow \calO_{n}$ such that $\phi^{[p]}_{f, n}$ agrees with $\phi_{f, n}$ at all Hecke operators away from $p$ and sends $U_{p}$ to $\epsilon_{p}(f)$. We will denote by $I^{[p]}_{f, n}$ the kernel of $\phi^{[p]}_{f, n}$. 
\item We have an isomorphism of $\calO_{n}$-modules of rank $1$
\begin{equation*}
S^{B}_{2}(N^{+}, \calO){/I_{f, n}}\xrightarrow{\cong}\rmH^{1}_{\sing}(K, \rmH^{1}(X_{K^{\ac}}, \calO(1)){/I^{[p]}_{f, n}}).
\end{equation*}
\end{enumerate}
\end{theorem}
\begin{proof}
In this proof we will set $\rmM=\rmH^{1}(X_{k^{\ac}}, \calO(1))$ and $\rmM_{n}=\rmH^{1}(X_{k^{\ac}}, \calO(1)){/I_{f, n}}$. Then we use the following formula as proved in \eqref{1-sing} and \eqref{1-potential}
\begin{equation}\label{sing-formula}
\rmH^{1}_{\sing}(K, \rmM)\cong \coker[\rmH^{0}(a_{0*}\calO)\xrightarrow{\rho} \rmH^{0}(a_{1*}\calO)\xrightarrow{\tau}  \rmH^{2}(a_{0*}\calO(1))]^{G_{k}}.
\end{equation}
Here 
$$\rmH^{0}(a_{0*}\calO)=\rmH^{0}(\PP^{1}(X^{B}_{+}), \calO)\oplus\rmH^{0}(\PP^{1}(X^{B}_{-}), \calO)$$
and we can identify it with the space
$S^{B}_{2}(N^{+},\calO)^{\oplus 2}$. Similarly under the Poincare duality, we can also identify  
\begin{equation*}
\rmH^{2}(a_{0*}\calO(1))=\rmH^{2}(\PP^{1}(X^{B}_{+}), \calO(1))\oplus\rmH^{2}(\PP^{1}(X^{B}_{-}), \calO(1))
\end{equation*}
with $S^{B}_{2}(N^{+}, \calO)^{\oplus 2}$. The space  $\rmH^{0}(a_{1*}\calO)$ can be identified with $S^{B}_{2}(pN^{+},\calO)$. Under these identifications, the composition appearing in \eqref{sing-formula}
\begin{equation*}
\rmH^{0}(a_{0*}\calO)\xrightarrow{\rho} \rmH^{0}(a_{1*}\calO) \xrightarrow{\tau} \rmH^{2}(a_{0*}\calO(1))
\end{equation*}
is given by the \emph{intersection matrix}
\begin{equation*}
\begin{pmatrix}
-(p+1) &T_{p}\\
T_{p} &-(p+1)\\
\end{pmatrix}
\end{equation*}
which we will also denote by $\nabla$. Since $p$ is $n$-admissible for $f$, the singular quotient of $\rmM_{n}$
\begin{equation*}
\rmH^{1}_{\sing}(K, \rmM_{n})\cong\coker[S^{B}_{2}(N^{+},\calO)^{\oplus 2}_{/I_{f, n}}\xrightarrow{\nabla}S^{B}_{2}(N^{+},\calO)^{\oplus 2}_{/I_{f, n})}]
\end{equation*}
is of rank one over $\calO_{n}$ and is isomorphic to $S^{B}_{2}(N^{+},\calO)_{/I_{f, n}}$. Note here the isomorphism between 
\begin{equation*}
\coker[S^{B}_{2}(N^{+},\calO)^{\oplus 2}_{/I_{f, n}}\xrightarrow{\nabla}S^{B}_{2}(N^{+},\calO)^{\oplus 2}_{/I_{f, n})}]
\end{equation*}
and $S^{B}_{2}(N^{+},\calO)^{\oplus 2}_{/I_{f, n}}$ is induced by the map
\begin{equation*}
S^{B}_{2}(N^{+},\calO)^{\oplus 2}_{/I_{f, n}}\rightarrow S^{B}_{2}(N^{+},\calO)_{/I_{f, n}}, \hphantom{aa} (x, y)\mapsto \frac{1}{2}(x+\epsilon_{p}(f)y)
\end{equation*}
This proves the second part of the theorem. 

By \cite[Theorem 5.8]{BD-Main} and \cite[\S3.5]{CH-2}, the natural $U_{p}$-action on 
\begin{equation*}
\rmH^{2}(a_{0*}\calO(1))\cong S^{B}_{2}(N^{+},\calO)^{\oplus 2}
\end{equation*}
is given by $(x, y)\mapsto (-py, x+T_{p}y)$.
We consider the automorphism $$\delta: S^{B}_{2}(N^{+},\calO)^{\oplus 2} \rightarrow S^{B}_{2}(N^{+},\calO)^{\oplus 2}$$ given by 
$(x, y)\mapsto (x+T_{p}y, y)$.
Then a quick calculation gives us that $\nabla\circ \delta=U^{2}_{p}-1$.  This means that the quotient 
\begin{equation*}
\frac{S^{B}_{2}(N^{+},\calO)^{\oplus 2}}{(I_{f, n},U^{2}_{p}-1)} \cong \coker[S^{B}_{2}(N^{+},\calO)^{\oplus 2}_{/I_{f, n}}\xrightarrow{\nabla}S^{B}_{2}(N^{+},\calO)^{\oplus 2}_{/I_{f, n})}]\end{equation*}
is of rank $1$. Since $p$ is $n$-admissible for $f$, we see immediately $U_{p}+\epsilon_{f}$ is invertible on $S^{B}_{2}(N^{+},\calO)^{\oplus 2}_{/I_{f, n}}$. Therefore we have
\begin{equation*}
\frac{S^{B}_{2}(N^{+},\calO)^{\oplus 2}}{(I_{f, n},U^{2}_{p}-1)} \cong \frac{S^{B}_{2}(N^{+},\calO)^{\oplus 2}}{(I_{f, n},U_{p}-\epsilon_{p}(f))}
\end{equation*}
and the latter quotient is of rank $1$ over $\calO_{n}$. Then the action of $\TT^{[p]}$ on this rank $1$ quotient gives the desired morphism 
$\phi^{[p]}_{f, n}: \TT^{[p]}\rightarrow \calO_{n}$.  

\end{proof}

We will now compare the ingredients in the above proof with those in \cite{BD-Main}.  Let $\mathcal{G}$ be the dual reduction graph of $X_{k}$. We denote by $\mathcal{V}(\mathcal{G})$ the set of vertices and $\mathcal{E}(\mathcal{G})$ the set of edges. Then we have the following identifications 
\begin{equation*}
\rmH^{0}(a_{0*}\Lambda)=\Lambda[\mathcal{V}(\mathcal{G})] 
\end{equation*}
and 
\begin{equation*}
\rmH^{0}(a_{1*}\Lambda)=\Lambda[\mathcal{E}(\mathcal{G})]
\end{equation*}
for $\Lambda$ equals to $\ZZ_{/l^{n}}$, $\ZZ_{l}$ or a finite extension of $\ZZ_{l}$.
Moreover under these identifications, the restriction map 
\begin{equation*}
\rmH^{0}(a_{0*}\Lambda)\xrightarrow{\rho} \rmH^{0}(a_{1*}\Lambda)
\end{equation*}
can be identified with 
\begin{equation*}
\Lambda[\mathcal{V}(\mathcal{G})]\xrightarrow{d^{*}=-s^{*}+t^{*}} \Lambda[\mathcal{E}(\mathcal{G})]. 
\end{equation*}
And the map  $$\rmH^{0}(a_{1*}\Lambda)\xrightarrow{\tau} \rmH^{2}(a_{0*}\Lambda(1))$$ can be identified with $$\Lambda[\mathcal{E}(\mathcal{G})]\xrightarrow{d_{*}=-s_{*}+t_{*}} \Lambda[\mathcal{V}(\mathcal{G})].$$ 
Here $s, t$ are the source and target maps 
\begin{equation*}
\mathcal{E}(\mathcal{G})\xrightarrow{s, t}\mathcal{V}(\mathcal{G})
\end{equation*}
defined in an obvious manner. Let $\mathcal{J}$ be the N\'{e}ron model of the Jacobian $\mathrm{Jac}(X_{K})$ of $X_{K}$ and let $\Phi$ be the group of connected components of special fiber $\mathcal{J}_{k}$. Denote by $\calX$, resp. $\calX^{\vee}$, the \emph{character group}, resp. the \emph{cocharacter group} of $\mathcal{J}_{k}$. The following proposition describes $\Phi$,  $\calX$ and $\calX^{\vee}$  in terms of the dual reduction graph $\mathcal{G}$.
\begin{proposition} \label{component-grp}
We have the following statements
\begin{enumerate}
\item There is a Hecke module isomorphism 
$$\ZZ[\mathcal{V}(\mathcal{G})]_{0}/d_{*}d^{*} \cong \Phi$$
where $\ZZ[\mathcal{V}(\mathcal{G})]_{0}$ be the image of $\ZZ[\mathcal{E}(\mathcal{G})]\xrightarrow{d_{*}}\ZZ[\mathcal{V}(\mathcal{G})]$. 

\item The kernel of the map $$\ZZ[\mathcal{V}(\mathcal{G})]\xrightarrow{d^{*}}\ZZ[\mathcal{E}(\mathcal{G})]$$ is Eisenstein. 

\item The cokernel of the map $$\ZZ[\mathcal{E}(\mathcal{G})]\xrightarrow{d_{*}}\ZZ[\mathcal{V}(\mathcal{G})]$$ is Eisenstein. 

\item There is an isomorphism of Hecke modules  $$\calX= \ker[\ZZ[\mathcal{E}(\mathcal{G})]\xrightarrow{d_{*}}\ZZ[\mathcal{V}(\mathcal{G})]].$$

\item There is an isomorphism of Hecke modules $$\calX^{\vee}= \coker[\ZZ[\mathcal{V}(\mathcal{G})]\xrightarrow{d^{*}}\ZZ[\mathcal{E}(\mathcal{G})]].$$

\end{enumerate}
\end{proposition}
\begin{proof}
The first statement follows from \cite[\S9.6, Theorem 1]{BLR1}. The second statement is the well known ``Ihara's lemma" for definite quaternion algebras in view of what we have discussed before. See \cite[Theorem 3.15]{Ri-100} for the version we need. 
\end{proof}
\begin{corollary}[{\cite[Corollary 5.18]{BD-Main}}]
Under the assumptions in Theorem \ref{level-raise-curve}, we have a more canonical isomorphism 
\begin{equation*}
\Phi_{\calO}/I^{[p]}_{f, n} \cong \rmH^{1}_{\sing}(K, \rmH^{1}(X_{K^{\ac}}, \calO(1))/I^{[p]}_{f, n}).
\end{equation*}
\end{corollary}
\begin{proof}
Recall the following formula coming out of the weight spectral sequence for Shimura curves
\begin{equation*}
\rmH^{1}_{\sing}(K, \rmH^{1}(X_{K^{\ac}}, \calO(1))/I^{[p]}_{f, n})\cong \coker[\rmH^{0}(a_{0*}\calO)/I^{[p]}_{f, n}\xrightarrow{\rho} \rmH^{0}(a_{1*}\calO)/I^{[p]}_{f, n}\xrightarrow{\tau}  \rmH^{2}(a_{0*}\calO(1))/I^{[p]}_{f, n}]^{G_{k}}.
\end{equation*}
From the discussion before and Proposition \ref{component-grp} $(1)$, we see that the right-hand side of the above equation can be identified with
$\Phi_{\calO}/I^{[p]}_{f,n}$.  
\end{proof}

\begin{lemma}\label{comp-char}
There is an isomorphism under the assumption in Theorem \ref{level-raise-curve}
\begin{equation*}
\calX^{\vee}_{\calO}/I^{[p]}_{f, n}\cong \Phi_{\calO}/I^{[p]}_{f, n}\cong S^{B}_{2}(N^{+}, \calO)/I_{f, n}.
\end{equation*}
\end{lemma}
\begin{proof}
We have an exact sequence 
\begin{equation*}
0\rightarrow \calX_{\calO} \rightarrow \calX^{\vee}_{\calO} \rightarrow \Phi_{\calO}\rightarrow 0
\end{equation*}
by \cite[5.11]{BD-Main} which induces
\begin{equation*}
\calX_{\calO}/I^{[p]}_{f, n} \rightarrow \calX^{\vee}_{\calO}/I^{[p]}_{f, n} \rightarrow \Phi_{\calO}/I^{[p]}_{f, n}\rightarrow 0.
\end{equation*}
Let $J$ be the Jacobian of the curve $X$. By the proof of \cite[Theorem 5.17]{BD-Main},  the Galois module $\mathrm{Ta}_{l}(J)/I^{[p]}_{f, n}$ is of rank $2$ over $\calO_{n}$ and hence $\calX^{\vee}_{\calO}/I^{[p]}_{f, n}$ is of rank $1$ over $\calO_{n}$ by the $p$-adic uniformization of $J$. The result follows as $\Phi_{\calO}/I^{[p]}_{f, n}\cong S^{B}_{2}(N^{+}, \calO)/I_{f, n}$ is also rank $1$ under our assumption by the previous discussions. 
\end{proof}

\section{Arithmetic level raising for triple product of Shimura curves}    

\subsection{Semistable model of triple product of Shimura curves} Recall in the last section we have defined the Shimura curve $X=X^{B^{\prime}}_{N^{+}, pN^{-}}$ with discriminant $pN^{-}$ over $\QQ$. Let $X^{3}$ be the threefold fiber product of $X$ over $\QQ$. Recall that we have the integral model $\mathfrak{X}$ of $X$ defined over $\calO_{K}$. Let $\mathfrak{X}^{3}$ be the threefold fiber product of $\mathfrak{X}$ over $\calO_{K}$.  First, we analyze the reduction of  $\mathfrak{X}^{3}$. We denote by $X^{3}_{k}$ the special fiber of  $\mathfrak{X}^{3}$. By Proposition \ref{curve-red}, we know each $X_{k}$ can be described as the union $\PP^{1}(X^{B}_{+})\cup \PP^{1}(X^{B}_{-})$ and therefore the special fiber $X^{3}_{k}$ can be described by the cube given below.
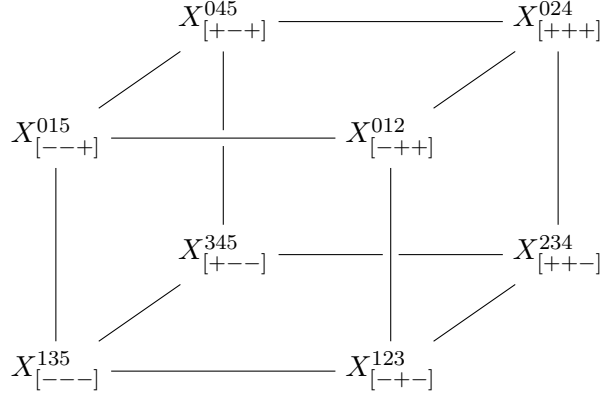
\begin{equation*}
\begin{tikzpicture}
  \matrix (m) [matrix of math nodes, row sep=2em,
    column sep=2em]{
     &  {X}^{045}_{[+-+]} & &  {X}^{024}_{[+++]}\\
         {X}^{015}_{[--+]}& &    {X}^{012}_{[-++]} & \\
    &  {X}^{345}_{[+--]}& &  {X}^{234}_{[++-]} \\
        {X}^{135}_{[---]} & &     {X}^{123}_{[-+-]} & \\}  ;
     \path
    (m-1-2) edge (m-1-4) edge (m-2-1) edge (m-3-2) 
    (m-1-4) edge (m-3-4) edge (m-2-3) 
    (m-2-1) edge [-,line width=6pt,draw=white] (m-2-3) edge (m-2-3) edge (m-4-1) 
    (m-3-2) edge  (m-3-4) edge (m-4-1) 
    (m-4-1) edge (m-4-3) 
    (m-3-4) edge (m-4-3) 
    (m-2-3) edge [-,line width=6pt,draw=white] (m-4-3) edge (m-4-3) ;
\end{tikzpicture}
\end{equation*}
We will explain the meaning of the simplices in this cube.
\begin{itemize}
\item The $0$-simplices are the vertices of the cube. They correspond to $3$-dimensional strata in $X^{3}_{k}$. Consider the vertex labeled by $X^{123}_{[-+-]}$ for example. The superscript $123$ has no real meaning and is simply used for ordering the vertices. This ordering is inherited from Liu's paper \cite{Liu-cubic} where the labels have real meanings in terms of achimedean places of a cubic field. The subscript $[-+-]$ means that $X^{123}_{[-+-]}$ is of the form
\begin{equation*}
\PP^{1}(X^{B}_{-})\times \PP^{1}(X^{B}_{+})\times \PP^{1}(X^{B}_{-}). 
\end{equation*}
\item The $1$-simplices are the edges of the cube. They correspond to $2$-dimensional strata in $X^{3}_{k}$. For example, we will label the edge between  ${X}^{135}_{[---]}$ and ${X}^{123}_{[-+-]}$ by $X^{1235}_{[-\pm-]}$. This means we will take the union on the superscript and take the intersection on the subscript. Then $X^{1235}_{[-\pm-]}$ is of the form
\begin{equation*}
\PP^{1}(X^{B}_{-})\times X^{B}_{0}(p) \times \PP^{1}(X^{B}_{-}).
\end{equation*}
\item The $2$-simplices are the faces of the cube. They correspond to $1$-dimensional strata in  $X^{3}_{k}$. We use similar convention as in the last point. For example
\begin{equation*}
X^{01235}_{[-\pm\pm]}=\PP^{1}(X^{B}_{-})\times X^{B}_{0}(p) \times X^{B}_{0}(p).
\end{equation*}
\item Finally, the $3$-simplex is the zero dimensional components given by 
\begin{equation*}
X^{012345}_{[\pm\pm\pm]}=X^{B}_{0}(p)\times X^{B}_{0}(p) \times X^{B}_{0}(p).
\end{equation*}
\end{itemize}
We will usually drop the subscript from the notations and put it back there when we need to know the exact form of the strata. By an easy computation on the local rings, we see that $\mathfrak{X}^{3}$ is not semi-stable. Following the procedure in \cite[Example 6.15]{GS95}, we can obtain a semistable model denoted by $\calY$  of $\mathfrak{X}^{3}$ over $\calO_{K}$. More precisely, to obtain $\calY$, we blow-up $\mathfrak{X}^{3}$ along the closed subscheme $X^{135}$, then we blow-up the strict transform of $X^{024}$. We denote by $\pi: \calY\rightarrow \mathfrak{X}^{3}$ the natural morphism between these two schemes given by the aforementioned process. The generic fiber $\calY_{K}$ agree with the generic fiber ${X}^{3}_{K}$ of $\mathfrak{X}^{3}$. The special fiber of $\calY$ will be denoted by $Y_{k}$ and it can be described by the following cube. The densely dotted line on the cube correspond to the new intersections between three dimensional strata caused by the blow-ups and they give new two dimensional strata. 
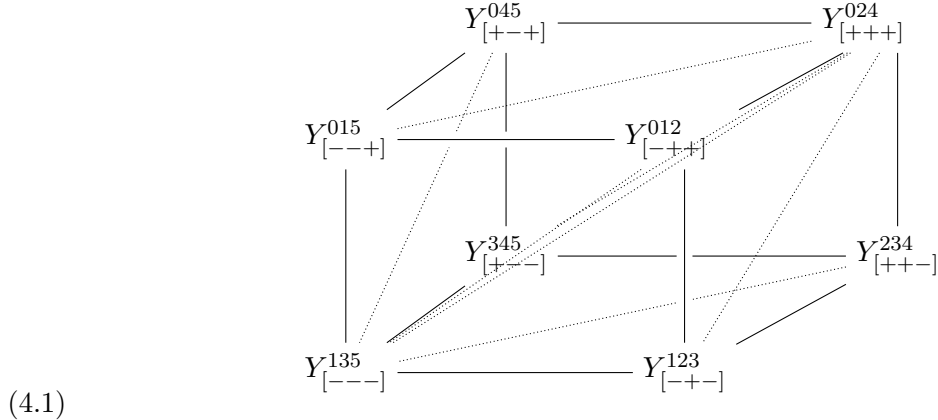
\begin{equation}\label{primitive-cube}
\begin{tikzpicture}
  \matrix (m) [matrix of math nodes, row sep=2em,
    column sep=2em]{
     &  {Y}^{045}_{[+-+]} \vphantom{f^{*}}  & &   {Y}^{024}_{[+++]}\hphantom{EEE}\\
      {Y}^{015}_{[--+]} \vphantom{f^{*}}& &         {Y}^{012}_{[-++]} \hphantom{Ee}& \\
    & {Y}^{345}_{[+--]} \vphantom{U}  & &       {Y}^{234}_{[++-]} \vphantom{E}\\
       {Y}^{135}_{[---]}  \vphantom{M} & &    {Y}^{123}_{[-+-]} \vphantom{N} & \\}  ;
     \path
    (m-1-2) edge (m-1-4) edge (m-2-1) edge (m-3-2) edge[densely dotted](m-4-1) 
    (m-1-4) edge (m-3-4) edge (m-2-3) edge[densely dotted] (m-3-2) edge[densely dotted](m-4-1)  edge[densely dotted] (m-2-1) edge[densely dotted] (m-4-3)
    (m-2-1) edge [-,line width=6pt,draw=white] (m-2-3) edge (m-2-3) edge (m-4-1) 
    (m-3-2) edge  (m-3-4) edge (m-4-1) 
    (m-4-1) edge (m-4-3) 
    (m-3-4) edge (m-4-3) edge[densely dotted](m-4-1)
    (m-2-3) edge [-,line width=6pt,draw=white] (m-4-3) edge (m-4-3) edge[densely dotted](m-4-1) ;
\end{tikzpicture}
\end{equation}
We will explain the meaning of some of the simplices of this cube. 
\begin{itemize}
\item The $0$-simplices are the vertices of the cube. They correspond to $3$-dimensional strata in ${Y}_{k}$. These are the strict transform of the corresponding strata in $X^{3}_{k}$. For example, $Y^{012}$ is the strict transform of $X^{012}$ under $\pi$. 

\item The $1$-simplices are the edges on the cube. They correspond to $2$-dimensional strata in ${Y}_{k}$. Notice that there are three types of edges: 
\begin{enumerate}
\item those correspond to the original edges in the cube for $X^{3}_{k}$, for example $Y^{0125}=Y^{012}\cap Y^{015}$; 
\item those correspond to the faces in the  cube for $X^{3}_{k}$, for example $Y^{01235}=Y^{012}\cap Y^{135}$; 
\item the one correspond to the ``main diagonal" of the cube $Y^{012345}=Y^{024}\cap Y^{135}$.
\end{enumerate}
\end{itemize}

\begin{proposition} \label{Yk}
We have the following descriptions of the $2$-dimensional and $3$-dimensional strata appear in the above list.
\begin{enumerate}
\item The three dimensional strata 
\begin{equation*}
Y^{i(i+1)(i+2)}
\end{equation*}
is the blow-up of $X^{i(i+1)(i+2)}$ along the one dimensional strata $X^{i(i+1)(i+2)(i+3)(i+5)}$ for $i\in \{0,1, 2, 3, 4, 5\}$.
\item The three dimensional strata 
\begin{equation*}
Y^{024}
\end{equation*}
is the blow-up of $X^{024}$ along the zero dimensional strata $X^{012345}$ followed by the blow-up of the strict transform of $X^{01234}\cup X^{01245}\cup X^{02345}$.

\item The three dimensional strata 
\begin{equation*}
Y^{135}
\end{equation*}
is the blow-up of $X^{135}$ along the zero dimensional strata $X^{012345}$ followed by the blow-up of the strict transform of $X^{01235}\cup X^{01345}\cup X^{12345}$.

\item The two dimensional strata 
\begin{equation*}
Y^{i(i+1)(i+2)(i+3)}
\end{equation*}
maps isomorphically to $X^{i(i+1)(i+2)(i+3)}$ for $i\in\{0, 1, 2, 3, 4, 5\}$.

\item The two dimensional strata 
\begin{equation*}
Y^{i(i+1)(i+2)(i+4)}
\end{equation*}
is the blow-up of $X^{i(i+1)(i+2)(i+4)}$ along $X^{012345}$ for $i\in\{0, 1, 2, 3, 4, 5\}$. 

\item The two dimensional strata 
\begin{equation*}
Y^{i(i+1)(i+2)(i+3)(i+5)}
\end{equation*}
is a $\PP^{1}$-bundle over $X^{i(i+1)(i+2)(i+3)(i+5)}$ for $i\in\{0, 1, 2, 3, 4, 5\}$. In fact, $Y^{i(i+1)(i+2)(i+3)(i+5)}$ is the exceptional divisor of the blow-up 
$\pi: Y^{i(i+1)(i+2)}\rightarrow X^{i(i+1)(i+2)}$.

\item The two dimensional strata 
\begin{equation*}
Y^{012345}
\end{equation*}
is a $\PP^{2}$-bundle over $X^{012345}$.
\end{enumerate}
Moreover all the strata of dimension $2$ or dimension $3$ are given in the above list. 
\end{proposition}

\begin{proof}
The proof of these statements are exactly the same as these given in \cite[Proposition B.39]{Liu-cubic}. Although we are working with different Shimura varities, the underlying local models are the same. 
\end{proof}

\subsection{Cohomology of blow-ups} We insert here a quick discussion of how to compute cohomology of varieties obtained by blow-ups, following \cite[\S 3]{Ito-p-uni}. Let $F$ be an algebraically closed field of any characteristic and let $X$ be a projective smooth irreducible variety over $F$ of dimension $n$. Let $Y_{1}, \cdots, Y_{r}\subset X$ be mutually disjoint smooth closed irreducible subvarieties of codimension $d\geq 2$. Let $i: Y=\sqcup^{r}_{t=1} Y_{t} \hookrightarrow X$ be the closed immersion. Let $\pi: \tilde{X}\rightarrow X$ be the blow-up of $X$ along $Y$ and $\tilde{Y}$ be the strict transform of $Y$. These maps will fit in the following cartesian diagram 
\begin{equation}
\begin{tikzcd}
\tilde{Y} \arrow[r, "\tilde{i}"] \arrow[d, "\pi_{\mid Y}"] & \tilde{X} \arrow[d, "\pi"] \\
Y \arrow[r, "i"]                & X    .          
\end{tikzcd}
\end{equation}
Since $Y$ is a disjoint union of mutually disjoint smooth irreducible closed subvarieties of codimension $d$, $\tilde{Y}$ is a $\PP^{d-1}$-bundle over $Y$. 
Let $\xi=c_{1}(\calO_{\tilde{Y}}(1))\in \rmH^{2}(\tilde{Y}, \Lambda(1))$ be the first Chern class of $\calO_{\tilde{Y}}(1)$ of the $\PP^{d-1}$-bundle $\tilde{Y}$. \begin{proposition}\label{bl-coho}
We have the following fomrula
\begin{equation*}
\begin{aligned}
\rmH^{k}(\tilde{X}, \Lambda(n))&\cong \rmH^{k-2}(Y, \Lambda(n-1))\oplus \cdots \oplus  \rmH^{k-2-2(d-2)}(Y, \Lambda(n-(d-2)))\xi^{d-2}
\oplus \rmH^{k}(X, \Lambda(n)). \\
\end{aligned}
\end{equation*}
\end{proposition}
\begin{proof}
This is a quite well-known results in \cite{sga5} and one can find a nice exposition of this result in \cite[(3.3)]{Ito-p-uni}. 
\end{proof}

\subsection{Level raising on triple product of Shimura curves } As in \S 2.1, $N$ is a positive integer which admits a factorization $N=N^{+}N^{-}$ with $(N^{+}, N^{-})=1$ such that $N^{-}$ is square free and has odd number of prime factors. Let
\begin{equation*}
\begin{aligned}
&f_{1}=\sum_{n\geq 1} a_{n}(f_{1})q^{n},\\
&f_{2}=\sum_{n\geq 1} a_{n}(f_{2})q^{n},\\
&f_{3}=\sum_{n\geq 1} a_{n}(f_{3})q^{n} \\
\end{aligned}
\end{equation*}
be a triple of weight $2$ normalized newforms in $ S^{\new}_{2}(\Gamma_{0}(N))$. We write $\triplef=(f_{1}, f_{2}, f_{3})$.
 For each $i=1, 2, 3$, let $E_{i}=\QQ(f_{i})$ be the Hecke field of $f_{i}$. Let $\calO_{i}$ be a finite extension of $\Lambda$ containing the ring of integers of $E_{i}$ with uniformizer $\varpi_{i}$. We will write $\FF_{i}$ as the residue field of $\calO_{i}$. Let $p$ be a prime away from $N$. Recall that $\TT=\TT_{N^{+}, N^{-}}$, respectively  $\TT^{[p]}=\TT_{N^{+}, N^{-}p}$, is the $l$-adic  Hecke algebra corresponding to the subspace of the cusp forms of level $N=N^{+}N^{-}$, respectively of level $Np=N^{+}N^{-}p$, which are new at primes dividing $N^{-}$, respectively at $pN^{-}$. Since $(f_{1}, f_{2}, f_{3})$ is a triple of eigenforms, we have morphisms 
$\phi_{i}: \TT\rightarrow \calO_{i}$ corresponding to $f_{i}$ and $\phi_{i, n}: \TT\rightarrow \calO_{i, n}$ corresponding to the reduction of $\phi_{i}$ modulo $\varpi^{n}_{i}$ for $i=1, 2, 3$. 
\begin{definition}
Let $n\geq 1$ be an integer. We say that a prime $p$ is \emph{$n$-admissible} for $\triplef=(f_{1}, f_{2}, f_{3})$ if 
\begin{enumerate}
\item $p\nmid Nl$;
\item $l\nmid p^{2}-1$;
\item $\varpi_{i}^{n}\mid p+1-\epsilon_{p,i}a_{p}(f_{i})$ with $\epsilon_{p,i}=\pm1$ for $i=1, 2, 3$;
\item $\epsilon_{p,1}\epsilon_{p,2}\epsilon_{p,3}=1$. 
\end{enumerate}
\end{definition}

Let $p$ be an \emph{$n$-admissible} prime for $\triplef$. We know  by Theorem \ref{level-raise-curve} that there are morphisms 
$\phi^{[p]}_{i, n}: \TT^{[p]}\rightarrow \calO_{i, n}$ that agree with $\phi_{i, n}: \TT\rightarrow \calO_{i, n}$
at all Hecke operators except those at $p$ and such that $\phi^{[p]}_{i, n}(U_{p})=\epsilon_{p, i}$ for $i=1,2,3$. We denote by $I_{{i}, n}$, respectively $I^{[p]}_{{i}, n}$, the kernel of the map $\phi_{i, n}$, respectively the kernel of the map $\phi^{[p]}_{i, n}$.  We also let $\mathfrak{m}_{i}$ be the maximal ideal of $\TT$ containing $I_{{i}, n}$ and let $\mathfrak{m}^{[p]}_{i}$ be the maximal ideal of $\TT^{[p]}$ containing $I^{[p]}_{{i}, n}$. We will always assume that the maximal ideal $\mathfrak{m}_{i}$ is residually irreducible. Let $\frakm_{\triplef}=(\frakm_{1}, \frakm_{2}, \frakm_{3})$ and $\frakm^{[p]}_{\triplef}=(\frakm^{[p]}_{1}, \frakm^{[p]}_{2}, \frakm^{[p]}_{3})$.

There  is an action of $\TT^{[p]}\times \TT^{[p]}\times  \TT^{[p]}$ on the $l$-adic cohomology of $\mathfrak{X}^{3}$ which extends to that of $\mathcal{Y}$. Let $\Lambda=\ZZ_{l}$ from now on. By the  K\"{u}nneth formula it makes sense to localize the cohomology 
\begin{equation*}
\rmH^{3}(\calY\otimes{K^{\ac}}, \Lambda(2))=\rmH^{3}(\mathfrak{X}^{3}\otimes{K^{\ac}}, \Lambda(2))
\end{equation*}
at the triple $\frakm_{\triplef^{[p]}}$. 
\begin{equation}\label{Kunneth}
\rmH^{3}(\mathfrak{X}^{3}\otimes{K^{\ac}}, \Lambda(2))_{\frakm_{\triplef^{[p]}}}=\otimes^{3}_{i=1}\rmH^{1}(X_{K^{\ac}}, \Lambda(1))_{\frakm^{[p]}_{i}}(-1). 
\end{equation}
This follows from the fact that the $\rmH^{0}$ and $\rmH^{2}$ of the Shimura curve $X_{K^{\ac}}$ are Eisenstein as Hecke modules and the K\"{u}nneth formula.  The same reasoning implies that $\rmH^{*}(\mathfrak{X}^{3}\otimes{K^{\ac}}, \Lambda(r))_{\frakm^{[p]}_{\triplef}}$ is non-zero only when $*=3$ for any integer $r$.  To analyze $\rmH^{3}(\mathfrak{X}^{3}\otimes{K^{\ac}}, \Lambda(2))$, we will localize its weight spectral sequence whose first page is given by \eqref{E1-primitive} at the triple $\frakm_{\triplef^{[p]}}$ and write the resulting spectral sequence by $\mathbb{E}(2)$. The ``untwisted" spectral sequence will be denoted by $\mathbb{E}$. 
\begin{lemma}
The $\rmE_{1}$-page of the spectral sequence $\mathbb{E}(2)$ is given below
\begin{equation*}\label{E1}
\begin{tikzpicture}[thick,scale=0.9, every node/.style={scale=0.6}]
\matrix (m) [matrix of math nodes,
    nodes in empty cells,nodes={minimum width=5ex,
    minimum height=5ex,outer sep=-5pt},
    column sep=1ex,row sep=1ex]{
              &      &     &     & \\
          6    &\rmH^{0}(a_{3*}\Lambda(-1))_{*} &  \rmH^{2}(a_{2*}\Lambda)_{*}&\rmH^{4}(a_{1*}\Lambda(1))_{*} & \rmH^{6}(a_{0*}\Lambda(2))_{*} \\
          4    &    &\rmH^{0}(a_{2*}\Lambda)_{*}  &\rmH^{2}(a_{1*}\Lambda(1))_{*}\oplus \rmH^{0}(a_{3*}\Lambda)_{*} &\rmH^{4}(a_{0*}\Lambda(2))_{*}\oplus\rmH^{2}(a_{2*}\Lambda(1))_{*}&\rmH^{4}(a_{1*}\Lambda(2))_{*}\\
          2    &       &  &\rmH^{0}(a_{1*}\Lambda(1))_{*} &\rmH^{2}(a_{0*}\Lambda(2))_{*}\oplus\rmH^{0}(a_{2*}\Lambda(1))_{*} &\rmH^{2}(a_{1*}\Lambda(2))_{*}\oplus\rmH^{0}(a_{3*}\Lambda(1))_{*}&\rmH^{2}(a_{2*}\Lambda(2))_{*}\\
          0   &        &                       &  &\rmH^{0}(a_{0*}\Lambda(2))_{*} &\rmH^{0}(a_{1*}\Lambda(2))_{*} & \rmH^{2}(a_{2*}\Lambda(2))&\rmH^{0}(a_{3*}\Lambda(2))_{*}&\\
         \quad\strut & -3  &  -2  &  -1  & 0 &1 &2 &3 & \strut \\};
\draw[thick] (m-1-1.east) -- (m-6-1.east) ;
\draw[thick] (m-6-1.north) -- (m-6-9.north) ;
\end{tikzpicture} 
\end{equation*}
where the subscript ${*}$ means localization at $\frakm_{\triplef^{[p]}}$. In particular it follows that $\mathbb{E}(2)$ will degenerate on its $\rmE_{2}$-page.
\end{lemma}
\begin{proof}
This follows from Proposition \ref{bl-coho} and the explicit descriptions of $X^{3}_{{k}}$ in \eqref{primitive-cube} which allow us to remove all the odd degree cohomology terms. 
  \end{proof}
  
\begin{lemma}
The spectral sequence $\mathbb{E}(2)$ satisfies Assumption \ref{assump-E}.
\end{lemma}
\begin{proof}
We need to analyze the cohomology term $\rmH^{2}(a_{1*}\Lambda(1))$. First we have 
\begin{equation}\label{a1}
\begin{aligned}
\rmH^{2}(a_{1*}\Lambda(1))= &\oplus^{5}_{i=0}\rmH^{2}(Y^{i(i+1)(i+2)(i+3)}_{k^{\ac}}, \Lambda(1))\\
                                                                       &\oplus\oplus^{5}_{i=0}\rmH^{2}(Y^{i(i+1)(i+2)(i+4)}_{k^{\ac}}, \Lambda(1))\\
                                                                       &\oplus\oplus^{5}_{i=0} \rmH^{2}(Y^{i(i+1)(i+2)(i+3)(i+5)}_{k^{\ac}}, \Lambda(1))\\
                                                                       &\oplus\rmH^{2}(Y^{012345}_{k^{\ac}}, \Lambda(1))\\
\end{aligned}
\end{equation}
by Proposition \ref{Yk}. 
We explicate the terms in the direct sum on the righthand side of \eqref{a1}. 
\begin{itemize}
\item For $\rmH^{2}(Y^{i(i+1)(i+2)(i+3)}_{k^{\ac}}, \Lambda(1))$, we only make it explicit for $i=0$. By Proposition \ref{Yk} $(4)$, we have
\begin{equation*}
\begin{aligned}
\rmH^{2}(Y^{0123}_{k^{\ac}}, \Lambda(1))&= \rmH^{2}(X^{0123}_{[-+\pm], k^{\ac}}, \Lambda(1))\\
                                                                                  &=\rmH^{2}(\PP^{1}(X^{B}_{-}), \Lambda(1))\otimes \rmH^{0}(\PP^{1}(X^{B}_{+}), \Lambda)\otimes \rmH^{0}(X^{B}_{0}(p), \Lambda)\\
                                                                                  &\oplus\rmH^{0}(\PP^{1}(X^{B}_{-}), \Lambda)\otimes \rmH^{2}(\PP^{1}(X^{B}_{+}), \Lambda(1))\otimes \rmH^{0}(X^{B}_{0}(p), \Lambda).\\
\end{aligned}
\end{equation*}

\item For $\rmH^{2}(Y^{i(i+1)(i+2)(i+4)}_{k^{\ac}}, \Lambda(1))$, we only make it explicit for $i=0$. By Proposition \ref{Yk} $(5)$, we have
\begin{equation*}
\begin{aligned}
\rmH^{2}(Y^{0124}_{k^{\ac}}, \Lambda(1))&= \rmH^{2}(X^{0124}_{[\pm++], k^{\ac}}, \Lambda(1))\oplus \rmH^{0}(X^{012345}_{k^{\ac}}, \Lambda)\\
                                                                                  &=\rmH^{0}(X^{B}_{0}(p)), \Lambda)\otimes \rmH^{2}(\PP^{1}(X^{B}_{+}), \Lambda(1))\otimes \rmH^{0}(\PP^{1}(X^{B}_{+}), \Lambda)\\
                                                                                   &\oplus\rmH^{0}(X^{B}_{0}(p)), \Lambda)\otimes \rmH^{0}(\PP^{1}(X^{B}_{+}), \Lambda)\otimes \rmH^{2}(\PP^{1}(X^{B}_{+}), \Lambda(1))\\
                                                                                  & \oplus \rmH^{0}(X^{012345}, \Lambda).\\
\end{aligned}
\end{equation*}

\item For $\rmH^{2}(Y^{i(i+1)(i+2)(i+3)(i+5)}_{k^{\ac}}, \Lambda(1))$, we again only make it explicit for $i=0$. By Proposition \ref{Yk} $(6)$, we have
\begin{equation*}
\begin{aligned}
\rmH^{2}(Y^{01235}_{k^{\ac}}, \Lambda(1))&= \rmH^{2}(\PP^{1}(X^{01235}_{[-\pm\pm], k^{\ac}}), \Lambda(1))\\ 
                                                                    &= \rmH^{0}(X^{01235}_{[-\pm\pm], k^{\ac}}), \Lambda) \\
                                                                    &=\rmH^{0}(\PP^{1}(X^{B}_{-}), \Lambda)\otimes \rmH^{0}(X^{B}_{0}(p), \Lambda)\otimes \rmH^{0}(X^{B}_{0}(p), \Lambda).\\
\end{aligned}
\end{equation*}

\item By Proposition \ref{Yk} $(7)$, we have
\begin{equation*}
\begin{aligned}
\rmH^{2}(Y^{012345}_{k^{\ac}}, \Lambda(1))&=\rmH^{0}(X^{012345}_{k^{\ac}}, \Lambda)\\
&=\rmH^{0}(X^{B}_{0}(p), \Lambda)\otimes\rmH^{0}(X^{B}_{0}(p), \Lambda)\otimes\rmH^{0}(X^{B}_{0}(p), \Lambda)
\end{aligned}
\end{equation*}
\end{itemize}
It follows from the above calculation that $G_{k}$ acts trivially on $\rmH^{2}(a_{1*}\Lambda(1))$. Using this we can immediately verify that Assumption \ref{assump-E} is satisfied for the spectral sequence $\mathbb{E}$. 
\end{proof}

The above lemma allows us to apply the exact sequence in \eqref{sing-exact} to calculate 
\begin{equation*}
\rmH^{1}_{\sing}(K, \rmH^{3}(X^{3}\otimes{K^{\ac}}, \Lambda(2))_{\frakm_{\triplef^{[p]}}}).
\end{equation*}
For $i=1, 2, 3$, we will set 
\begin{equation*}
\rmT^{[p]}_{i}=\rmH^{1}(X_{\QQ^{\ac}}, \calO_{i}(1))_{\mathfrak{m}^{[p]}_{i}}
\end{equation*}
and let $\rmT^{[p]}_{i, n}= \rmT^{[p]}_{i}/I^{[p]}_{i, n}$ be the natural quotient. 
Let 
\begin{equation*}
\rmM^{[p]}(\triplef)=\rmT^{[p]}_{1}\otimes \rmT^{[p]}_{2}\otimes \rmT^{[p]}_{3}
\end{equation*}
which is a $G_{\QQ}$-module over $\calO=\calO_{1}\otimes\calO_{2}\otimes\calO_{3}$ and let 
\begin{equation*}
\rmM^{[p]}_{n}(\triplef)=\rmT^{[p]}_{1, n}\otimes \rmT^{[p]}_{2, n}\otimes \rmT^{[p]}_{3, n}
\end{equation*} 
which is a $G_{\QQ}$-module over $\calO_{n}=\calO_{1, n}\otimes \calO_{2, n}\otimes \calO_{3, n}$. 

\begin{lemma}\label{3-fin}
Let $p$ be an $n$-admissible prime for $\triplef$. The module $\rmM^{[p]}_{n}(\triplef)$ is unramified viewed as an $\calO_{n}[G_{K}]$-module and is isomorphic to 
\begin{equation*}
\calO_{n}\oplus\calO^{\oplus 3}_{n}(1)\oplus\calO^{\oplus 3}_{n}(2) \oplus \calO_{n}(3).
\end{equation*}
The singular quotient $\rmH^{1}_{\sing}(K, \rmM^{[p]}_{n}(\triplef)(-1))$ is free of rank $3$ over $\calO_{n}$
\end{lemma}
\begin{proof}
The first part follows from the fact that $p$ is an $n$-admissible prime for $\triplef$. In particular, $(3)$ in the definition of $n$-admissible prime for $\triplef$ implies that $\rmT^{[p]}_{i, n}\cong \calO_{i, n}\oplus\calO_{i, n}(1)$ as an $\calO_{K}$-module. See also Lemma \ref{p-lower}. For the second part, we have 
\begin{equation*}
\begin{aligned}
\rmH^{1}_{\sing}(K, \rmM^{[p]}_{n}(\triplef)(-1)) &\cong\rmH^{1}(I_{K}, \rmM^{[p]}_{n}(\triplef)(-1))^{G_{k}} \\
                                                      &\cong\mathrm{Hom}(I_{K}, \rmM^{[p]}_{n}(\triplef)(-1))^{G_{k}} \\
                                                      &\cong\mathrm{Hom}(\Lambda(1), \rmM^{[p]}_{n}(\triplef)(-1))^{G_{k}} .\\
\end{aligned}
\end{equation*}
Then the result follows from the first part. 
\end{proof}
The following theorem is the \emph{ramified arithmetic level raising for the triple product of Shimura curves} and is one of the main results in this article. Recall that we denote by $\Phi$ the group of connected components of the special fiber of the N\'{e}ron model of the Jacobian of the Shimura curve $X_{K}$.
\begin{theorem}\label{arithmetic-level-raising}
Let  $p$ be an $n$-admissible prime for the triple $\triplef=(f_{1}, f_{2}, f_{3})$.  For $i=1, 2, 3$, assume that 
\begin{enumerate}
\item the maximal ideal $\frakm_{i}$ is residually irreducible;
\item each $S^{B}_{2}(N^{+}, \calO_{i})_{\frakm_{i}}$ is free of rank $1$ over $\TT_{\frakm_{i}}$.
\end{enumerate}
Then we have an isomorphism 
\begin{equation*}
 \rmH^{1}_{\sing}(K, \rmM^{[p]}_{n}(\triplef)(-1))\cong \oplus^{3}_{j=1}(\otimes^{3}_{i=1}S^{B}_{2}(N^{+}, \calO_{i}){/I_{i,n}}).
\end{equation*}
More canonically, we have the isomorphism
\begin{equation*}
\rmH^{1}_{\sing}(K, \rmM^{[p]}_{n}(\triplef)(-1))  \cong \oplus^{3}_{j=1}(\otimes^{3}_{i=1}\Phi_{\calO_{i}}/I^{[p]}_{i,n}).
\end{equation*} 
\end{theorem}

\subsection{Proof of the arithmetic level raising} To prove the preceding theorem, we need a different presentation of the potential map. Recall that we have the following exact sequence
\begin{equation*}
A_{2}(Y_{k}, \Lambda)^{0}_{\frakm^{[p]}_{\triplef}}\xrightarrow{\nabla}A^{2}(Y_{k}, \Lambda)^{0}_{\frakm^{[p]}_{\triplef}}\xrightarrow{\eta} \rmH^{1}_{\sing}(K, \rmH^{3}(X^{3}\otimes{K^{\ac}}, \Lambda(2))_{\frakm^{[p]}_{\triplef}})\rightarrow 0
\end{equation*}
where 
\begin{equation*}
\begin{aligned}
&A^{2}(Y_{k}, \Lambda)^{0}_{\frakm^{[p]}_{\triplef}}=\im[\rmH^{2}(a_{1*}\Lambda(1))_{\frakm^{[p]}_{\triplef}}\xrightarrow{\tau}\rmH^{4}(a_{0*}\Lambda(2))_{\frakm^{[p]}_{\triplef}}]^{G_{k}}\\
&A_{2}(Y_{k}, \Lambda)^{0}_{\frakm^{[p]}_{\triplef}}=\im[\rmH^{2}(a_{0*}\Lambda(1))_{\frakm^{[p]}_{\triplef}}\xrightarrow{\rho}\rmH^{2}(a_{1*}\Lambda(1))_{\frakm^{[p]}_{\triplef}}]^{G_{k}}.\\
\end{aligned}
\end{equation*}

\begin{lemma}
We have the following statements.
\begin{enumerate}
\item $A_{2}(Y_{k}, \Lambda)^{0}_{\frakm^{[p]}_{\triplef}}\cong\coker[\rmH^{0}(a_{1*}\Lambda)_{\frakm^{[p]}_{\triplef}}\xrightarrow{\tau}\rmH^{2}(a_{0*}\Lambda(1))_{\frakm^{[p]}_{\triplef}}]^{G_{k}}$.
\item $A^{2}(Y_{k}, \Lambda)^{0}_{\frakm^{[p]}_{\triplef}}\cong\ker[\rmH^{4}(a_{0*}\Lambda(2))_{\frakm^{[p]}_{\triplef}}\xrightarrow{\rho}\rmH^{4}(a_{1*}\Lambda(2))_{\frakm^{[p]}_{\triplef}}]^{G_{k}}$.
\item The map $\nabla$ is induced by the composite $\tau\circ\rho$ under the above isomorphisms.
\end{enumerate}
\end{lemma}
\begin{proof}
By the previous reasoning, $\rmH^{2}(X^{3}\otimes{K^{\ac}}, \Lambda(1))$ is zero as it is Eisenstein and therefore $\mathbb{E}^{0, 2}_{2}(1)$ is zero. It follows then that 
\begin{equation*}
\ker[\rmH^{2}(a_{0*}\Lambda(1))_{\frakm^{[p]}_{\triplef}}\xrightarrow{\rho}\rmH^{2}(a_{1*}\Lambda(1))_{\frakm^{[p]}_{\triplef}}]\cong \im[\rmH^{0}(a_{1*}\Lambda)_{\frakm^{[p]}_{\triplef}}\xrightarrow{\tau}\rmH^{2}(a_{0*}\Lambda(1))_{\frakm^{[p]}_{\triplef}}].
\end{equation*}
Then $(1)$ is clear. The proof of $(2)$ is completely the same using the fact that $\rmH^{4}(X^{3}\otimes{K^{\ac}}, \Lambda(2))$ and thus $\mathbb{E}^{0, 4}_{2}(2)$ vanish this time. The rest of the claims follow by constructions.
\end{proof}

\begin{lemma}\label{feature-cycle}
We have the following isomorphisms
\begin{enumerate}
\item $A_{2}(Y_{k}, \Lambda)^{0}_{\frakm^{[p]}_{\triplef}}\cong\oplus^{5}_{i=0}\rmH^{0}(X^{i(i+1)(i+2)(i+3)(i+5)}_{k}, \Lambda)_{\frakm^{[p]}_{\triplef}}\oplus \rmH^{0}(X^{012345}_{k}, \Lambda)_{\frakm^{[p]}_{\triplef}}$;
\item  $A^{2}(Y_{k}, \Lambda)^{0}_{\frakm^{[p]}_{\triplef}}\cong\oplus^{5}_{i=0}\rmH^{2}(X^{i(i+1)(i+2)(i+3)(i+5)}_{k}, \Lambda(1))_{\frakm^{[p]}_{\triplef}}\oplus \rmH^{0}(X^{012345}_{k}, \Lambda)_{\frakm^{[p]}_{\triplef}}$.
\end{enumerate}
\end{lemma}
\begin{proof}
The second part follows from the first part by duality and we only prove the first part.  We have by Proposition \ref{Yk}
\begin{equation*}
\rmH^{2}(a_{0*}\Lambda(1))_{\frakm^{[p]}_{\triplef}}=\oplus^{5}_{i=0}\rmH^{2}(Y^{i(i+1)(i+2)}_{k^{\ac}}, \Lambda(1))_{\frakm^{[p]}_{\triplef}}\oplus \rmH^{2}(Y^{024}_{k^{\ac}}, \Lambda(1))_{\frakm^{[p]}_{\triplef}}\oplus \rmH^{2}(Y^{135}_{k^{\ac}}, \Lambda(1))_{\frakm^{[p]}_{\triplef}}.
\end{equation*}
By Proposition \ref{Yk} $(1)$ and Proposition \ref{bl-coho}, we have 
\begin{equation*}
\begin{aligned}
&\rmH^{2}(Y^{i(i+1)(i+2)}_{k^{\ac}}, \Lambda(1))=\rmH^{2}(X^{i(i+1)(i+2)}_{k^{\ac}}, \Lambda(1))\oplus \rmH^{0}(X^{i(i+1)(i+2)(i+3)(i+5)}_{k^{\ac}}, \Lambda);\\
&\rmH^{2}(Y^{024}_{k^{\ac}}, \Lambda(1))=\rmH^{2}(X^{024}_{k^{\ac}}, \Lambda(1))\oplus \rmH^{0}(X^{01234}_{k^{\ac}}, \Lambda) \oplus \rmH^{0}(X^{01245}_{k^{\ac}}, \Lambda) \oplus \rmH^{0}(X^{02345}_{k^{\ac}}, \Lambda)\oplus \rmH^{0}(X^{012345}_{k^{\ac}}, \Lambda);\\
&\rmH^{2}(Y^{135}_{k^{\ac}}, \Lambda(1))=\rmH^{2}(X^{135}_{k^{\ac}}, \Lambda(1))\oplus \rmH^{0}(X^{01235}_{k^{\ac}}, \Lambda) \oplus \rmH^{0}(X^{01345}_{k^{\ac}}, \Lambda) \oplus \rmH^{0}(X^{12345}_{k^{\ac}}, \Lambda)\oplus \rmH^{0}(X^{012345}_{k^{\ac}}, \Lambda).
\end{aligned}
\end{equation*}
On the other hand, we have 
\begin{equation*}
\begin{aligned}
\rmH^{0}(a_{1*}\Lambda)= &\oplus^{5}_{i=0}\rmH^{0}(X^{i(i+1)(i+2)(i+3)}_{k^{\ac}}, \Lambda)\\
                                               &\oplus\oplus^{5}_{i=0}\rmH^{0}(X^{i(i+1)(i+2)(i+4)}_{k^{\ac}}, \Lambda)\\
                                               &\oplus\oplus^{5}_{i=0} \rmH^{0}(X^{i(i+1)(i+2)(i+3)(i+5)}_{k^{\ac}}, \Lambda)\\
                                               &\oplus\rmH^{0}(X^{012345}_{k^{\ac}}, \Lambda).\\
\end{aligned}
\end{equation*}
We claim that the terms 
\begin{equation*}
\begin{aligned}
\rmH^{2}(X^{i(i+1)(i+2)}_{k^{\ac}}, \Lambda(1)), \hphantom{a} \rmH^{2}(X^{024}_{k^{\ac}}, \Lambda(1)),\hphantom{b}\rmH^{2}(X^{135}_{k^{\ac}}, \Lambda(1))\\ 
\end{aligned}
\end{equation*}
in $\rmH^{2}(a_{0*}\Lambda(1))$ are cancelled by the terms 
\begin{equation*}
\begin{aligned}
\rmH^{0}(X^{i(i+1)(i+2)(i+3)}_{k^{\ac}}, \Lambda),\hphantom{c}\rmH^{0}(X^{i(i+1)(i+2)(i+4)}_{k^{\ac}}, \Lambda)\\ 
\end{aligned}
\end{equation*}
in $\rmH^{0}(a_{1*}\Lambda)$ under the Gysin map $\tau$ after localizing at $\frakm_{\triplef^{[p]}}$. For example, consider the term 
\begin{equation*}
\begin{aligned}
\rmH^{2}(X^{012}_{[-++],k^{\ac}}, \Lambda(1))&=\rmH^{2}(\PP^{1}(X^{B}_{-}), \Lambda(1))\otimes \rmH^{0}(\PP^{1}(X^{B}_{+}), \calO)\otimes \rmH^{0}(\PP^{1}(X^{B}_{+}), \Lambda)\\
&\oplus \rmH^{0}(\PP^{1}(X^{B}_{-}), \Lambda)\otimes \rmH^{2}(\PP^{1}(X^{B}_{+}), \Lambda(1))\otimes \rmH^{0}(\PP^{1}(X^{B}_{+}), \Lambda)\\
&\oplus \rmH^{0}(\PP^{1}(X^{B}_{-}), \Lambda)\otimes \rmH^{0}(\PP^{1}(X^{B}_{+}), \Lambda)\otimes \rmH^{2}(\PP^{1}(X^{B}_{+}), \Lambda(1))\\
\end{aligned}
\end{equation*}
then it is cancelled  by 
\begin{equation*}
\begin{aligned}
&\rmH^{0}(X^{0124}_{[\pm++], k^{\ac}}, \Lambda)=\rmH^{0}(X^{B}_{0}(p), \Lambda)\otimes \rmH^{0}(\PP^{1}(X^{B}_{+}), \Lambda) \otimes \rmH^{0}(\PP^{1}(X^{B}_{+}), \Lambda)\\
&\rmH^{0}(X^{0245}_{[+\pm+], k^{\ac}}, \Lambda)=\rmH^{0}(\PP^{1}(X^{B}_{+}), \Lambda)\otimes \rmH^{0}(X^{B}_{0}(p), \Lambda) \otimes \rmH^{0}(\PP^{1}(X^{B}_{+}), \Lambda)\\
&\rmH^{0}(X^{0234}_{[++\pm], k^{\ac}}, \Lambda)=\rmH^{0}(\PP^{1}(X^{B}_{+}), \Lambda)\otimes \rmH^{0}(\PP^{1}(X^{B}_{+}), \Lambda) \otimes \rmH^{0}(X^{B}_{0}(p), \Lambda)\\
\end{aligned}
\end{equation*}
after applying the localizing at $\frakm^{[p]}_{\triplef}$. This follows from Proposition \ref{component-grp}(4) that the Gysin map is surjective up to an Eisenstein part. Similar computations as in this case proves the claim. Then it is clear that  
\begin{equation*}
A_{2}(Y_{k}, \Lambda)^{0}_{\frakm^{[p]}_{\triplef}}\cong\oplus^{5}_{i=0}\rmH^{0}(X^{i(i+1)(i+2)(i+3)(i+5)}_{k}, \Lambda)_{\frakm^{[p]}_{\triplef}}\oplus \rmH^{0}(X^{012345}_{k}, \Lambda)_{\frakm^{[p]}_{\triplef}}.
\end{equation*}
 This finishes the proof of the first part. 
 \end{proof}
 
 \begin{remark}
Those surfaces of the form $Y^{i(i+1)(i+2)(i+3)(i+5)}$ and $Y^{012345}$ should be considered as the \emph{featuring cycles} in the terminology of \cite[\S 3.2]{Liu-cubic}. To compute the potential map, we need to understand the intersection matrix of these featuring cycles. The computation below can be viewed as a down-to-earth way of computing such intersection matrix although the intersection numbers do not appear explicitly. \end{remark}
 
\begin{myproof}{Theorem}{\ref{arithmetic-level-raising}}
The first isomorphism in the statement of the theorem is clearly a consequence of the second more canonical isomorphism in light of what we have explained in the Shimura curve case in Lemma \ref{comp-char}. Therefore we will prove the more canonical isomorphism.

Recall that we have by Lemma \ref{feature-cycle}
\begin{enumerate}
\item $A_{2}(Y_{k}, \Lambda)^{0}_{\frakm^{[p]}_{\triplef}}=\oplus^{5}_{i=0}\rmH^{0}(X^{i(i+1)(i+2)(i+3)(i+5)}_{k}, \Lambda)_{\frakm^{[p]}_{\triplef}}\oplus \rmH^{0}(X^{012345}_{k}, \Lambda)_{\frakm^{[p]}_{\triplef}}$;
\item  $A^{2}(Y_{k}, \Lambda)^{0}_{\frakm^{[p]}_{\triplef}}=\oplus^{5}_{i=0}\rmH^{2}(X^{i(i+1)(i+2)(i+3)(i+5)}_{k}, \Lambda(1))_{\frakm^{[p]}_{\triplef}}\oplus \rmH^{0}(X^{012345}_{k}, \Lambda)_{\frakm^{[p]}_{\triplef}}$.
\end{enumerate}
with the potential map $A_{2}(Y_{k}, \Lambda)^{0}_{\frakm^{[p]}_{\triplef}}\xrightarrow{\nabla} A^{2}(Y_{k}, \Lambda)^{0}_{\frakm^{[p]}_{\triplef}}$ induced by $\tau\circ\rho$.  

We consider the terms 
\begin{equation*}
\begin{aligned}
&\rmH^{2}(X^{01235}_{[-\pm\pm]}, \Lambda(1))_{\frakm^{[p]}_{\triplef}}=\rmH^{2}(\PP^{1}(X^{B}_{-}), \Lambda(1))_{\frakm^{[p]}_{1}}\otimes \rmH^{0}(X^{B}_{0}(p), \Lambda)_{\frakm^{[p]}_{2}}\otimes \rmH^{0}(X^{B}_{0}(p), \Lambda)_{\frakm^{[p]}_{3}}\\
&\rmH^{2}(X^{02345}_{[+\pm\pm]}, \Lambda(1))_{\frakm^{[p]}_{\triplef}}=\rmH^{2}(\PP^{1}(X^{B}_{+}), \Lambda(1))_{\frakm^{[p]}_{1}}\otimes \rmH^{0}(X^{B}_{0}(p), \Lambda)_{\frakm^{[p]}_{2}}\otimes \rmH^{0}(X^{B}_{0}(p), \Lambda)_{\frakm^{[p]}_{3}}\\
\end{aligned}
\end{equation*}
in  $A^{2}(Y_{k}, \Lambda)^{0}_{\frakm^{[p]}_{\triplef}}$ and consider the images of the terms
\begin{equation*}
\begin{aligned}
&\rmH^{0}(X^{01235}_{[-\pm\pm]}, \Lambda)_{\frakm^{[p]}_{\triplef}}=\rmH^{0}(\PP^{1}(X^{B}_{-}), \Lambda)_{\frakm^{[p]}_{1}}\otimes \rmH^{0}(X^{B}_{0}(p), \Lambda)_{\frakm^{[p]}_{2}}\otimes \rmH^{0}(X^{B}_{0}(p), \Lambda)_{\frakm^{[p]}_{3}}\\
&\rmH^{0}(X^{02345}_{[+\pm\pm]}, \Lambda)_{\frakm^{[p]}_{\triplef}}=\rmH^{0}(\PP^{1}(X^{B}_{+}), \Lambda)_{\frakm^{[p]}_{1}}\otimes \rmH^{0}(X^{B}_{0}(p), \Lambda)_{\frakm^{[p]}_{2}}\otimes \rmH^{0}(X^{B}_{0}(p), \Lambda)_{\frakm^{[p]}_{3}}\\
\end{aligned}
\end{equation*}
from $A_{2}(Y_{k}, \lambda)^{0}_{\frakm^{[p]}_{\triplef}}$ via $\nabla$. The map $\nabla$ restricted to  $\rmH^{0}(X^{01235}_{k}, \Lambda)_{\frakm^{[p]}_{\triplef}}$ fits into the following diagram
\begin{equation*}
\begin{tikzcd}
\rmH^{2}(Y^{135}_{k}, \Lambda(1))_{\frakm^{[p]}_{\triplef}} \arrow[r, "\rho"]                 &\rmH^{2}(Y^{012345}_{k}, \Lambda(1))_{\frakm^{[p]}_{\triplef}}  \arrow[r,"\tau"]                 & \rmH^{4}(Y^{135}_{k}, \Lambda(2))_{\frakm^{[p]}_{\triplef}}           \\
 \rmH^{0}(X^{01235}_{k}, \Lambda)_{\frakm^{[p]}_{\triplef}}\arrow[u, hook] \arrow[r, "\rho"] &  \rmH^{0}(X^{012345}_{k}, \Lambda)_{\frakm^{[p]}_{\triplef}} \arrow[u, hook] \arrow[r, "\tau"] &  \rmH^{2}(X^{01235}_{k}, \Lambda(1))_{\frakm^{[p]}_{\triplef}} \arrow[u, hook]
\end{tikzcd}
\end{equation*}
and similarly the map $\nabla$ restricted to  $\rmH^{0}(X^{02345}_{k}, \Lambda)_{\frakm^{[p]}_{\triplef}}$ fits into the following diagram
\begin{equation*}
\begin{tikzcd}
\rmH^{2}(Y^{024}_{k}, \Lambda(1))_{\frakm^{[p]}_{\triplef}} \arrow[r, "\rho"]                 &\rmH^{2}(Y^{012345}_{k}, \Lambda(1))_{\frakm^{[p]}_{\triplef}}  \arrow[r,"\tau"]                 & \rmH^{4}(Y^{024}_{k}, \Lambda(2))_{\frakm^{[p]}_{\triplef}}           \\
 \rmH^{0}(X^{02345}_{k}, \Lambda)_{\frakm^{[p]}_{\triplef}}\arrow[u, hook] \arrow[r, "\rho"] &  \rmH^{0}(X^{012345}_{k}, \Lambda)_{\frakm^{[p]}_{\triplef}} \arrow[u, hook] \arrow[r, "\tau"] &  \rmH^{2}(X^{02345}_{k}, \Lambda(1))_{\frakm^{[p]}_{\triplef}}. \arrow[u, hook]
\end{tikzcd}
\end{equation*}
Here the top row of the diagram identifies the locations of the featuring cycles and the second row is the actural restriction of the map $\nabla$ to the featuring cycles. Putting these diagrams together, the restriction of $\nabla$ at
\begin{equation*}
\rmH^{0}(X^{01235}_{[-\pm\pm], k}, \Lambda)_{\frakm^{[p]}_{\triplef}}\oplus\rmH^{0}(X^{02345}_{[+\pm\pm], k}, \Lambda)_{\frakm^{[p]}_{\triplef}}\xrightarrow{\nabla} \rmH^{2}(X^{01235}_{[-\pm\pm], k}, \Lambda)_{\frakm^{[p]}_{\triplef}}\oplus\rmH^{2}(X^{02345}_{[+\pm\pm], k}, \Lambda)_{\frakm^{[p]}_{\triplef}}
\end{equation*}
is given by the composite
\begin{equation*}
\begin{tikzcd}
\rmH^{0}(\PP^{1}(X^{B}_{\pm}), \Lambda)_{\frakm^{[p]}_{1}}\otimes \rmH^{0}(X^{B}_{0}(p), \Lambda)_{\frakm^{[p]}_{2}}\otimes \rmH^{0}(X^{B}_{0}(p), \Lambda)_{\frakm^{[p]}_{3}}\arrow[d] \\
\rmH^{0}(X^{B}_{0}(p), \Lambda)_{\frakm^{[p]}_{1}}\otimes \rmH^{0}(X^{B}_{0}(p), \Lambda)_{\frakm^{[p]}_{2}}\otimes \rmH^{0}(X^{B}_{0}(p), \Lambda)_{\frakm^{[p]}_{3}} \arrow[d] \\
\rmH^{2}(\PP^{1}(X^{B}_{\pm}), \Lambda(1))_{\frakm^{[p]}_{1}}\otimes \rmH^{0}(X^{B}_{0}(p), \Lambda)_{\frakm^{[p]}_{2}}\otimes \rmH^{0}(X^{B}_{0}(p), \Lambda)_{\frakm^{[p]}_{3}}\\         
\end{tikzcd}
\end{equation*} 
where we write $\rmH^{*}(\PP^{1}(X^{B}_{\pm}), \Lambda)$ for $\rmH^{*}(\PP^{1}(X^{B}_{-}), \Lambda)\oplus \rmH^{*}(\PP^{1}(X^{B}_{+}), \Lambda)$.
Notice that the composite of the above map is given by the intersection matrix 
\begin{equation*}
\begin{pmatrix}
-(p+1) &T_{p}\\
T_{p} &-(p+1)\\
\end{pmatrix}
\end{equation*}
on the first factors of the tensor products and the identity map restricted to the rest of the two factors by what we have explained in the Shimura curve case treated as in Theorem \ref{level-raise-curve}. After quotient out by the ideals $(I^{[p]}_{1, n}, I^{[p]}_{2, n}, I^{[p]}_{3, n})$, we see the quotient of 
\begin{equation}
\rmH^{2}(X^{01235}_{[-\pm\pm]}, \Lambda(1))_{\frakm^{[p]}_{\triplef}}\oplus \rmH^{2}(X^{02345}_{[+\pm\pm]}, \Lambda(1))_{\frakm^{[p]}_{\triplef}}
\end{equation}
in $A^{2}(Y_{k}, \Lambda)^{0}_{\frakm^{[p]}_{\triplef}}$ by the image of 
\begin{equation}
\rmH^{0}(X^{01235}_{[-\pm\pm]}, \Lambda)_{\frakm^{[p]}_{\triplef}}\oplus\rmH^{0}(X^{02345}_{[+\pm\pm]}, \Lambda)_{\frakm^{[p]}_{\triplef}}
\end{equation}
in $A_{2}(Y_{k}, \Lambda)^{0}_{\frakm^{[p]}_{\triplef}}$
can be identified with
\begin{equation*}
\Phi_{\Lambda}/I^{[p]}_{1, n}\otimes  \rmH^{0}(X^{B}_{0}(p), \Lambda)_{\frakm^{[p]}_{2}}/I^{[p]}_{2,n}\otimes  \rmH^{0}(X^{B}_{0}(p), \Lambda)_{\frakm^{[p]}_{3}}/I^{[p]}_{3,n}.
\end{equation*}

Next we consider the term
\begin{equation*}
\rmH^{2}(X^{01235}_{[-\pm\pm]}, \Lambda(1))_{\frakm^{[p]}_{\triplef}} 
\end{equation*}
in  $A^{2}(Y_{k}, \Lambda)^{0}_{\frakm^{[p]}_{\triplef}}$ and the image of the terms
\begin{equation*}
\rmH^{0}(X^{01345}_{[\pm-\pm]}, \Lambda)_{\frakm^{[p]}_{\triplef}}\hphantom{a}\text{and}\hphantom{a}\rmH^{0}(X^{01234}_{[\pm+\pm]}, \Lambda)_{\frakm^{[p]}_{\triplef}} 
\end{equation*}
in  $A_{2}(Y_{k}, \Lambda)^{0}_{\frakm^{[p]}_{\triplef}}$ via $\nabla$. 

Then we have similar diagrams as in the previous case 
\begin{equation*}
\begin{tikzcd}
\rmH^{2}(Y^{135}_{k}, \Lambda(1))_{\frakm^{[p]}_{\triplef}} \arrow[r, "\rho"]                 &\rmH^{2}(Y^{012345}_{k}, \Lambda(1))_{\frakm^{[p]}_{\triplef}}  \arrow[r,"\tau"]                 & \rmH^{4}(Y^{135}_{k},\Lambda(2))_{\frakm^{[p]}_{\triplef}}           \\
 \rmH^{0}(X^{01345}_{k}, \Lambda)_{\frakm^{[p]}_{\triplef}}\arrow[u, hook] \arrow[r, "\rho"] &  \rmH^{0}(X^{012345}_{k}, \Lambda)_{\frakm^{[p]}_{\triplef}} \arrow[u, hook] \arrow[r, "\tau"] &  \rmH^{2}(X^{01235}_{k}, \Lambda(1))_{\frakm^{[p]}_{\triplef}}. \arrow[u, hook]
\end{tikzcd}
\end{equation*}
and 
\begin{equation*}
\begin{tikzcd}
\rmH^{2}(Y^{024}_{k},\Lambda(1))_{\frakm^{[p]}_{\triplef}} \arrow[r, "\rho"]                 &\rmH^{2}(Y^{012345}_{k}, \Lambda(1))_{\frakm^{[p]}_{\triplef}}  \arrow[r,"\tau"]                 & \rmH^{4}(Y^{135}_{k}, \Lambda(2))_{\frakm^{[p]}_{\triplef}}           \\
 \rmH^{0}(X^{01234}_{k}, \Lambda)_{\frakm^{[p]}_{\triplef}}\arrow[u, hook] \arrow[r, "\rho"] &  \rmH^{0}(X^{012345}_{k}, \Lambda)_{\frakm^{[p]}_{\triplef}} \arrow[u, hook] \arrow[r, "\tau"] &  \rmH^{2}(X^{01235}_{k}, \Lambda(1))_{\frakm^{[p]}_{\triplef}}. \arrow[u, hook]
\end{tikzcd}
\end{equation*}
Putting these together, the restriction at 
\begin{equation*}
\rmH^{0}(X^{01345}_{[-\pm\pm], k}, \Lambda)_{\frakm^{[p]}_{\triplef}}\oplus\rmH^{0}(X^{01234}_{[+\pm\pm], k}, \Lambda)_{\frakm^{[p]}_{\triplef}} \xrightarrow{\nabla} \rmH^{2}(X^{01235}_{[-\pm\pm]}, \Lambda(1))_{\frakm^{[p]}_{\triplef}}
\end{equation*}
 is given by 
\begin{equation*}
\begin{tikzcd}
\rmH^{0}(X^{B}_{0}(p), \Lambda)_{\frakm^{[p]}_{1}}\otimes \rmH^{0}(\PP^{1}(X^{B}_{\pm}), \Lambda)_{\frakm^{[p]}_{2}}\otimes\rmH^{0}(X^{B}_{0}(p), \Lambda)_{\frakm^{[p]}_{3}}\arrow[d] \\
\rmH^{0}(X^{B}_{0}(p), \Lambda)_{\frakm^{[p]}_{1}}\otimes \rmH^{0}(X^{B}_{0}(p), \Lambda)_{\frakm^{[p]}_{2}}\otimes \rmH^{0}(X^{B}_{0}(p), \Lambda)_{\frakm^{[p]}_{3}} \arrow[d] \\
\rmH^{2}(\PP^{1}(X^{B}_{-}), \Lambda(1))_{\frakm^{[p]}_{1}}\otimes \rmH^{0}(X^{B}_{0}(p), \Lambda)_{\frakm^{[p]}_{2}}\otimes \rmH^{0}(X^{B}_{0}(p), \Lambda)_{\frakm^{[p]}_{3}}.\\         
\end{tikzcd}
\end{equation*} 
In the above diagram, the first arrow is given by the restriction morphism 
\begin{equation*}
\rmH^{0}(\PP^{1}(X^{B}_{-}), \Lambda)_{\frakm^{[p]}_{2}}\oplus  \rmH^{0}(\PP^{1}(X^{B}_{+}), \Lambda)_{\frakm^{[p]}_{2}}\rightarrow \rmH^{0}(X^{B}_{0}(p), \Lambda)_{\frakm^{[p]}_{2}}
\end{equation*}
on the second factor of the tensor product and the identity map on the rest of the factors. The second arrow is given by the Gysin morphism 
\begin{equation*}
\rmH^{0}(X^{B}_{0}(p), \Lambda)_{\frakm^{[p]}_{1}}\xrightarrow{\tau} \rmH^{2}(\PP^{1}(X^{B}_{-}), \Lambda(1))_{\frakm^{[p]}_{1}}
\end{equation*}
on the first factor of the tensor product which is surjective up to Eisenstein part by Proposition \ref{component-grp} $(2)$. Therefore by Proposition \ref{component-grp} $(3)$,  the quotient of term
\begin{equation*}
\rmH^{2}(X^{01235}_{[-\pm\pm]}, \Lambda(1))_{\frakm^{[p]}_{\triplef}}
\end{equation*} 
in  $A_{2}(Y_{k}, \Lambda)^{0}_{\frakm^{[p]}_{\triplef}}$ by the terms
\begin{equation*}
\rmH^{0}(X^{01345}_{[\pm-\pm], k}, \Lambda)_{\frakm^{[p]}_{\triplef}}\oplus\rmH^{0}(X^{01234}_{[\pm+\pm], k}, \Lambda)_{\frakm^{[p]}_{\triplef}}
\end{equation*}
 in  $A^{2}(Y_{k}, \Lambda)^{0}_{\frakm^{[p]}_{\triplef}}$ is given by 
\begin{equation*}
\rmH^{2}(\PP^{1}(X^{B}_{-}), \Lambda(1))/I^{[p]}_{1, n}\otimes\calX^{\vee}_{\Lambda}/I^{[p]}_{2, n}\otimes \rmH^{0}(X^{B}_{0}(p), \Lambda)/I^{[p]}_{3, n}
\end{equation*}
after we quotient out the ideals $(I^{[p]}_{1, n}, I^{[p]}_{2, n}, I^{[p]}_{3, n})$.

Finally, we consider the term  $\rmH^{0}(X^{012345}_{k}, \Lambda)_{\frakm^{[p]}_{\triplef}}$ in $A_{2}(Y_{k}, \Lambda)^{0}_{\frakm^{[p]}_{\triplef}}$ and the same cohomology group $\rmH^{0}(X^{012345}_{k}, \Lambda)_{\frakm^{[p]}_{\triplef}}$ in $A^{2}(Y_{k}, \Lambda)^{0}_{\frakm^{[p]}_{\triplef}}$. It is not difficult to see by a similar reasoning as before that the restriction of $\nabla$ on this term is given by the identity map. Therefore this term does not contribute to the quotient.

If we make similar computations for all the terms in $A_{2}(Y_{k}, \Lambda)^{0}_{\frakm^{[p]}_{\triplef}}$ and $A^{2}(Y_{k}, \Lambda)^{0}_{\frakm^{[p]}_{\triplef}}$, we obtain the following isomorphism by \eqref{sing-exact}
\begin{equation*}
\begin{aligned}
 \rmH^{1}_{\sing}(K, \rmM^{[p]}_{n}(\triplef)(-1))  \cong& \Phi_{\calO_{1}}/I^{[p]}_{1,n}\otimes \calX^{\vee}_{\calO_{2}}/I^{[p]}_{2,n}\otimes \calX^{\vee}_{\calO_{2}}/I^{[p]}_{3,n}\\
& \calX^{\vee}_{\calO_{1}}/I^{[p]}_{1,n}\otimes\Phi_{\calO_{2}}/I^{[p]}_{2,n}\otimes \calX^{\vee}_{\calO_{3}}/I^{[p]}_{3,n}\\
&\calX^{\vee}_{\calO_{1}}/I^{[p]}_{1,n}\otimes \calX^{\vee}_{\calO_{2}}/I^{[p]}_{2,n}\otimes\Phi_{\calO_{3}}/I^{[p]}_{3,n}.
\end{aligned}
\end{equation*} 
Finally we apply Lemma \ref{comp-char} and arrive at the desired isomorphism
\begin{equation*}
\rmH^{1}_{\sing}(K, \rmM^{[p]}_{n}(\triplef)(-1))  \cong (\Phi_{\calO_{1}}/I^{[p]}_{1,n}\otimes \Phi_{\calO_{2}}/I^{[p]}_{2,n}\otimes \Phi_{\calO_{3}}/I^{[p]}_{3,n})^{\oplus 3}.
\end{equation*} 
\end{myproof}

\begin{corollary}\label{main-coro}
Let $p$ be an $n$-admissible prime for $\triplef$. Under the assumption in the Theorem \ref{arithmetic-level-raising}, we have the following statements. 
\begin{enumerate}
\item The singular quotient $\rmH^{1}_{\sing}(\QQ_{p}, \rmM^{[p]}_{n}(\triplef)(-1))$ is free of rank $3$ over $\calO_{n}$. 
\item We have a canonical isomorphism
\begin{equation*}
\rmH^{1}_{\sing}(\QQ_{p}, \rmM^{[p]}_{n}(\triplef)(-1)) \cong \oplus^{3}_{j=1}(\otimes^{3}_{i=1}\Phi_{\calO_{i}}/I^{[p]}_{i,n})
\end{equation*}
which induces an isomorphism
\begin{equation*}
 \rmH^{1}_{\sing}(\QQ_{p}, \rmM^{[p]}_{n}(\triplef)(-1))\cong \oplus^{3}_{j=1}(\otimes^{3}_{i=1}S^{B}_{2}(N^{+}, \calO_{i})/I_{i,n})
 \end{equation*}
\end{enumerate}
\end{corollary}
\begin{proof}
The first part follows from the same proof as in Lemma \ref{3-fin}. For the second part, it follows from the discussion in see \cite[\S 1.7]{BD-Mumford} that the non-trivial element in the Galois group $\Gal(K/\QQ_{p})$ acts by $U_{p}$ on the group of connected components $\Phi$ for the Shimura curve $X$. Thus it acts by the product of $(\epsilon_{p,1}, \epsilon_{p,2}, \epsilon_{p,3})$ on each copy of 
$\otimes^{3}_{i=1}\Phi_{\calO_{i}}/I^{[p]}_{i,n}$
in $\oplus^{3}_{j=1}(\otimes^{3}_{i=1}\Phi_{\calO_{i}}/I^{[p]}_{i,n})$. 
Note that the product of $(\epsilon_{p,1}, \epsilon_{p,2}, \epsilon_{p,3})$ is $1$ by the definition of an $n$-admissible prime $p$ for $\triplef$.  The result follows. 
\end{proof}

\subsection{Diagonal cycle classes and the first reciprocity law} Recall $X=X^{B^{\prime}}_{N^{+}, pN^{-}}$ is the Shimura curve associated to $B^{\prime}$ defined over $\QQ$ with integral model $\mathfrak{X}$ over $\ZZ_{(p)}$. We let 
\begin{equation*}
\theta: \mathfrak{X}\rightarrow \mathfrak{X}^{3}
\end{equation*}
be the diagonal embedding of $\mathfrak{X}$ in the triple fiber product of $\mathfrak{X}$ over  $\ZZ_{(p)}$. Then we obtain a class 
\begin{equation*}
\theta_{*}[\mathfrak{X}\otimes\QQ]\in \mathrm{CH}^{2}(\mathfrak{X}^{3}\otimes \QQ)
\end{equation*}
which we will refer to as the \emph{Gross-Schoen diagonal cycles}. Since $\rmH^{*}(\mathfrak{X}^{3}\otimes{\QQ^{\ac}}, \Lambda(2))_{\frakm^{[p]}_{\triplef}}$ vanishes unless $*=3$ as we have assumed that the maximal ideals in ${\frakm^{[p]}_{\triplef}}$ are residually irreducible, the cycle class map and the Hochschild-Serre spectral sequence gives rise to  the following Abel-Jacobi map 
\begin{equation*}
\mathrm{AJ}^{\circ}_{\triplef}: \Ch^{2}(\mathfrak{X}^{3}\otimes\QQ)\rightarrow \rmH^{1}(\QQ, \rmH^{3}(\mathfrak{X}^{3}\otimes{\QQ^{\ac}}, \Lambda(2))_{\frakm^{[p]}_{\triplef}}).
\end{equation*}
Note that by \eqref{Kunneth}  we have the following isomorphism
\begin{equation*}
\rmH^{3}(\mathfrak{X}^{3}\otimes{\QQ^{\ac}}, \Lambda(2))_{\frakm^{[p]}_{\triplef}}\cong \otimes^{3}_{i=1}\rmH^{1}(\mathfrak{X}\otimes{\QQ^{\ac}}, \Lambda(1))_{\frakm^{[p]}_{i}}.
\end{equation*}
For $i=1, 2, 3$, recall that 
\begin{itemize}
\item $\rmT^{[p]}_{i}=\rmH^{1}(\mathfrak{X}\otimes{\QQ^{\ac}}, \calO_{i}(1))_{\mathfrak{m}^{[p]}_{i}}$;  
\item $\rmT^{[p]}_{i, n}= \rmT^{[p]}_{i}/I^{[p]}_{i, n}$; 
\item $\rmM^{[p]}(\triplef)=\rmT^{[p]}_{1}\otimes \rmT^{[p]}_{2}\otimes \rmT^{[p]}_{3}$;
\item $\rmM^{[p]}_{n}(\triplef)=\rmT^{[p]}_{1, n}\otimes \rmT^{[p]}_{2, n}\otimes \rmT^{[p]}_{3, n}$.
\end{itemize}
Note we have a natural map  
\begin{equation*}
\otimes^{3}_{i=1}\rmH^{1}(\mathfrak{X}\otimes{\QQ^{\ac}}, \Lambda(1))_{\frakm^{[p]}_{i}}\rightarrow \rmM^{[p]}(\triplef)=\otimes^{3}_{i=1}\rmT_{i}
 \end{equation*}
which can be composed with the Abel-Jacobi map $\mathrm{AJ}^{\circ}_{\triplef}$ to obtain the following Abel-Jacobi map for $\rmM^{[p]}(\triplef)(-1)$
\begin{equation*}
\mathrm{AJ}_{\triplef}: \Ch^{2}(\mathfrak{X}^{3}\otimes \QQ)\rightarrow \rmH^{1}(\QQ, \rmM^{[p]}(\triplef)(-1)).\\
\end{equation*} 
We also have a natural quotient map $\rmM^{[p]}(\triplef)(-1)\rightarrow \rmM^{[p]}_{n}(\triplef)(-1)$ which we can further compose with $\mathrm{AJ}_{\triplef}$ to obtain the Abel-Jacobi map for $\rmM^{[p]}_{n}(\triplef)(-1)$
\begin{equation*}
\mathrm{AJ}_{\triplef, n}: \Ch^{2}(\mathfrak{X}^{3}\otimes \QQ)\rightarrow \rmH^{1}(\QQ, \rmM^{[p]}_{n}(\triplef)(-1)).\\
\end{equation*}
Now we consider the following diagram 
\begin{equation*}
\begin{tikzcd}
\Ch^{2}(\mathfrak{X}^{3}\otimes \QQ) \arrow[r, "\mathrm{AJ}_{\triplef, n}"] \arrow[rdd, bend right, "\partial_{p} \mathrm{AJ}_{\triplef, n}"'] &  \rmH^{1}(\QQ, \rmM^{[p]}_{n}(\triplef)(-1)) \arrow[d, "\mathrm{loc}_{p}"] \\
                                   &  \rmH^{1}(\QQ_{p}, \rmM^{[p]}_{n}(\triplef)(-1)) \arrow[d, "\partial_{p}"] \\
                                   &   \rmH^{1}_{\sing}(\QQ_{p}, \rmM^{[p]}_{n}(\triplef)(-1))              
\end{tikzcd}
\end{equation*}
where $\partial_{p}\mathrm{AJ}_{\triplef, n}$ is the composite of the right part of the diagram. Then we define the \emph{Gross-Schoen diagonal cycle class} to be image of  $\theta_{*}[\mathfrak{X}\otimes\QQ]$ under the map $\mathrm{AJ}_{\triplef}$ and we denote this element by 
\begin{equation*}
\Theta^{[p]}\in  \rmH^{1}(\QQ, \rmM^{[p]}(\triplef)(-1)).
\end{equation*}
Similarly we denote the image of  $\theta_{*}[\mathfrak{X}\otimes\QQ]$ under the map $\mathrm{AJ}_{\triplef, n}$ by
\begin{equation*}
\Theta^{[p]}_{n}\in  \rmH^{1}(\QQ, \rmM^{[p]}_{n}(\triplef)(-1)).
\end{equation*}
We will be concerned with the singular residue at $p$ of the element $\Theta^{[p]}_{n}$ in the following. By definition, this is the image of the cycle  $\theta_{*}[\mathfrak{X}\otimes\QQ]$ under the map $\partial_{p}\mathrm{AJ}_{\triplef, n}$ which we will denote by
\begin{equation*}
\partial_{p}\Theta^{[p]}_{n }\in \rmH^{1}_{\sing}(\QQ_{p}, \rmM^{[p]}_{n}(\triplef)(-1)).
\end{equation*}
By Corollary \ref{main-coro},  we can view $\partial_{p}\Theta^{[p]}_{n}$ as an element of  
\begin{equation*}
\oplus^{3}_{j=1}(\otimes^{3}_{i=1}S^{B}_{2}(N^{+}, \calO_{i}){/I_{i,n}}).
\end{equation*}
For $j=1, 2, 3$, we denote by 
\begin{equation*}
\partial^{(j)}_{p}\Theta^{[p]}_{n}\in \otimes^{3}_{i=1}S^{B}_{2}(N^{+}, \calO_{i}){/I_{i,n}}
\end{equation*}
its projection to the $j$-th direct summand in 
\begin{equation*}
\oplus^{3}_{j=1}(\otimes^{3}_{i=1}S^{B}_{2}(N^{+}, \calO_{i}){/I_{i,n}}).
\end{equation*} 
We define a pairing 
\begin{equation}\label{pairing}
(\hphantom{a},\hphantom{a}):\otimes^{3}_{i=1}S^{B}_{2}(N^{+}, \calO_{i})\times \otimes^{3}_{i=1}S^{B}_{2}(N^{+}, \calO_{i}) \rightarrow \calO
\end{equation}
by the following formula. Let 
\begin{equation*}
\zeta_{1}\otimes \zeta_{2}\otimes \zeta_{3}, \phi_{1}\otimes \phi_{2}\otimes \phi_{3}\in \otimes^{3}_{i=1}S^{B}_{2}(N^{+}, \calO_{i})
\end{equation*}
then we define  
\begin{equation*}
(\zeta_{1}\otimes \zeta_{2}\otimes \zeta_{3},  \phi_{1}\otimes \phi_{2}\otimes \phi_{3})=\sum_{z_{1}, z_{2}, z_{3}} \zeta_{1}\phi_{1}(z_{1})\otimes\zeta_{2}\phi_{2}(z_{2})\otimes\zeta_{3}\phi_{3}(z_{3})
\end{equation*}
where $(z_{1}, z_{2}, z_{3})$ runs through all the elements in the finite set $(X^{B})^{3}$. Note that when $\calO_{1}=\calO_{2}=\calO_{3}$, the pairing which apriori valued in $\calO\otimes \calO\times \calO$ can be taken to be valued in $\calO$ by using the natural multiplication map. This pairing subsequently induces a pairing 
\begin{equation*}
\begin{aligned}
(\hphantom{a},\hphantom{a}):\otimes^{3}_{i=1}S^{B}_{2}(N^{+}, \calO_{i})[I_{i, n}]\times \otimes^{3}_{i=1}S^{B}_{2}(N^{+}, \calO_{i}){/I_{i, n}} \rightarrow \calO_{n}.\\
\end{aligned}
\end{equation*}
The following theorem is the analogue of the \emph{first reciprocity law} for Heegner points on Shimura curves \cite[Theorem 4.1]{BD-Main}, the \emph{congruence formulae} in \cite[Theorem 4.11]{Liu-HZ} and \cite[Theorem 4.5]{Liu-cubic}.
\begin{theorem}[First reciprocity law]\label{recip}
Let $p$ be an $n$-admissible prime for $\triplef$. We assume the assumptions in Theorem \ref{arithmetic-level-raising} are satisfied. Let $\phi_{1}\otimes \phi_{2}\otimes \phi_{3}\in \otimes^{3}_{i=1}S^{B}_{2}(N^{+}, \calO_{i})[I_{i,n}]$. Then for $j=1, 2, 3$, the following formula holds
\begin{equation*}
(\partial^{(j)}_{p}\Theta^{[p]}_{n}, \phi_{1}\otimes \phi_{2}\otimes \phi_{3})=(p+1)^{3}\sum_{z\in X^{B}}\phi_{1}(z)\otimes\phi_{2}(z)\otimes\phi_{3}(z).
\end{equation*} 
\end{theorem}

\begin{proof}
We only prove the formula  for $j=1$, the other cases are proved exactly the same way. Consider the diagonal embedding of 
\begin{equation*}
\theta: \mathfrak{X}\rightarrow \mathfrak{X}^{3}
\end{equation*}
of the model of $X$ over $\calO_{K}$ into its threefold fiber product. Since $\mathfrak{X}$ is regular, the map $\theta$ extends to a map $\tilde{\theta}: \mathfrak{X}\rightarrow \calY$ such that $\pi\circ\tilde{\theta}=\theta$ by the universal property of the blow-up. We use the same notation $\tilde{\theta}: X_{k}\rightarrow Y_{k}$ to denote the map induced on the special fiber.  We apply Proposition \ref{cal-aj} to calculate $\partial^{(1)}_{p}\Theta^{[p]}_{n}$.  Thus we need to find the image of 
\begin{equation*}
Y^{(0)}_{k}\times_{Y_{k}} \tilde{\theta}_{*}X_{k} 
\end{equation*}
under the cycle class map in $A^{2}(Y_{k}, \Lambda)^{0}_{\frakm^{[p]}_{\triplef}}$. By Lemma \ref{feature-cycle} and the proof of Theorem \ref{arithmetic-level-raising}, the class $\partial^{(1)}_{p}\Theta^{[p]}_{n}$ is represented by the image of the characteristic function $\mathbf{1}_{0}(p)$ on $X^{B}_{0}(p)$ under the map
\begin{equation}
\rmH^{0}(X^{B}_{0}(p), \Lambda)\rightarrow \rmH^{2}(\PP^{1}(X^{B}_{\pm}), \Lambda(1))\otimes\rmH^{0}(X^{B}_{0}(p), \Lambda)\otimes\rmH^{0}(X^{B}_{0}(p), \Lambda).
\end{equation}
which is induced by the Gysin map on the first factor of the tensor product and by the identity maps on the rest of the two factors of the tensor product. Recall that we have two natural transition maps of Shimura sets
\begin{equation*}
\pi_{+}: X^{B}_{0}(p)\rightarrow X^{B} \text{\hphantom{a}and\hphantom{b}} \pi_{-}: X^{B}_{0}(p)\rightarrow X^{B}.
\end{equation*}
By making explicit the maps
\begin{equation*}
\begin{aligned}
&\rmH^{2}(\PP^{1}(X^{B}_{\pm}), \Lambda(1))\otimes\rmH^{0}(X^{B}_{0}(p), \Lambda)\otimes\rmH^{0}(X^{B}_{0}(p), \Lambda)\\
&\rightarrow \Phi_{\calO_{1}}/I^{[p]}_{1,n}\otimes \calX^{\vee}_{ \calO_{2}}/I^{[p]}_{2,n}\otimes \calX^{\vee}_{ \calO_{3}}/I^{[p]}_{3,n}\\
&\rightarrow \Phi_{\calO_{1}}/I^{[p]}_{1,n}\otimes \Phi_{ \calO_{2}}/I^{[p]}_{2,n}\otimes \Phi_{ \calO_{3}}/I^{[p]}_{3,n}\\
&\cong S^{B}_{2}(N^{+}, \calO_{1}){/I_{1,n}}\otimes S^{B}_{2}(N^{+}, \calO_{2}){/I_{2,n}}\otimes S^{B}_{2}(N^{+}, \calO_{3}){/I_{3,n}}.\\
\end{aligned}
\end{equation*}
The element $\partial^{(1)}_{p}\Theta^{[p]}_{n}\in S^{B}_{2}(N^{+}, \calO_{1}){/I_{1,n}}\otimes S^{B}_{2}(N^{+}, \calO_{2}){/I_{2,n}}\otimes S^{B}_{2}(N^{+}, \calO_{3}){/I_{3,n}}$ is given by 
\begin{equation*}
\frac{\theta_{*}(\pi_{+, *}+\epsilon_{p,1}\pi_{-, *})}{2}(\mathbf{1}_{0}(p))\otimes\frac{\theta_{*}(\pi_{+, *}+\epsilon_{p,2}\pi_{-, *})}{2}(\mathbf{1}_{0}(p))\otimes \frac{\theta_{*}(\pi_{+, *}+\epsilon_{p,3}\pi_{-, *})}{2}(\mathbf{1}_{0}(p))
\end{equation*}
where we abuse the notation and denote by $\theta: X^{B}\rightarrow (X^{B})^{3}$ the diagonal embedding of the Shimura set $X^{B}$.
Since $\pi_{+, *}\mathbf{1}_{0}(p)=\epsilon_{p, i}\pi_{-, *}\mathbf{1}_{0}(p)$ is the constant function on $X^{B}$ with value $p+1$, we have
\begin{equation*}
\begin{aligned}
&(\partial^{(1)}_{p}\Theta^{[p]}_{n}, \phi_{1}\otimes \phi_{2}\otimes \phi_{3})\\
&= (\frac{\theta_{*}(\pi_{+, *}+\epsilon_{p,1}\pi_{-,*})}{2}(\mathbf{1}_{0}(p))\otimes\frac{\theta_{*}(\pi_{+, *}+\epsilon_{p, 2}\pi_{-, *})}{2}(\mathbf{1}_{0}(p))\otimes\frac{\theta_{*}(\pi_{+, *}+\pi_{-, *})}{2}(\mathbf{1}_{0}(p)), \phi_{1}\otimes \phi_{2}\otimes \phi_{3})\\
&=(p+1)^{3}\sum_{z\in X^{B}}\phi_{1}(z)\otimes\phi_{2}(z)\otimes\phi_{3}(z).\\
\end{aligned}
\end{equation*}
The formula is proved. 
\end{proof}
\section{The Bloch-Kato conjecture for the triple product motive}
\subsection{Selmer groups of triple product motive} Let $f=\sum_{n\geq 1}a_{n}(f)q^{n}\in S^{\new}_{2}(\Gamma_{0}(N))$ be a normalized newform of weight $2$. We assume that $N$ admits a factorization $N=N^{+}N^{-}$ such that $(N^{+}, N^{-})=1$ and $N^{-}$ is square-free and is a product of odd number of primes. Let $E=\QQ(f)$ be the Hecke field of $f$ and $\lambda$ be a place of $E$ over $l$. We denote by $E_{\lambda}$ the completion of $E$ at $\lambda$. Let $\calO=\calO_{E_{\lambda}}$ be its ring of integers. To the newform $f$, we can attach a Galois representation
\begin{equation*}
\rho_{f}: G_{\QQ}\rightarrow \mathrm{GL}(\rmV_{f})
\end{equation*}
over $E_{\lambda}$ with determinant the $l$-adic cyclotomic character and satisfying 
\begin{equation*}
\mathrm{tr}(\rho_{f}(\mathrm{Frob}_{p}))=a_{p}(f) \text{ for all } p\nmid N. 
\end{equation*}
Let $\rmT_{f}$ be a Galois stable $\calO$-lattice in $\rmV_{f}$ and for each $n\geq 1$ we put 
\begin{equation*}
\rmT_{f, n}=\rmT_{f}/\varpi^{n}. 
\end{equation*}
Let $\bar{\rho}_{f}$  be the residual representation of $\rho_{f}$. Let $\rmA_{f}=\rmV_{f}/\rmT_{f}$. We set $\rmA_{f, n}=\ker[\rmA_{f}\xrightarrow{\varpi^{n}}\rmA_{f}]$. We denote by $\QQ(\bar{\rho}_{f})$ the field extension of $\QQ$ in ${\QQ}^{\ac}$ cut out by $\bar{\rho}_{f}$. To the newform $f$, we can associate maps $\phi_{f}: \TT\rightarrow \calO$ and $\phi_{f, n}: \TT\rightarrow \calO_{n}$ corresponding to the Hecke eigensystem of $f$. Recall we denote by $I_{f, n}$ the kernel of the map $\phi_{f, n}$.  Let $p$ be an $n$-admissible prime for $f$, by Theorem \ref{level-raise-curve}, we have a morphism $\phi^{[p]}_{f, n}: \TT^{[p]}\rightarrow \calO_{n}$ with kernel $I^{[p]}_{f, n}$. We have $\frakm_{f}=I_{f, 1}$ the maximal ideal in $\TT$ containing $I_{f, n}$ and  $\frakm^{[p]}_{f}=I_{f, 1}$ the maximal ideal in $\TT^{[p]}$ containing $I^{[p]}_{f, n}$

Now we consider a triple $\triplef=(f_{1}, f_{2}, f_{3})$ of newforms in  $S^{\new}_{2}(\Gamma_{0}(N))^{3}$. We will denote the representation $\rmV_{f_{i}}$ simply by $\rmV_{i}$ for $i=1, 2, 3$. They are defined over the $l$-adic Hecke field $E_{\lambda_{i}}$ of $f_{i}$ where $\lambda_{i}$ is a place of $\QQ(f_{i})$ over $l$. Let $\calO_{i}$ be the ring of integers of $E_{\lambda_{i}}$ with a uniformizer $\varpi_{i}$ and we put $\calO_{i, n}=\calO_{i}/\varpi^{n}_{i}$. Similarly as before, we have the $\calO_{i}$-lattice $\rmT_{i}$ and the $\calO_{i, n}$-module $\rmT_{i,n}$. Let $I_{i, n}=I_{f_{i}, n}$ and $\frakm_{i}=\frakm_{f_{i}}$. We have the fields $\QQ(\bar{\rho}_{i}):=\QQ(\bar{\rho}_{f_{i}})$ defined by the residual Galois representations $\bar{\rho}_{f_{i}}$. Let 
\begin{equation*}
\rmV(\triplef)=\rmV_{1}\otimes \rmV_{2}\otimes \rmV_{3}
\end{equation*}
be the triple tensor product representation associated to the triple $\triplef$.  Simarly, we put
\begin{equation*}
\rmM(\triplef)=\rmT_{1}\otimes \rmT_{2}\otimes \rmT_{3} 
\end{equation*}
and
\begin{equation*}
\rmM_{n}(\triplef)=\rmT_{1, n}\otimes \rmT_{2, n}\otimes \rmT_{3, n} .
\end{equation*}
We have the following result about the Galois cohomology of the representation $\rmV(\triplef)(-1)$ which says the representation $\rmV(\triplef)(-1)$ is tamely pure in the sense \cite[Definition 3.3]{Liu-HZ}.
\begin{lemma}
For all places $v\neq l$ of $\QQ$, we have
\begin{equation*}
\rmH^{1}(\QQ_{v}, \rmV(\triplef)(-1))=0.
\end{equation*}
\end{lemma}
\begin{proof}
The proof of this lemma is the same as that of \cite[Lemma 4.6]{Liu-cubic} by replacing the the elliptic curve by the Jacobians of suitable Shimura curves. 
\end{proof}

\begin{definition}\label{BK-grp}
The \emph{Bloch-Kato Selmer group} $\rmH^{1}_{f}(\QQ, \rmV(\triplef)(-1))$ is the subspace of classes $s\in \rmH^{1}(\QQ, \rmV(\triplef)(-1))$ such that 
\begin{equation*}
{\rm{loc}}_{l}(s)\in \rmH^{1}_{f}(\QQ_{l}, \rmV(\triplef)(-1)):=\ker[ \rmH^{1}_{f}(\QQ_{l}, \rmV(\triplef)(-1))\rightarrow \rmH^{1}_{f}(\QQ_{l}, \rmV(\triplef)\otimes\rmB_{\mathrm{cris}}(-1))].
\end{equation*}
\end{definition}

Let $(\pi_{1}, \pi_{2}, \pi_{3})$ be the triple of automorphic representation of $\GL_{2}(\mathbf{A})$ associated to the triple $(f_{1}, f_{2}, f_{3})$. Then one can attach the \emph{triple product L-function} 
\begin{equation*}
L(f_{1}\otimes f_{2}\otimes f_{3}, s)=L(s-\frac{3}{2}, \pi_{1}\otimes\pi_{2}\otimes \pi_{3}, r)
\end{equation*}
where $L(s-\frac{3}{2}, \pi_{1}\otimes\pi_{2}\otimes \pi_{3}, r)$ is the Langlands $L$-function attached to $r$ which is the natural eight-dimensional representation of the $L$-group of $\GL_{2}\times\GL_{2}\times\GL_{2}$. We will be concerned with case when \emph{global root number} 
\begin{equation*}
\epsilon(\pi_{1}\otimes\pi_{2}\otimes\pi_{3}, r)=1
\end{equation*}
that is the order of vanishing of the triple product $L$-function $L(f_{1}\otimes f_{2}\otimes f_{3}, s)$ at the central critical point $s=2$ is even. The following proposition relates $L(f_{1}\otimes f_{2}\otimes f_{3}, 2)$ to certain explicit period integral appeared in the statement of Theorem \ref{recip}.
\begin{proposition}\label{period}
Let $(f^{B}_{1}, f^{B}_{2}, f^{B}_{3})\in S^{B}_{2}(N^{+}, \calO)$ be the Jacquet-Langlands transfer of the triple $(f_{1}, f_{2}, f_{3})$. If the value $L(f_{1}\otimes f_{2}\otimes f_{3}, 2)$ is non-zero then 
\begin{equation}
I(f^{B}_{1}, f^{B}_{2}, f^{B}_{3})=\sum_{z\in X^{B}} f^{B}_{1}(z)\otimes f^{B}_{2}(z)\otimes f^{B}_{3}(z)
\end{equation} 
is non-zero.  
\end{proposition}
\begin{proof}
This follows from the main result of \cite{KH91} which resolves a conjecture of Jacquet. See also \cite{GK92} and \cite{Ichino} for refined formulas relating
the central critical values of $L(f_{1}\otimes f_{2}\otimes f_{3}, 2)$ to the period integral $I(f^{B}_{1}, f^{B}_{2}, f^{B}_{3})$. 
\end{proof}

We make the following conjecture for the motive attached to the triple $\triplef=(f_{1}, f_{2}, f_{3})$ that we hope to address in a future work.

\begin{conjecture}\label{main-conj}
Suppose that the triple $\triplef=(f_{1}, f_{2}, f_{3})$ satisfies the following assumptions: for each $i=1,2, 3$
\begin{enumerate}
\item the maximal ideals $\frakm_{i}$ are all residually irreducible;
\item the $\TT_{\frakm_{i}}$-module $S^{B}_{2}(N^{+}, \calO_{i})_{\frakm_{i}}$ is free of rank $1$;
\item the residual Galois representations $\bar{\rho}_{i}$ are surjective and the fields $\QQ(\bar{\rho}_{i})$ are linearly disjoint. 
\end{enumerate}
If the central critical value $L(f_{1}\otimes f_{2}\otimes f_{3}, 2)$ is non-zero, then the Bloch-Kato Selmer group vanishes
\begin{equation*}
\rmH^{1}_{f}(\QQ, \rmV(\triplef)(-1))=0. 
\end{equation*}
\end{conjecture}

Note that the assumptions made in the conjecture guarantees us that there are abundance of $n$-admissible primes for $\triplef$. 
\begin{lemma}\label{infty-adm}
Under the assumption of Theorem \ref{main-conj}, there are infinitely many $n$-admissible primes $p$ for $\triplef$.
\end{lemma}
\begin{proof}
Let $\rho_{i, n}: G_{\QQ}\rightarrow \GL_{2}(\calO_{i, n})$ be the representation on $\rmT_{i, n}$ defined by reducing $\rmT_{i}$ modulo $\varpi_{i}^{n}$. Consider the direct sum representation 
\begin{equation*}
\rho_{1, n}\oplus \rho_{2, n}\oplus \rho_{3, n}: G_{\QQ}\rightarrow \GL_{2}(\calO_{1, n})\times \GL_{2}(\calO_{2, n})\times \GL_{2}(\calO_{3, n}).
\end{equation*}
Since $\bar{\rho}_{i}$ is surjective and the fields $\QQ(\bar{\rho}_{i})$ are linearly disjoint, we can find infinitely many primes $p$ such that 
\begin{equation*}
\rho_{1, n}\oplus \rho_{2, n}\oplus \rho_{3, n}(\Frob_{p})=\begin{pmatrix}\epsilon_{1}p&0\\0&\epsilon_{1}\\ \end{pmatrix}\times \begin{pmatrix}\epsilon_{2}p&0\\0&\epsilon_{2}\\ \end{pmatrix}\times \begin{pmatrix}\epsilon_{3}p&0\\0&\epsilon_{3}\\ \end{pmatrix}\in \GL_{2}(\calO_{1, n})\times \GL_{2}(\calO_{2, n})\times \GL_{2}(\calO_{3, n})
\end{equation*}
with $l\nmid p^{2}-1$ and $\epsilon_{1}\epsilon_{2}\epsilon_{3}=1$ by the Chebotarev density theorem. By definition one can check that $p$ is an $n$-admissible prime for $(f_{1}, f_{2}, f_{3})$. 
\end{proof}

For $i=1, 2, 3$, recall that  $\rmT^{[p]}_{i}=\rmH^{1}(X_{\QQ^{\ac}}, \calO_{i}(1))_{\mathfrak{m}^{[p]}_{i}}$
and $\rmT^{[p]}_{i, n}= \rmT^{[p]}_{i}/I^{[p]}_{i, n}$.  We have the following lemma which we have already implicitly used in Lemma \ref{3-fin}. 

\begin{lemma}\label{p-lower}
Let $p$ be an $n$-admissible prime for $\triplef$. There is an isomorphism of $G_{\QQ}$-representations
\begin{equation*}
\rmT^{[p]}_{i, n}\cong \rmT_{i,n}.
\end{equation*}
\end{lemma}
\begin{proof}
This follows from the same proof given for \cite[Theorem 5.17]{BD-Main}. 
\end{proof}
\begin{remark}
This Lemma implies that we have an isomorphism $\rmM^{[p]}_{n}(\triplef)\cong \rmM_{n}(\triplef)$ of $G_{\QQ}$-modules. Under the assumption of Conjecture \ref{main-conj}, in particular $L(f_{1}\otimes f_{2} \otimes f_{3}, s)$ is non-vanishing at $s=2$, the period integral $I(f^{B}_{1}, f^{B}_{2}, f^{B}_{3})$ is non-vanishing by Proposition \ref{period} . Let $p$ be an $n$-admissible prime for $\triplef$, the cohomology class $\Theta^{[p]}_{n}$ has the property that $\partial_{p}\Theta^{[p]}_{n}$ is non-zero in $\rmH^{1}_{\sing}(\QQ_{p}, \rmM^{[p]}_{n}(\triplef)(-1))\cong\rmH^{1}_{\sing}(\QQ_{p}, \rmM_{n}(\triplef)(-1))$ for some $n$ by the first reciprocity law in Theorem \ref{recip}. Therefore the class $\Theta^{[p]}_{n}$ could be viewed as an annihilator of the Selmer group.  However as we have seen in Lemma \ref{3-fin}, the singular quotient at $p$ is of rank $3$ and therefore the class $\Theta^{[p]}_{n}$ alone can not fill up the whole singular quotient and therefore is not enough to bound the Selmer group. However we conjecture here that there exist three global cohomology classes in $\rmH^{1}(\QQ, \rmM_{n}(\triplef)(-1))$ which are intimately related to $\Theta^{[p]}_{n}$ and satisfy a similar reciprocity law as in Theorem \ref{recip}. Using these three classes, one can indeed prove Conjecture \ref{main-conj}.
\end{remark}
\subsection{The symmetric cube motive} We specialize the discussions in this article to the case when $\triplef=(f_{1}, f_{2}, f_{3})=(f, f, f)$ for a single modular form $f$. In this case we have a factorization
\begin{equation*}
\rmV_{f}^{\otimes 3}(-1)= \mathrm{Sym}^{3} \rmV_{f}(-1)\oplus \rmV_{f}\oplus \rmV_{f}
\end{equation*}
where we will refer to $\mathrm{Sym}^{3}(\rmV_{f})(-1)$ as the \emph{symmetric cube component} of $\rmV_{f}^{\otimes 3}(-1)$. Corresponding to this factorization, we have a factorization of the $L$-function
\begin{equation*}
L(f\otimes f \otimes f, s)= L(\mathrm{Sym}^{3}f, s) L(f, s-1)^{2}.
\end{equation*}
We can define the Bloch-Kato Selmer group 
\begin{equation*}
\rmH^{1}_{f}(\QQ, \mathrm{Sym}^{3} \rmV_{f}(-1))
\end{equation*}
for $\mathrm{Sym}^{3} \rmV_{f}(-1)$ exactly the same way as we did for the triple product representation in Definition \ref{BK-grp}. We will prove the following result towards the rank $0$ case of the Bloch-Kato conjecture for $\mathrm{Sym}^{3} \rmV_{f}(-1)$. 

\begin{theorem}\label{main-symm}
Suppose that the modular form $f$ satisfies the following assumptions: 
\begin{enumerate}
\item the maximal ideals $\frakm_{f}$ are all residually irreducible;
\item the $\TT_{\frakm_{f}}$-module $S^{B}_{2}(N^{+}, \calO)_{\frakm_{f}}$ is free of rank $1$;
\item  $\bar{\rho}_{f}$ is surjective;
\item  the value $L(f, 1)$ is non-vanishing.
\end{enumerate}
If the central critical value $L(\mathrm{Sym}^{3}f, 2)$ is non-zero,  then the Bloch-Kato Selmer group 
\begin{equation*}
\rmH^{1}_{f}(\QQ, \mathrm{Sym}^{3} \rmV_{f}(-1))=0. 
\end{equation*}
\end{theorem}

\begin{remark}
Let $f^{B}$ be the Jacquet-Langlands transfer of $f$ to $S^{B}_{2}(N^{+}, \calO)$. Then $(4)$ in the assumptions of the above theorem implies that the period integral $I(f^{B}, f^{B}, f^{B})$ is non-vanishing if $L(\mathrm{Sym}^{3}f, 2)$ is non-vanishing. 
\end{remark}

\subsection{Proof of Theorem \ref{main-symm}}
We will prove the theorem in this subsection. We use the following set of notations 
\begin{itemize}
 \item $\mathrm{N}^{\diamond}(\triplef)(-1)=\mathrm{Sym}^{3}\rmA_{f}(-1)$;
 \item $\mathrm{N}^{\diamond}_{n}(\triplef)(-1)=\mathrm{Sym}^{3}\rmA_{f, n}(-1)$;
 \item $\rmM^{\diamond}(\triplef)(-1)=\mathrm{Sym}^{3}\rmT_{f}(-1)$;
 \item $\rmM^{\diamond}_{n}(\triplef)(-1)=\mathrm{Sym}^{3}\rmT_{f, n}(-1)$.
\end{itemize}
We need to slightly modify the notion for $n$-admissible primes for $f$ in this case to incorporate the sign change phenomenon in the triple product setting. 
\begin{definition}\label{n-adm}
Let $n\geq 1$ be an integer, a prime $p$ is \emph{$(n, 1)$-admissible} for $f$ if 
\begin{enumerate}
\item $p\nmid Nl$
\item $l\nmid p^{2}-1 $
\item $\varpi^{n}\mid p+1-\epsilon_{p}(f)a_{p}(f)$ with $\epsilon_{p}(f)=1$. 
\end{enumerate}
\end{definition}
It is easy to see that under the assumptions in Theorem \ref{main-symm}, there are infinitely many $(n, 1)$-admissible prime for $f$ following the proof of Lemma \ref{infty-adm}. The $G_{\QQ}$-equivariant pairing 
\begin{equation*}
\mathrm{N}^{\diamond}_{n}(\triplef)(-1)\times \rmM^{\diamond}_{n}(\triplef)(-1)\rightarrow \calO_{n}(1)
\end{equation*}
induces for each place $v$ of $\QQ$ a local Tate duality
\begin{equation*}
(\hphantom{a}, \hphantom{b})_{v}: \rmH^{1}(\QQ_{v}, \rmN^{\diamond}_{n}(\triplef)(-1))\times \rmH^{1}(\QQ_{v}, \rmM^{\diamond}_{n}(\triplef)(-1)) \rightarrow \rmH^{1}(\QQ_{v}, \calO_{n}(1))\cong \calO_{n}.
\end{equation*}
For $s\in  \rmH^{1}(\QQ, \rmN^{\diamond}_{n}(\triplef)(-1))$ and $t\in  \rmH^{1}(\QQ, \rmM^{\diamond}_{n}(\triplef)(-1))$, we will write the pairing $(s, t)_{v}$ instead of $(\loc_{v}(s), \loc_{v}(t))_{v}$.  Let 
\begin{equation*}
\rmH^{1}_{f}(\QQ_{l}, \rmM^{\diamond}(-1))
\end{equation*}
be the pullback of  $\rmH^{1}_{f}(\QQ_{l}, \mathrm{Sym}^{3}\rmV_{f}(-1))$ under the natural map 
$$\rmH^{1}(\QQ_{l}, \rmM^{\diamond}(\triplef)(-1))\rightarrow \rmH^{1}(\QQ_{l}, \mathrm{Sym}^{3}\rmV_{f}(-1)).$$
We define 
\begin{equation*}
\rmH^{1}_{f}(\QQ_{l}, \rmM^{\diamond}_{n}(\triplef)(-1))
\end{equation*}
to be the reduction of $\rmH^{1}(\QQ_{l}, \rmM^{\diamond}(\triplef)(-1))$ modulo $\varpi^{n}$. Similarly, we let 
\begin{equation*}
\rmH^{1}_{f}(\QQ_{l}, \rmN^{\diamond}(\triplef)(-1))
\end{equation*}
be the image of $\rmH^{1}_{f}(\QQ_{l}, \mathrm{Sym}^{3}\rmV_{f}(-1))$ in $\rmH^{1}(\QQ_{l}, \rmN^{\diamond}(\triplef)(-1))$. Then we define 
\begin{equation*}
\rmH^{1}_{f}(\QQ_{l}, \rmN^{\diamond}_{n}(\triplef)(-1))
\end{equation*}
to be the pullback of 
$\rmH^{1}_{f}(\QQ_{l}, \rmN^{\diamond}(\triplef)(-1))$
under the natural map 
\begin{equation*}
\rmH^{1}(\QQ_{l}, \rmN^{\diamond}_{n}(\triplef)(-1))\rightarrow \rmH^{1}(\QQ_{l}, \rmN^{\diamond}(\triplef)(-1)). 
\end{equation*}

\begin{lemma}\label{sel-pairing}
We have the following statements.
\begin{enumerate}
\item The sum $\sum_{v}(\hphantom{a},\hphantom{b})_{v}$ restricted to $\rmH^{1}(\QQ, \rmN^{\diamond}_{n}(\triplef)(-1))\times \rmH^{1}(\QQ, \rmM^{\diamond}_{n}(\triplef)(-1))$ is trivial. Here $v$ runs through all the places in $\QQ$. 
\item  For every $v\neq l$, there exists an integer $n_{v}\geq 1$, independent of $n$ such that the image of the pairing 
\begin{equation*}
(\hfill, \hfill)_{v}: \rmH^{1}(\QQ_{v}, \rmN^{\diamond}_{n}(\triplef)(-1))\times \rmH^{1}(\QQ_{v}, \rmM^{\diamond}_{n}(\triplef)(-1)) \rightarrow \rmH^{1}(\QQ_{v}, \calO_{n}(1))\cong \calO_{n}.
\end{equation*} 
is annihilated by $\varpi^{n_{v}}$.
\item For every $v\neq l$, $\rmH^{1}_{\mathrm{fin}}(\QQ_{v}, \rmN^{\diamond}_{n}(\triplef)(-1))$ is orthogonal to $\rmH^{1}_{\mathrm{fin}}(\QQ_{v}, \rmM^{\diamond}_{n}(\triplef)(-1))$ under the pairing $(\hphantom{a}, \hphantom{b})_{v}$. Similarly,  $\rmH^{1}_{f}(\QQ_{l}, \rmN^{\diamond}_{n}(\triplef)(-1))$  is orthogonal to $\rmH^{1}_{f}(\QQ_{l}, \rmM^{\diamond}_{n}(\triplef)(-1))$.
\item Let $p$ be an $(n, 1)$-admissible prime for $f$, then we have a perfect pairing 
\begin{equation*}
\rmH^{1}_{\mathrm{fin}}(\QQ_{p}, \rmN^{\diamond}_{n}(\triplef)(-1))\times \rmH^{1}_{\mathrm{sing}}(\QQ_{p}, \rmM^{\diamond}_{n}(\triplef)(-1))\rightarrow \calO_{n}
\end{equation*}
of free $\calO_{n}$-modules of rank $1$. 
\end{enumerate}
\end{lemma}
\begin{proof}
The statement $(1)$ follows from global class field theory.  Part $(2)$ follows from the fact that $\rmH^{1}(\QQ_{v}, \rmV^{\diamond}(\triplef))=0$ for all $v\nmid l$ and thus $\rmH^{1}(\QQ_{v}, \rmM^{\diamond}(\triplef))$ is torsion for all $v\nmid l$, see  \cite[Lemma 4.3]{Liu-HZ}. Part $(3)$ is well-known, see \cite[Theorem 2.17(e)]{DDT} for the first statement and \cite[Lemma 4.8]{Liu-HZ} for the second statement. 

For $(4)$, it follows from the definition of an $(n,1)$-admissible prime for $f$ that $\rmM_{n}(\triplef)$ is unramified at $p$ and
$\rmM_{n}(\triplef)\cong\calO_{n}\oplus \calO^{\oplus 3}_{n}(1) \oplus \calO^{\oplus 3}_{n}(2) \oplus\calO_{n}(3)$
as a Galois representation of $G_{\QQ_{p}}$. Then it follows from a simple computation that $ \rmM^{\diamond}_{n}(\triplef)\cong\calO_{n}\oplus \calO_{n}(1) \oplus \calO_{n}(2) \oplus\calO_{n}(3)$. From this, it follows immediately that both $\rmH^{1}_{\mathrm{sing}}(\QQ_{p}, \rmM^{\diamond}_{n}(\triplef)(-1))$ and $\rmH^{1}_{\mathrm{fin}}(\QQ_{p}, \rmM^{\diamond}_{n}(\triplef)(-1))$ are of rank $1$ over $\calO_{n}$. The last claim is also clear form this. 
\end{proof}

Let $p$ be an $(n, 1)$-admissible prime for $f$. Recall that the class $\theta_{*}[\mathfrak{X}\otimes\QQ]\in \mathrm{CH}^{2}(\mathfrak{X}^{3}\otimes \QQ)$ and the Abel-Jacobi map 
\begin{equation*}
\mathrm{AJ}_{\triplef, n}: \Ch^{2}(\mathfrak{X}^{3}\otimes \QQ^{\ac})\rightarrow \rmH^{1}(\QQ, \rmM^{[p]}_{n}(\triplef)(-1))
\end{equation*}
defined in \S 4.5. By Lemma \ref{p-lower}, we have an isomorphism  $\rmM^{[p]}_{n}(\triplef)(-1)\cong \rmM_{n}(\triplef)$ of $G_{\QQ}$-modules.  Therefore we can project $\rmH^{1}(\QQ, \rmM^{[p]}_{n}(\triplef)(-1))$ to the symmetric component $\rmH^{1}(\QQ, \rmM^{\diamond}_{n}(\triplef)(-1))$. Therefore we arrive at the following  Abel-Jacobi map for $\rmM^{\diamond}_{n}(\triplef)(-1)$:
\begin{equation*}
\mathrm{AJ}^{\diamond}_{\triplef, n}: \Ch^{2}(\mathfrak{X}^{3}\otimes \QQ^{\ac})\rightarrow \rmH^{1}(\QQ, \rmM^{\diamond}_{n}(\triplef)(-1))
\end{equation*}
by composing $\mathrm{AJ}_{\triplef, n}$ with the projection map. We will denote by $\Theta^{\diamond [p]}_{n}$ the image of  $\theta_{*}[\mathfrak{X}\otimes\QQ]\in \mathrm{CH}^{2}(\mathfrak{X}^{3}\otimes \QQ)$ under $\mathrm{AJ}^{\diamond}_{\triplef, n}$. We denote by $\partial_{p}\Theta^{\diamond[p]}_{n}$ the image of $\Theta^{\diamond [p]}_{n}$ in the singular quotient $\rmH^{1}_{\sing}(\QQ_{p}, \rmM^{\diamond}_{n}(\triplef)(-1))$. The following proposition summarizes the arithmetic level raising and the first explicit reciprocity law for the symmetric cube representation $\rmM^{\diamond}_{n}(\triplef)(-1)$.
\begin{proposition}\label{reci-symm}
Let $p$ be an  $(n, 1)$-admissible prime for $f$. We assume the assumptions in Theorem \ref{main-symm} are satisfied. Then we have the following.
\begin{enumerate}
\item There is an isomorphism
\begin{equation*}
\rmH^{1}_{\sing}(\QQ_{p}, \rmM^{\diamond}_{n}(\triplef)(-1))\cong \otimes^{3}_{i=1}S^{B}_{2}(N^{+}, \calO){/I_{f,n}}.
\end{equation*}
\item Let $\phi\otimes \phi\otimes \phi\in \otimes^{3}_{i=1}S^{B}_{2}(N^{+}, \calO)[I_{f,n}]$. Then we have the following reciprocity formula: \begin{equation*}
(\partial_{p}\Theta^{\diamond[p]}_{n}, \phi\otimes \phi\otimes \phi)=(p+1)^{3}\sum_{z\in X^{B}}\phi(z)\phi(z)\phi(z).
\end{equation*} 
\end{enumerate}
\end{proposition}
\begin{proof}
These follow from our main results Corollary \ref{main-coro} and Theorem \ref{recip} by projecting from $\rmM_{n}(\triplef)(-1)$ to $\rmM^{\diamond}_{n}(\triplef)(-1)$. See also the calculation in  the proof of  Lemma \ref{sel-pairing} $(4)$.
\end{proof}

\begin{myproof}{Theorem}{\ref{main-symm}} We prove the theorem by contradiction. Suppose that $\rmH^{1}_{f}(\QQ, \mathrm{Sym}^{3}\rmV_{f}(-1))$ has dimension $>0$. Then one can find a free $\calO_{n}$-module $S$ of rank $1$ that is contained in $\rmH^{1}_{f}(\QQ, \rmN^{\diamond}_{n}(\triplef)(-1))$ by \cite[Lemma 5.9]{Liu-HZ}. Let $s$ be a generator of $S$. By the same argument as in \cite[Lemma 4.14, Lemma 4.16]{Liu-cubic}, we can choose an $(n, 1)$-admissible prime $p$ for $f$ with the property that $\loc_{p}(s)\neq 0\in \rmH^{1}_{\mathrm{fin}}(\QQ_{p}, \rmN^{\diamond}_{1}(\triplef)(-1))$. Moreover by \cite[Lemma 3.4, Remark 4.7]{Liu-HZ}, we have
\begin{itemize}
\item $\loc_{v}(s)\in \rmH^{1}_{\mathrm{fin}}(\QQ_{v}, \rmN^{\diamond}_{n}(\triplef)(-1))$ for $v\not\in\{l, N\}$;
\item $\loc_{l}(s)\in \rmH^{1}_{f}(\QQ_{l}, \rmN^{\diamond}_{n}(\triplef)(-1))$.
\end{itemize}

The class  $\Theta^{\diamond[p]}_{n} \in\rmH^{1}(\QQ, \rmM^{\diamond}_{n}(\triplef)(-1))$ satisfies the following properties:
\begin{itemize}
\item $\loc_{v}(\Theta^{\diamond[p]}_{n})\in \rmH^{1}_{\mathrm{fin}}(\QQ_{v}, \rmM^{\diamond}_{n}(\triplef)(-1))$ for all $v\not\in \{p, l, N\}$,
\item $\loc_{l}(\Theta^{\diamond[p]}_{n})\in \rmH^{1}_{f}(\QQ_{l}, \rmM^{\diamond}_{n}(\triplef)(-1))$.
\end{itemize}
These properties follow from the fact that the integral model $\mathfrak{X}^{3}$ has good reduction at a place $v\not\in \{p, l, N\}$ and at $l$ by \cite{Nekovar}. By Proposition \ref{period} and the assumptions in the theorem, there exists an integer $n_{I}< n$ and an element $\phi\otimes\phi\otimes\phi\in \otimes^{3}_{i=1}S^{B}_{2}(N^{+}, \calO)[I_{f,n}]$ such that 
\begin{equation*}
\varpi^{n_{I}}\nmid \sum_{z\in X^{B}}\phi(z)\phi(z)\phi(z).
\end{equation*}
It follows that $\varpi^{n_{I}}\nmid(\partial_{p}\Theta^{\diamond[p]}_{n}, \phi\otimes \phi\otimes \phi)$ by Proposition \ref{reci-symm} and therefore $\varpi^{n_{I}}\nmid \partial_{p}\Theta^{\diamond[p]}_{n}$.
One can also find an integer $n_{N}< n$  depending on the level $N$ and the class $s$ such that $(\varpi^{n_{N}}s, \kappa)_{v}=0$ for any place $v\mid N$ and any $\kappa\in  \rmH^{1}(\QQ_{v}, \rmM^{\diamond}_{n}(\triplef)(-1))$. This follows from Lemma \ref{sel-pairing} $(2)$. We choose $n$ such that $n> n_{I}+n_{N}$. By 
Lemma \ref{sel-pairing} $(1)$, we have 
\begin{equation*}
\varpi^{n}\mid \sum_{v}(\varpi^{n_{N}}s, \Theta^{\diamond[p]}_{n})_{v}=0.
\end{equation*}
By the properties of the classes of $s$ and $\Theta^{\diamond[p]}_{n}$ recalled above, the above equation implies that 
\begin{equation*}
\varpi^{n}\mid (\varpi^{n_{N}}s, \Theta^{\diamond[p]}_{n})_{p}=0.
\end{equation*}
It follows that $\varpi^{n_{I}}\mid (s, \Theta^{\diamond[p]}_{n})_{p}$ as $n_{I}< n-n_{N}$. This is a contradiction since
\begin{equation*}
\varpi^{n_{I}}\mid (s, \Theta^{\diamond[p]}_{n})_{p}= (\loc_{p}(s), \partial_{p}\Theta^{\diamond[p]}_{n})_{p}
\end{equation*}
which implies that $\varpi^{n_{I}}\mid \partial_{p}\Theta^{\diamond[p]}_{n}$.  
\end{myproof}

\end{document}